\documentclass[reqno, 12pt]{amsart}
\usepackage{amsmath}
\usepackage{amsfonts}
\usepackage{amssymb}
\usepackage{amsthm}
\usepackage[T1]{fontenc}
\usepackage[utf8]{inputenc}
\usepackage{authblk}
\usepackage{hyperref}

\usepackage{mathrsfs}
\usepackage{appendix}
\usepackage{graphicx}
\usepackage{indentfirst}
\usepackage{scalefnt}
\usepackage{footnote}
\usepackage{anyfontsize}
\usepackage{titling}
\usepackage{abstract}
\usepackage{setspace}
\newtheorem{theorem}{Theorem}[section]
\newtheorem{lemma}[theorem]{Lemma}
\newtheorem{corollary}[theorem]{Corollary}
\newtheorem{proposition}[theorem]{Proposition}
\newtheorem{definition}[theorem]{Definition}

\theoremstyle{definition}
\newtheorem{remark}[theorem]{Remark}
\allowdisplaybreaks[4]
\numberwithin{equation}{section}

\usepackage{geometry}
\newcommand{\de}{\mathrm{d}}

\begin{document}
	\title{Sharp gradient stability for a class of Hardy-Sobolev-Maz'ya inequalities}
	
	\author{Wei Dai, Jingze Fu and An Zhang}
	
	\address{School of Mathematical Sciences, Beihang University (BUAA), Beijing 100191, P. R. China, and Key Laboratory of Mathematics, Informatics and Behavioral Semantics, Ministry of Education, Beijing 100191, P. R. China}
	\email{weidai@buaa.edu.cn}
	
	\address{School of Mathematical Sciences, Beihang University (BUAA), Beijing 100191, P. R. China}
	\email{fujingze@buaa.edu.cn}
	
	\address{School of Mathematical Sciences, Beihang University (BUAA), Beijing 100191, P. R. China}
	\email{anzhang@buaa.edu.cn}	
	
	\maketitle
	\begin{abstract}
		In this paper, we proved the sharp gradient stability for a class of Hardy-Sobolev-Maz'ya inequalities with partial (stronger) singular weight and non-radial extremal functions. Our result seems to be the first stability result for non-radial extremal functions. The presence of partial (stronger) singular weight brings substantial new challenges, requiring us to significantly refine the techniques from \cite{DT1,FN,FZ} and introduce some new ideas to handle both the cylindrical symmetry of non-radial extremal functions and the partial (stronger) singular weight structure. Key technical innovations include new compact embedding with strong singularity, non-degeneracy and spectral property of the linearized operator $\mathcal{L}_{v}$ generated by non-radial extremal function $v$ and new refined spectral inequalities, which are crucial for our analysis. Since the extremal function $v$ is non-radial, ODE approach fails, we use binary PDE to prove the spectral property of $\mathcal{L}_{v}$. Surprisingly, the sharp exponent $\gamma=\max\{2,p\}$ in our sharp gradient stability inequality \eqref{eq:se} is independent of the partial weight dimension $k$, while the extremal manifold depends on $k$.
	\end{abstract}
	
	
	\smallskip
	
	
	\begin{small}	
		
		\section{Introduction}
In this paper, we investigate the sharp gradient stability for the following Hardy-Sobolev-Maz'ya inequalities \cite{BT}: there exists a sharp constant $S=S(n,p,k)>0$ such that for any function $u\in D^{1,p}(\mathbb{R}^n)$,
		\begin{equation}\label{eq:hsm}
			S\biggl(\int_{\mathbb{R}^n}|y|^{-1}|u|^{p_1^*}\,\de x\biggr)^{\frac{1}{p_1^*}}\leq \biggl(\int_{\mathbb{R}^n}|Du|^p \,\de x\biggr)^{\frac{1}{p}},
		\end{equation}
		where $x=(y,z)\in\mathbb{R}^k\times\mathbb{R}^{n-k}$ with $3\leq k\leq n$. The extremal functions \eqref{eq:em} for the sharp inequality \eqref{eq:hsm} satisfy the following Euler-Lagrange equation
		\begin{equation}\label{eq:ele}
			-\Delta_pv=S^p\|v\|_{L^{p_1^*}(\mathbb{R}^n;|y|^{-1})}^{p-p_1^*}|y|^{-1}v^{p_1^*-1},
		\end{equation}
		where the $p$-Laplacian
		$-\Delta_pu=-\mathrm{div}(|Du|^{p-2}Du)$
		and $S$ is the sharp constant in \eqref{eq:hsm}.		
		
		\subsection{Background and stability for the Sobolev inequality}\label{sub1-sob}
		
		The question of quantitative stability was first raised by Brezis and Lieb in \cite{BL} for the classical Sobolev inequality.
		Given $n\geq2$ and $1<p<n$, denote by  $D^{1,p}(\mathbb{R}^n)$ the closure of $C_c^{\infty}(\mathbb{R}^n)$ with respect to the norm
		\begin{equation*}
			\|u\|_{D^{1,p}(\mathbb{R}^n)}:=\biggl(\int_{\mathbb{R}^n}|Du|^p\,\de x\biggr)^{\frac{1}{p}},
		\end{equation*}
		where $Du$ is the gradient of function $u$.
		The $p$-Sobolev inequality says that, there exists a sharp positive constant $\mathscr S=\mathscr S(n,p)>0$ depending on $n$ and $p$ such that
		\begin{equation}\label{eq:sob}
			\|Du\|_{L^p(\mathbb{R}^n)}\geq \mathscr S \|u\|_{L^{p^*}(\mathbb{R}^n)},
		\end{equation}
		where $p^*=\frac{np}{n-p}$ is the Sobolev critical exponent.	
		Define the $p$-Sobolev deficit by
		\begin{equation*}
			d(u):=\frac{\|Du\|_{L^p(\mathbb{R}^n)}}{\|u\|_{L^{p^*}(\mathbb{R}^n)}}-\mathscr S,\quad\quad\forall\,u\in D^{1,p}(\mathbb{R}^n).
		\end{equation*}
		Then it is natural to ask the following question as \cite{BL}: whether the deficit $d(u)$ can be bounded from below by the Sobolev distance between $u$ and the extremal manifold $\mathscr{M}$ consisting of Aubin-Talenti-type extremal functions.
		A few years after \cite{BL}, the original problem for $p=2$ was solved  by Bianchi and Egnell \cite{BE}, proving that there  exists a constant $c=c(n)>0$ such that
		\begin{equation*}
			d(u)\geq c\inf_{v\in\mathscr{M}}\biggl(\frac{\|D(u-v)\|_{L^2(\mathbb{R}^n)}}{\|Du\|_{L^2(\mathbb{R}^n)}}\biggr)^2,\quad\quad\forall\,u\in D^{1,2}(\mathbb{R}^n),
		\end{equation*}
		which is sharp both in terms of the distance and  the exponent.
		However, the Bianchi-Egnell method in \cite{BE} relies heavily on the Hilbert structure of $D^{1,2}(\mathbb{R}^n)$, which breaks down unfortunately for general case $p\neq 2$.
		The general case therefore remains open for decades, requiring new ideas and techniques.
		
		About 20 years later, Cianchi, Fusco, Maggi and Pratelli \cite{CFMP} proved the first result for all $p\in(1,n)$: there exists some constant $c=c(n,p)>0$ such that
		\begin{equation}\label{eq:cfmp}
			d(u)\geq c\inf_{v\in\mathscr{M}}\biggl(\frac{\|u-v\|_{L^{p^*}(\mathbb{R}^n)}}{\|u\|_{L^{p^*}(\mathbb{R}^n)}}\biggr)^\gamma,\quad\quad\forall\,u\in D^{1,p}(\mathbb{R}^n),
		\end{equation}
		where the distance and  the explicit exponent $\gamma$ are both non-sharp.
		The authors introduced a beautiful combination of techniques coming from symmetrization and optimal transport, which were further developed by Figalli, Maggi and Pratelli in \cite{FaMP} to obtain sharp stability theorems for the anisotropic Sobolev and log-Sobolev inequalities on functions of bounded variation. 		See also \cite{Ca,Fn,FnMP} for earlier results.
		
		The first stability result using Sobolev distance for the Sobolev inequality \eqref{eq:sob} with $p>2$ is due to Figalli and Neumayer \cite{FN}: for some explicit exponent $\gamma=\gamma(n,p)>0$, there exists some constant $c=c(n,p)>0$ such that
		\begin{equation}\label{eq:fz}
			d(u)\geq c\inf_{v\in\mathscr{M}}\biggl(\frac{\|Du-Dv\|_{L^p(\mathbb{R}^n)}}{\|Du\|_{L^p(\mathbb{R}^n)}}\biggr)^\gamma,\quad\quad\forall\,u\in D^{1,p}(\mathbb{R}^n).
		\end{equation}
		Neumayer \cite{Nr} subsequently generalized this result to the full range $p\in(1,n)$ through a simpler proof.
		However, the exponents $\gamma$ in \cite{FN,Nr} are both non-sharp since their proofs rely heavily on the result in \cite{CFMP}.
		Note that, the strategy in \cite{Nr} cannot give the sharp exponent, even if one could prove the $L^p$ stability  \eqref{eq:cfmp} with a sharp exponent.
		
		In a recent breakthrough work,  Figalli and Zhang \cite{FZ} solved this problem completely for any $1<p<n$.
		Surprisingly, the sharp exponent  is given by
		$\gamma=\max\{2,p\}, $
		which is dimension-independent but exhibits a $p$-dependent behavior: it varies with respect to $p$ for $p>2$ while remains the constant $2$ for $p\le 2$.
		
		As mentioned at the beginning, there are also some important further progress on the sharp stability constant $c=c(n,p)$.
		In particular, Dolbeault, Esteban, Figalli, Frank and Loss \cite{DEFFL} recently proved a remarkable result with sharp stability constants for the classical Sobolev and log-Sobolev inequalities ($p=2$).
		The result is much more complete than \cite{BE}  and \cite{BDNS,BDS1,BDS2,DjT,Kt} in the sense that, the explicit stability constant for the Sobolev inequality is dimensionally sharp, which further yields a dimension-free stability constant for the Gaussian log-Sobolev inequality.
		It also does not require any restriction on the functions by entropy method like \cite{BDNS}.

For more literature on the stability of Sobolev-type inequalities, refer to \cite{BWW,BE,BDNS,BDS1,BDS2,Ce,CFL,CKW,CLT1,CFW,Ca,CFMP,DHP,DSW,DE,DZ_of,DZ_ff,FaMP,Frank,Frank1,FnMP,FN,FZ,GW,LZ_rt,LZZ,Nr,Nv,PYZ,Rb,WW22,WW23,WW24,YZ,ZZ,ZZZ}. For the stability of isoperimetric and  Brunn-Minkowski inequalities, c.f. \cite{CL,FHT1,FHT2,FJ1,FJ2,FMM,FaMP_i,FaMP_b,FnMP_i,FZ_i,HST1,HST2}. We refer to \cite{BDNS,BDS1,BDS2,CLT2,DFHQW,DHL,DLQ,DEFFL,DjT,Kt} for some further progress on the sharp (stability) constant and \cite{CDL,CDQ0,CF,CS,DDGL,DGHP,DLL,DQ,FM1,FM2} for other interesting results.

		\subsection{Our main result on stability for the Hardy-Sobolev-Maz'ya inequalities}
		It is natural to consider the stability problem for weighted Sobolev inequalities.
		So far, there have been many works investigating the stability for Sobolev inequalities with radial weights, for instance, the Hardy-Sobolev and Caffarelli-Kohn-Nirenberg inequalities \cite{DT1,DT2,WW22,WW24,ZZ}.
		In this paper, based on the classification result of extremal functions in \cite{LM}, we study the quantitative stability of the Hardy-Sobolev-Maz'ya inequalities, in which case we consider the partial (stronger) singular weight $|y|^{-1}$, where $y\in \mathbb{R}^{k}$ with $k\leq n$, see \eqref{eq:hsm}.
		
		In contrast to both the unweighted case and the radially weighted scenario, the presence of partial weights with stronger singularities generates substantial technical obstacles.
		For instance, when $k<n$, the partial singular weight $|y|^{-1}$ with $y\in \mathbb{R}^{k}$ in the Hardy-Sobolev-Maz'ya inequalities \eqref{eq:hsm} generates stronger singularities and the extremal functions of the Hardy-Sobolev-Maz'ya inequalities \eqref{eq:hsm} lose full radial symmetry and instead exhibit only cylindrical symmetry. Cylindrical symmetry of the extremal functions has been proved by Secchi, Smets and Willem in \cite{SSW} by symmetrization arguments, see also \cite{MS04} for $p=2$. Very recently, for $3\leq k\leq n-1$, Lin and Ma \cite{LM} classified all positive finite energy cylindrically symmetric solutions of the corresponding Euler-Lagrange equation \eqref{eq:ele}, and hence obtain the best constant and extremal functions for the Hardy-Sobolev-Maz'ya inequalities \eqref{eq:hsm}.
		
		Given $n \geq 4$ and $1 < p < n$, let  $p_1^*:=\frac{p(n-1)}{n-p}$ be the weighted exponent.
		Define the weighted spaces $D^{1,p}(\mathbb{R}^n;w)$ and $L^p(\mathbb{R}^n;w)$ as the closure of $C_c^{\infty}(\mathbb{R}^n)$ with respect to the norms
		\begin{equation*}
			\| u\|_{L^p(\mathbb{R}^n;w)}:=\biggl(\int_{\mathbb{R}^n}w|u|^p \,\de x \biggr)^{1/p}\quad\mathrm{and}\quad\| u\|_{D^{1,p}(\mathbb{R}^n;w)}:=\biggl(\int_{\mathbb{R}^n}w|Du|^p \,\de x \biggr)^{1/p},
		\end{equation*}
		where $w\geq0$ is a weight function.
		We consider the following class of Hardy-Sobolev-Maz'ya inequalities \cite{BT}: there exists a sharp constant $S=S(n,p,k)>0$ such that for any function $u\in D^{1,p}(\mathbb{R}^n)$,
		\begin{equation}\label{eq:hsm}
			S\biggl(\int_{\mathbb{R}^n}|y|^{-1}|u|^{p_1^*}\,\de x\biggr)^{\frac{1}{p_1^*}}\leq \biggl(\int_{\mathbb{R}^n}|Du|^p \,\de x\biggr)^{\frac{1}{p}},
		\end{equation}
		where $x=(y,z)\in\mathbb{R}^k\times\mathbb{R}^{n-k}$ with $3\leq k\leq n$.
		Let $\mathcal{M}$ be the $(n-k+2)$-dimensional manifold of all extremal functions for  \eqref{eq:hsm} given by (see \cite{LM})
		\begin{equation}\label{eq:em}
			v_{a,\lambda,z'}(x)=a\lambda^{\frac{n-p}{p}}\bigl[(1+\lambda|y|)^2+|\lambda z-z'|^2\bigr]^{-\frac{n-p}{2(p-1)}},\quad\quad a\in\mathbb{R}\backslash\{0\},\:\lambda>0,\: z'\in\mathbb{R}^{n-k}.
		\end{equation}
		The sharp constant $S$ can be computed easily using the extremal functions. In particular, for any fixed nonzero function $v=v_{a,\lambda,z'}\in\mathcal{M}$,
		\begin{equation*}
			\|Dv\|_{L^p(\mathbb{R}^n)}=S\|v\|_{L^{p_1^*}(\mathbb{R}^n;|y|^{-1})},	\end{equation*}
		and
		the tangent space of $\mathcal M$ at point $v$  is
		\begin{equation*}
			T_v\mathcal{M}=\mathrm{span}\{v, \partial_\lambda v, \partial_{z'_1}v,\ldots,\partial_{z'_{n-k}}v\}.
		\end{equation*}
		The extremal functions \eqref{eq:em} for the sharp inequality \eqref{eq:hsm} satisfy the following Euler-Lagrange equation
		\begin{equation}\label{eq:ele}
			-\Delta_pv=S^p\|v\|_{L^{p_1^*}(\mathbb{R}^n;|y|^{-1})}^{p-p_1^*}|y|^{-1}v^{p_1^*-1},
		\end{equation}
		where the $p$-Laplacian
		$-\Delta_pu=-\mathrm{div}(|Du|^{p-2}Du)$
		and $S$ is the sharp constant in \eqref{eq:hsm}.

		Similar to the quantitative stability for the Sobolev inequality in subsection \ref{sub1-sob}, in order to study the quantitative stability for the Hardy-Sobolev-Maz'ya inequalities \eqref{eq:hsm}, we define the deficit for any function $u\in D^{1,p}(\mathbb{R}^n)$ by
		\begin{equation*}
			\delta(u):=\frac{\|Du\|_{L^p(\mathbb{R}^n)}}{\|u\|_{L^{p_1^*}(\mathbb{R}^n;|y|^{-1})}} - S.
		\end{equation*}
		Our goal is to establish a lower bound for the deficit $\delta(u)$ analogous to \eqref{eq:fz}, featuring a sharp exponent in terms of the Sobolev distance between $u$ and the extremal manifold $\mathcal M$.
		
		To this end, we consider the corresponding linearized problem.
		Differentiating \eqref{eq:ele} for function $v=v_{a,\lambda,z'}\in\mathcal{M}$ with respect to either $\lambda$ or $z'_i$ ($1\leq i\leq n-k$) yields the following: for any $w\in\mathrm{span}\{\partial_\lambda v, \partial_{z'_1}v,\ldots,\partial_{z'_{n-k}}v\}$,
		\begin{equation*}
			-\mathrm{div}\bigl(|Dv|^{p-2}Dw+(p-2)|Dv|^{p-4}(Dv\cdot Dw)Dv\bigr) =(p_1^*-1)S^p\|v\|_{L^{p_1^*}(\mathbb{R}^n;|y|^{-1})}^{p-p_1^*}|y|^{-1}v^{p_1^*-2}w.
		\end{equation*}
		This naturally leads us to investigate the linearized $p$-Laplace operator
		\begin{equation}\label{eq:lv}
			\mathcal{L}_v[\varphi]:=-\mathrm{div}\bigl(|Dv|^{p-2}D\varphi+(p-2)|Dv|^{p-4}(Dv\cdot D\varphi)Dv\bigr),
		\end{equation}
		defined on the weighted space $L^2(\mathbb{R}^n;|y|^{-1}v^{p_1^*-2})$.

\medskip
		
We show in Proposition \ref{prop:ds} that $\mathcal{L}_v$ has a discrete spectrum. Moreover, the eigenspaces corresponding to the first and second eigenvalues coincide with the tangent space $T_v\mathcal{M}$. We prove the following non-degeneracy result and spectral property for the linearized operator $\mathcal{L}_v$.
		\begin{theorem}\label{thm:ef}
			Let $\alpha_i$ and $E_i$ be the $i$-th eigenvalue and eigenspace of the operator $\mathcal{L}_v$ defined by \eqref{eq:lv} in the weighted space $L^2(\mathbb{R}^n;|y|^{-1}v^{p_1^*-2})$. Then we have
			\begin{align}
				\alpha_1=(p-1)S^p\|v\|_{L^{p_1^*}(\mathbb{R}^n;|y|^{-1})}^{p-p_1^*},\quad\quad &E_1=\mathrm{span}\{v\},\label{eq:e1}\\
				\alpha_2=(p_1^*-1)S^p\|v\|_{L^{p_1^*}(\mathbb{R}^n;|y|^{-1})}^{p-p_1^*},\quad\quad &E_2=\mathrm{span}\{\partial_\lambda v, \partial_{z'_1}v,\ldots,\partial_{z'_{n-k}}v\}.\label{eq:e2}
			\end{align}
		\end{theorem}
		
\begin{remark}
For non-degeneracy result on the linearized operator $\mathcal{L}_v$ in the special Laplacian case $p=2$, refer to \cite{CFMS}, see also \cite{CPY}. Our Theorem \ref{thm:ef} extend their results from $p=2$ to general $1<p<n$.
\end{remark}

		Based on Theorem \ref{thm:ef}, we can prove our main result, the following sharp stability theorem on Hardy-Sobolev-Maz'ya inequalities.
		\begin{theorem}\label{thm:main}
			Let $n \geq 4$, $1 < p < n$, $3\leq k\leq n-1$ and $\gamma=\max\{2,p\}$. Then there exists a constant $c=c(n,p,k)>0$ such that, for all $u\in D^{1,p}(\mathbb{R}^n)$,
			\begin{equation}\label{eq:se}
				\delta(u)\geq c(n,p,k)\inf_{v\in\mathcal{M}}\biggl(\frac{\|D(u-v)\|_{L^p(\mathbb{R}^n)}}{\|Du\|_{L^p(\mathbb{R}^n)}}\biggr)^\gamma.
			\end{equation}
			Moreover, the exponent $\gamma=\max\{2,p\}$ is sharp for \eqref{eq:se} to hold.
		\end{theorem}
		\begin{remark}
			Surprisingly, our main theorem (Theorem \ref{thm:main}) indicates that the sharp exponent $\gamma=\max\{2,p\}$ is independent of the weight dimension $k$, while the extremal manifold $\mathcal{M}$ depends on it. Notably, the case $k=2$ remains open, even the classification of the extremal functions of \eqref{eq:hsm} is currently unknown (see \cite{LM}).
		\end{remark}
		\begin{remark}
			Following the approach of \cite[Remark 1.2]{FZ}, we can verify the sharpness of the decay exponent $\gamma=\max\{2,p\}$.
			
			On the one hand, consider the test functions $u_i(x):=v(A_ix)$ ($i\in \mathbb N^+$),   where $v=v_{1,1,0}\in\mathcal{M}$ and  $A_i\in\mathbb{R}^{n\times n}$ is the diagonal matrix
			\begin{equation*}
				A_i=\mathrm{diag}\biggl(1,\ldots,1,1+\frac{1}{i}\biggr).
			\end{equation*}
			A direct calculation shows that:
			\begin{itemize}
				\item the deficit $\delta(u_i)$ decays as $i^{-2}$,
				\item the right-hand side of \eqref{eq:se} decays as $i^{-\gamma}$. \end{itemize}
			Consequently, \eqref{eq:se} fails for any exponent $\gamma<2$.
			
			On the other hand, for any fixed non-trivial $\varphi\in C_c^{\infty}(\mathbb{R}^n)$, consider $\tilde{u}_i:=v+\varphi(x_i+\cdot)$, where $\{x_i=(y_i,z_i)\}_{i\in \mathbb N^+}\subseteq\mathbb{R}^n$ satisfies $|x_i|\rightarrow\infty$ and $|y_i|\geq C$ for some constant $C$ as $i\rightarrow\infty$.
			One can check that
			\begin{equation*}
				\|D\tilde{u}_i\|_{L^p(\mathbb{R}^n)}^p=\|Dv\|_{L^p(\mathbb{R}^n)}^p + \|\varphi\|_{L^p(\mathbb{R}^n)}^p+r_{i,1}
			\end{equation*}
			and
			\begin{equation*}
				\|\tilde{u}_i\|_{L^{p_1^*}(\mathbb{R}^n;|y|^{-1})}^{p_1^*}=\|v\|_{L^{p_1^*}(\mathbb{R}^n;|y|^{-1})}^{p_1^*}+\|\varphi\|_{L^{p_1^*}(\mathbb{R}^n;|y|^{-1})}^{p_1^*}+ r_{i,2},
			\end{equation*}
			where the remainder terms are bounded by $|r_{i,1}|+|r_{i,2}|\leq C\big(v(x_i)+|Dv(x_i)|\big)\leq Cv(x_i)\rightarrow 0$ as $i\rightarrow\infty$.
			Hence, choosing a sequence $\varepsilon_i\rightarrow0$ such that $v(x_i)\ll\varepsilon_i\ll1$, the functions $\hat{u}_i:=v+\varepsilon_i\varphi(x_i+\cdot)$ satisfy
			\begin{equation*}
				\|D\hat{u}_i\|_{L^p(\mathbb{R}^n)}^p=\|Dv\|_{L^p(\mathbb{R}^n)}^p + \varepsilon_i^p\|\varphi\|_{L^p(\mathbb{R}^n)}^p+o(\varepsilon_i^p)
			\end{equation*}
			and
			\begin{equation*}
				\|\hat{u}_i\|_{L^{p_1^*}(\mathbb{R}^n;|y|^{-1})}^{p_1^*}=\|v\|_{L^{p_1^*}(\mathbb{R}^n;|y|^{-1})}^{p_1^*}+\varepsilon_i^{p_1^*}\|\varphi\|_{L^{p_1^*}(\mathbb{R}^n;|y|^{-1})}^{p_1^*} + o(\varepsilon_i^{p_1^*}).
			\end{equation*}
			Therefore, one can discover that the deficit $\delta(\hat{u}_i)$ behaves as $\varepsilon_i^p$, while the right-hand side of \eqref{eq:se} behaves as $\varepsilon_i^{\gamma}$. Therefore, \eqref{eq:se} fails also for any exponent  $\gamma<p$.
			
			These two necessary conditions  confirm the sharpness of our exponent $\gamma=\max\{2,p\}$.
		\end{remark}
		
		\subsection{Main ideas of the proof}
		Following the approach of \cite{FN,FZ}, we begin our proof along the basic lines of \cite{BE}.
		Specifically, for any function $u$ that is close to $\mathcal{M}$, we select $v\in\mathcal{M}$ closest to $u$ and set
		\[\varphi:=\frac{u-v}{\|Du-Dv\|_{L^p(\mathbb{R}^n)}} \qquad \text{and}\qquad \varepsilon:=\|Du-Dv\|_{L^p(\mathbb{R}^n)}.\]
		This allows us to express $u$  as $v+\varepsilon\varphi$.
		Expanding the deficit $\delta(u)$ in $\varepsilon$, we aim to obtain a lower bound of $\delta(u)$ in terms of $\|\varphi\|_{D^{1,p}(\mathbb{R}^n)}^\gamma$.
		
		The Taylor expansion of $\delta(u)$ yields that
		\begin{equation*}
			\delta(v+\varepsilon\varphi)=\varepsilon^2Q_v[\varphi]+o\bigl(\varepsilon^2\|D\varphi\|_{L^p(\mathbb{R}^n)}^2\bigr),
		\end{equation*}
		where $Q_v[\cdot]$ is  a quadratic form depending on $v$.
		If $\varphi$ is orthogonal to the tangent space $T_v\mathcal{M}$ in the weighted space $L^2(\mathbb{R}^n;|y|^{-1}v^{p_1^*-2})$, spectral analysis shows that $Q_v[\varphi]$ controls the weighted norm $\|D\varphi\|_{L^p(\mathbb{R}^n;|y|^{-1}v^{p_1^*-2})}^2$.
		Consequently, we obtain the lower bound
		\begin{equation*}
			\delta(v+\varepsilon\varphi)\geq c\varepsilon^2\|D\varphi\|_{L^p(\mathbb{R}^n;|y|^{-1}v^{p_1^*-2})}^2+o\bigl(\varepsilon^2\|D\varphi\|_{L^p(\mathbb{R}^n)}^2\bigr).
		\end{equation*}
		
		If $p=2$, the result follows directly from the smallness of  $\varepsilon\ll1$ under the orthogonality condition.
		When $p\not=2$, however, several challenging technical obstacles arise:
		\begin{itemize}
			\item $L^p(\mathbb{R}^n)$ is non-Hilbert for $p\not=2$, thus the analysis will become much more complicated.
			
			\item For $p>2$, the weighted $L^2$-norm cannot control the $L^p$-norm of $D\varphi$.
			
			\item For $p<2$, $D^{1,p}$-norm is weaker than any weighted $D^{1,2}$-norm.
			
			\item If $p\leq\frac{2n}{n+1}$, $L^{p_1^*}(\mathbb{R}^n;|y|^{-1})$-norm is weaker than any weighted $L^{2}$-norm.
		\end{itemize}
		Furthermore, comparing with the stability for the classical Sobolev inequality in \cite{FN,FZ} and the stability for the Hardy-Sobolev inequality in \cite{DT1}, we have at least two more crucial difficulties on the stability for the Hardy-Sobolev-Maz'ya inequalities:
		\begin{itemize}
			\item The partial singular weight $|y|^{-1}$  in the Hardy-Sobolev-Maz'ya inequalities (involving only $k$-dimensional part of the spatial variables) generates stronger singularities and complicates the key estimates.
			
			\item The extremal functions only possess cylindrical symmetry rather than radial symmetry, hence the non-degeneracy and spectral analysis of $\mathcal{L}_v$ can not be deduced by the spectral analysis of ODE and Sturm-Liouville theory via a spherical harmonic decomposition.
		\end{itemize}
		To overcome all those difficulties mentioned above, we adapted strategies from  \cite{DT1,FN,FZ} with a series of new ideas and techniques. Specifically, we have developed several innovative approaches:
		\begin{itemize}
			\item For the non-degeneracy and spectral analysis of the linearized operator $\mathcal{L}_v$, since the extremal function $v$ is merely cylindrically symmetric and is not radially symmetric, the usual ODE approach does not work. We apply the spherical harmonic decomposition to the linearized equation \eqref{eq:phi} with respect to partial variable $y\in \mathbb{R}^{k}$ and then apply the spherical harmonic decomposition twice again with respect to partial variable $z\in\mathbb{R}^{n-k}$, and hence reduce the non-degeneracy and spectral analysis of $\mathcal{L}_v$ to spectral analysis of a PDE involving two variables instead of ODE, see the proof of Theorem \ref{thm:ef} in subsection 3.1.
			
			\item For partial (stronger) singular weight estimates, we introduced new transformation techniques to effectively handle the non-radial singularities, and proved and applied new compact embedding with strong singularity, see Section 2.
			
			\item We significantly modified those methods from  \cite{DT1,FN,FZ} to deal with the cylindrical symmetry and partial (stronger) singular weight structure, and established new refined spectral inequalities, see sub-section 3.2.
		\end{itemize}

		\subsection{Structure of the paper}
		In Section \ref{sec:cr}, we prove some compactness results that form the foundation for subsequent analysis.
		In Section \ref{sec:sg}, we first prove the discreteness of the spectrum of the linearized operator in Proposition \ref{prop:ds}. Then, we characterize the eigenspaces corresponding to the first and second eigenvalues and hence prove Theorem \ref{thm:ef}, which indicates the non-degeneracy of the Euler-Lagrange equations. Finally, based on Theorem \ref{thm:ef}, we establish a crucial spectral estimate.
		In Section \ref{sec:main}, we complete the proof of Theorem \ref{thm:main} by combining the spectral gap theorem with the concentration-compactness arguments.
		In Appendices, section \ref{sec-cc} contains a proof of the concentration-compactness theorem and section \ref{sec-ue} collects some useful estimates for reader's convenience.
		
		\vspace{1em}
		
		\noindent\textbf{Notation.}
		Throughout this paper, we adopt the following conventions: positive constants are denoted by $C(\cdot)$ and $c(\cdot)$, where the parentheses indicate all dependent parameters.
		Typically, $C$ denotes a big constant $\ge 1$ and $c$ denotes a small constant $\le 1$, whose value may vary from line to line.
		The relation $a\sim b$ indicates the existence of constants $C_1, C_2>0$ such that $C_1b\leq a\le  C_2 b$. The Euclidean ball centered at $x$ with radius $r$ is denoted by $B(x, r)$.
		We use $\rightarrow$ to denote strong convergence and $\rightharpoonup$ to denote weak convergence.
		
		\section{Compactness Results}\label{sec:cr}
		
		\subsection{A rough estimate for any $1<p<n$.}
		The following weighted compact embedding theorem will be applied in Section \ref{sec:sg} to prove $\mathcal{L}_v$ has a discrete spectrum.
		Throughout Section \ref{sec:cr} and Section \ref{sec:sg} we assume that $v:=v_{a_0,1,0}$ with $a_0>0$ such that $\frac{1}{2}\leq\|v\|_{L^{p_1^*}(\mathbb{R}^n)}\leq2$.
		\begin{proposition}\label{prop:ce}
			Let $1<p<n$. Then $D^{1,2}(\mathbb{R}^n;|Dv|^{p-2})$ compactly embeds into $L^2(\mathbb{R}^n;|y|^{-1}v^{p_1^*-2})$.
		\end{proposition}
		To prove this, we first introduce an estimate that will be useful also later.
		Define the cylinders $Y_{R}:=\bigl\{x=(y,z)\in\mathbb{R}^k\times\mathbb{R}^{n-k}\: :\: |y|\leq R\bigr\}$ and $Z_{R}:=\bigl\{x=(y,z)\in\mathbb{R}^k\times\mathbb{R}^{n-k}\: :\: |z|\leq R\bigr\}$, then we have the following lemma.
		\begin{lemma}
			Let $1<p<n$. \\
			(1) There exists a constant $C =C(n,p,k)>0$ such that for all $\varphi\in L^2(\mathbb{R}^n;|y|^{-1}v^{p_1^*-2}) \cap D^{1,2}(\mathbb{R}^n;|Dv|^{p-2})$, we have
			\begin{equation}\label{eq:pti}
				\int_{\mathbb{R}^n}|y|^{-1}v^{p_1^*-2}\varphi^2\,\de x\leq C(n,p,k)\int_{\mathbb{R}^n}|Dv|^{p-2}|D\varphi|^2\,\de x.
			\end{equation}
			(2)	Moreover, there exists $\beta=\beta(n,p,k)>0$, such that for all $\rho\in(0,1)$,
			\begin{equation}\label{eq:ptib}
				\int_{Y_\rho}|y|^{-1}v^{p_1^*-2}\varphi^2\,\de x\leq C(n,p,k)\rho^{\beta}\int_{\mathbb{R}^n}|Dv|^{p-2}|D\varphi|^2\,\de x,
			\end{equation}
			and also
			\begin{equation}\label{eq:ptiob}
				\int_{\mathbb{R}^n\backslash B(0,\rho^{-1})}|y|^{-1}v^{p_1^*-2}\varphi^2\,\de x\leq C(n,p,k)\rho\int_{\mathbb{R}^n}|Dv|^{p-2}|D\varphi|^2\,\de x.
			\end{equation}
		\end{lemma}
		
		\begin{proof}
			\emph{(1)} \ First we assume $\varphi\in C_c^1(\mathbb{R}^n)$.
			Define
			\begin{equation*}
				G(u):=\int_{\mathbb{R}^n}|Du|^p\,\de x -S^p \biggl(\int_{\mathbb{R}^n}|y|^{-1}|u|^{p_1^*}\,\de x\biggr)^{p/p_1^*}.
			\end{equation*}
			Since $v$ is a local minimum of functional $G$, it holds
			\begin{equation}\label{eq:fc}
				\begin{split}
					0\leq{}&\frac{\mathrm{d}^2}{\mathrm{d}\varepsilon^2}\bigg{|}_{\varepsilon=0}G(v+\varepsilon\varphi)\\
					={}&p\int_{\mathbb{R}^n}|Dv|^{p-2}|D\varphi|^2\,\de x + p(p-2)\int_{\mathbb{R}^n}|Dv|^{p-4}(Dv\cdot D\varphi)^2\,\de x\\
					&-pp_1^*\biggl(\frac{p}{p_1^*}-1\biggr)S^p\biggl(\int_{\mathbb{R}^n}|y|^{-1}|v|^{p_1^*}\,\de x\biggr)^{\frac{p}{p_1^*}-2}\biggl(\int_{\mathbb{R}^n}|y|^{-1}|v|^{p_1^*-2}v\varphi \,\de x\biggr)^2\\
					&-p(p_1^*-1)S^p\biggl(\int_{\mathbb{R}^n}|y|^{-1}|v|^{p_1^*}\,\de x\biggr)^{\frac{p}{p_1^*}-1}\int_{\mathbb{R}^n}|y|^{-1}|v|^{p_1^*-2}\varphi^2 \,\de x.
				\end{split}
			\end{equation}
			By H$\rm{\ddot{o}}$lder's inequality, we have
			\begin{equation*}
				\biggl(\int_{\mathbb{R}^n}|y|^{-1}|v|^{p_1^*-2}v\varphi \,\de x\biggr)^2\leq\int_{\mathbb{R}^n}|y|^{-1}|v|^{p_1^*}\,\de x\cdot\int_{\mathbb{R}^n}|y|^{-1}|v|^{p_1^*-2}\varphi^2 \,\de x,
			\end{equation*}
			and also
			\begin{equation*}
				\int_{\mathbb{R}^n}|Dv|^{p-4}(Dv\cdot D\varphi)^2\,\de x\leq\int_{\mathbb{R}^n}|Dv|^{p-2}|D\varphi|^2\,\de x.
			\end{equation*}
			Thus it follows that, if $p\geq2$,
			\begin{equation*}
				\begin{split}
					0\leq{}\frac{\mathrm{d}^2}{\mathrm{d}\varepsilon^2}\bigg{|}_{\varepsilon=0}G(v+\varepsilon\varphi)
					\leq{}&p(p-1)\int_{\mathbb{R}^n}|Dv|^{p-2}|D\varphi|^2\,\de x\\ &-p(p-1)S^p\biggl(\int_{\mathbb{R}^n}|y|^{-1}|v|^{p_1^*}\,\de x\biggr)^{\frac{p}{p_1^*}-1}\int_{\mathbb{R}^n}|y|^{-1}|v|^{p_1^*-2}\varphi^2\,\de x,
				\end{split}
			\end{equation*}
			and if $p<2$,
			\begin{equation*}
				\begin{split}
					0\leq{}\frac{\mathrm{d}^2}{\mathrm{d}\varepsilon^2}\bigg{|}_{\varepsilon=0}G(v+\varepsilon\varphi)
					\leq{}&p\int_{\mathbb{R}^n}|Dv|^{p-2}|D\varphi|^2\,\de x\\ &-p(p-1)S^p\biggl(\int_{\mathbb{R}^n}|y|^{-1}|v|^{p_1^*}\,\de x\biggr)^{\frac{p}{p_1^*}-1}\int_{\mathbb{R}^n}|y|^{-1}|v|^{p_1^*-2}\varphi^2\,\de x.
				\end{split}
			\end{equation*}
			Noticing that $\int_{\mathbb{R}^n}|y|^{-1}|v|^{p_1^*}\,\de x$ is a constant depending only on $n$, $p$ and $k$, thus \eqref{eq:pti} holds for any $\varphi\in C_c^1(\mathbb{R}^n)$, and then for any $\varphi\in D^{1,2}(\mathbb{R}^n;|Dv|^{p-2})$ by approximation argument.
			
			\emph{(2)} \ To prove \eqref{eq:ptib}, we first observe that
			\begin{equation}\label{eq:vdv}
				|v(x)|\sim\bigl[(1+|y|)^2+|z|^2\bigr]^{\frac{p-n}{2(p-1)}}\quad\quad\mathrm{and}\quad\quad|Dv|\sim\bigl[(1+|y|)^2+|z|^2\bigr]^{\frac{1-n}{2(p-1)}},
			\end{equation}
			so it follows that
			\begin{equation}\label{eq:vdvp}
				v^{p_1^*-2}\sim\biggl[\bigl(1+|y|\bigr)^2+|z|^2\biggr]^{\frac{(2-p)n-p}{2(p-1)}}\quad\mathrm{and}\quad\quad|Dv|^{p-2}\sim\biggl[\bigl(1+|y|\bigr)^2+|z|^2\biggr]^{\frac{(2-p)n+p-2}{2(p-1)}}.
			\end{equation}
			We can still assume $\varphi\in C_c^1(\mathbb{R}^n)$.
			By Fubini's theorem, it is easy to see that \eqref{eq:ptib} is equivalent to
			\begin{equation}\label{eq:ptib1}
				\int_{\mathbb{R}^{n-k}}\int_{\{|y|\leq\rho\}}|y|^{-1}|v|^{p_1^*-2}\varphi^2\,\de y\de z\leq C(n,p,k)\rho^{\beta}\int_{\mathbb{R}^{n-k}}\int_{\mathbb{R}^k}|Dv|^{p-2}|D\varphi|^2\,\de y\de z,
			\end{equation}
			where $\{|y|\leq\rho\}=\{y\in\mathbb{R}^{k}:|y|\leq\rho\}$.
			By \eqref{eq:pti} and \eqref{eq:vdvp}, we know
			\begin{equation*}
				\begin{split}
					\int_{\mathbb{R}^n}|Dv|^{p-2}|D\varphi|^2\,\de y\geq{}&C(n,p,k)\int_{\mathbb{R}^n}|y|^{-1}v^{p_1^*-2}\varphi^2 + |Dv|^{p-2}|D\varphi|^2\,\de x\\
					\geq{}&C(n,p,k)\int_{\mathbb{R}^{n-k}}\biggl(\int_{\{|y|\leq\rho\}}|y|^{-1}\varphi^2+|D\varphi|^2 \,\de y\biggr)\bigl(1+|z|^2\bigr)^{\frac{(2-p)n-p}{2(p-1)}}\,\de z.
				\end{split}
			\end{equation*}
			Defining  $\bar{\varphi}:=|y|^{-\frac{1}{2}}\varphi$ and applying the weighted Sobolev inequality in \cite[(2.3.37)]{Mv} to $\bar{\varphi}$, it holds
			\begin{equation*}
				\begin{split}
					\int_{\mathbb{R}^n}|Dv|^{p-2}|D\varphi|^2\,\de y\geq{}&C(n,p,k)\int_{\mathbb{R}^{n-k}}\biggl(\int_{\{|y|\leq\rho\}}\bar{\varphi}^2 + \Bigl|D\Bigl(|y|^{\frac{1}{2}}\bar{\varphi}\Bigr)\Bigr|^2 \,\de y\biggr)\bigl(1+|z|^2\bigr)^{\frac{(2-p)n-p}{2(p-1)}}\,\de z\\
					\geq{}&C(n,p,k)\int_{\mathbb{R}^{n-k}}\biggl(\int_{\{|y|\leq\rho\}}\bar{\varphi}^2 + |y||D\bar{\varphi}|^2 \,\de y\biggr)\bigl(1+|z|^2\bigr)^{\frac{(2-p)n-p}{2(p-1)}}\,\de z\\
					\geq{}&C(n,p,k)\int_{\mathbb{R}^{n-k}}\biggl(\int_{\{|y|\leq\rho\}}|y|^{-\frac{q}{2}}|\varphi|^q \,\de y\biggr)^{\frac{2}{q}}\bigl(1+|z|^2\bigr)^{\frac{(2-p)n-p}{2(p-1)}}\,\de z,
				\end{split}
			\end{equation*}	
			where $q>2$ is a constant.
			Therefore, by H$\rm{\ddot{o}}$lder's inequality, one can deduce
			\begin{equation*}
				\begin{split}
					\int_{\mathbb{R}^{n-k}}\int_{\{|y|\leq\rho\}}&|y|^{-1}|v|^{p_1^*-2}\varphi^2\,\de y\de z\leq{} C(n,p,k)\int_{\mathbb{R}^{n-k}}\int_{\{|y|\leq\rho\}}|y|^{-1}\varphi^2 \,\de y(1+|z|^2)^{\frac{(2-p)n-p}{2(p-1)}}\,\de z\\
					\leq{}&C(n,p,k)\rho^{k\bigl(1-\frac{2}{q}\bigr)}\int_{\mathbb{R}^{n-k}}\biggl(\int_{\{|y|\leq\rho\}}|y|^{-\frac{q}{2}}|\varphi|^q\biggr)^{\frac{2}{q}}\bigl(1+|z|^2\bigr)^{\frac{(2-p)n-p}{2(p-1)}}\,\de z\\
					\leq{}&C(n,p,k)\rho^{k\bigl(1-\frac{2}{q}\bigr)}\int_{\mathbb{R}^{n-k}}|Dv|^{p-2}|D\varphi|^2 \,\de y.
				\end{split}
			\end{equation*}	
			Thus, \eqref{eq:ptib1} has been proved, so \eqref{eq:ptib} holds.
			
			To prove \eqref{eq:ptiob}, we will use Lemma \ref{lem:sobs}. 		First, since $|x|\leq|y|+|z|$ and $\mathbb{R}^n\backslash B(0,\rho^{-1})\subseteq(\mathbb{R}^n\backslash Y_{\frac{1}{2}\rho^{-1}})\cup(\mathbb{R}^n\backslash Z_{\frac{1}{2}\rho^{-1}})$,  we have
			\begin{equation}\label{eq:cfb}
				\int_{\mathbb{R}^n\backslash B(0,\rho^{-1})}|y|^{-1}v^{p_1^*-2}\varphi^2\,\de x\leq\int_{\mathbb{R}^n\backslash Y_{\frac{1}{2}\rho^{-1}}}|y|^{-1}v^{p_1^*-2}\varphi^2\,\de x+\int_{\mathbb{R}^n\backslash Z_{\frac{1}{2}\rho^{-1}}}|y|^{-1}v^{p_1^*-2}\varphi^2\,\de x.
			\end{equation}
			Combining \eqref{eq:vdvp} with the weighted Hardy-Sobolev inequality \eqref{eq:hsi} with ($R=\frac12\, \rho^{-1}$), we obtain
			\begin{equation}\label{eq:yl}
				\begin{split}
					\int_{\mathbb{R}^n\backslash Y_R}&|y|^{-1}v^{p_1^*-2}\varphi^2\,\de x\leq C(n,p,k)R^{-1}\int_{\mathbb{R}^n}|x|^{\frac{(2-p)n-p}{p-1}}\varphi^2\,\de x\\
					&\leq C(n,p,k)R^{-1}\int_{\mathbb{R}^n}|x|^{\frac{(2-p)n-p}{p-1}+2}|D\varphi|^2\,\de x\leq C(n,p,k)R^{-1}\int_{\mathbb{R}^n}|Dv|^{p-2}|D\varphi|^2\,\de x.
				\end{split}
			\end{equation}
			By \eqref{eq:vdvp} and \eqref{eq:hsi}, it holds
			\begin{equation}\label{eq:z2}
				\begin{split}
					&\quad \int_{\mathbb{R}^n\backslash Z_R}|y|^{-1}v^{p_1^*-2}\varphi^2\,\de x\\
					\leq{}&C(n,p,k)	\int_{\mathbb{R}^n\backslash Z_R}|y|^{-1}(1+|y|+|z|)^{-2}(1+|y|+|z|)^{\frac{(2-p)n-p}{p-1}+2}\varphi^2\,\de x\\
					\leq{}&C(n,p,k)R^{-1}\int_{\{|z|>R\}}\int_{\mathbb{R}^k}|y|^{-2}(1+|y|+|z|)^{\frac{(2-p)n-p}{p-1}+2}\varphi^2\,\de y\de z\\
					\leq{}&C(n,p,k)R^{-1}\int_{\{|z|>R\}}\int_{\mathbb{R}^k}\Bigl|D_y\Bigl((1+|y|+|z|)^{\frac{(2-p)n-p}{2(p-1)}+1}\varphi\Bigr)\Bigr|^2\,\de y\de z\\
					\leq{}&C(n,p,k)R^{-1}\int_{\mathbb{R}^n\backslash Z_R}(1+|x|)^{\frac{(2-p)n-p}{p-1}}\varphi^2\,\de x + C(n,p,k)R^{-1}\int_{\mathbb{R}^n}|Dv|^{p-2}|D\varphi|^2\,\de x,
				\end{split}
			\end{equation}
			where $D_y:=(D_{y_1},D_{y_2},\ldots,D_{y_k})$ and $|D_y\varphi|\leq|D\varphi|$ holds naturally.
			By mollifying $\varphi$ and multiplying by a smooth cutoff $\eta\in C_0^{\infty}\bigl(\mathbb{R}^n\backslash Z_R\bigr)$, we may assume without loss of generality that $\varphi\in C_0^{1}\bigl(\mathbb{R}^n\backslash Z_R\bigr)$.
			Thus, applying the Caffarelli-Kohn-Nirenberg inequality \eqref{eq:ckni} (with $r=q=2$, $\beta=\frac{(2-p)n-p}{2(p-1)}$ and $\alpha=\frac{(2-p)n-p}{2(p-1)}+1$) to the first term in the left-hand side of \eqref{eq:z2}, we obtain
			\begin{equation}\label{eq:z1}
				\begin{split}
					\int_{\mathbb{R}^n\backslash Z_R}(1+|x|)^{\frac{(2-p)n-p}{p-1}}\varphi^2\,\de x\leq{}& C(n,p,k)\int_{\mathbb{R}^n}|x|^{\frac{(2-p)n-p}{p-1}}\varphi^2\,\de x\\
					\leq{}&C(n,p,k)\int_{\mathbb{R}^n}|x|^{\frac{(2-p)n-p}{p-1}+2}|D\varphi|^2\,\de x\\
					\leq{}&C(n,p,k)\int_{\mathbb{R}^n}(1+|x|)^{\frac{(2-p)n-p}{p-1}+2}|D\varphi|^2\,\de x\\
					\leq{}&C(n,p,k)\int_{\mathbb{R}^n}|Dv|^{p-2}|D\varphi|^2\,\de x,
				\end{split}
			\end{equation}
			where we have used the fact $\varphi\in C_0^{1}\bigl(\mathbb{R}^n\backslash Z_R\bigr)$.
			Then combining \eqref{eq:z1} with \eqref{eq:z2}, we get
			\begin{equation}\label{eq:zl}
				\int_{\mathbb{R}^n\backslash Z_R}|y|^{-1}v^{p_1^*-2}\varphi^2\,\de x\leq C(n,p,k)R^{-1}\int_{\mathbb{R}^n}|Dv|^{p-2}|D\varphi|^2\,\de x.
			\end{equation}
			We complete the proof of \eqref{eq:ptiob} by combining \eqref{eq:yl}, \eqref{eq:zl} and \eqref{eq:cfb}.
		\end{proof}
		
		\begin{proof}[Proof of Proposition \ref{prop:ce}]
			Let $\{\varphi_i\}_{i\in \mathbb N^+}$ be a sequence of functions in $D^{1,2}(\mathbb{R}^n;|Dv|^{p-2})$ with uniformly bounded norm.
			It follows from \eqref{eq:pti} that $\varphi_i$ is also uniformly bounded in $L^2(\mathbb{R}^n;|y|^{-1}v^{p_1^*-2})$.
			Since $|Dv|^{p-2}$ is locally bounded away from $0$ and $\infty$ in $\mathbb{R}^n$, $\{\varphi_i\}_{i\in \mathbb N^+}$ is also uniformly bounded in $D_{\rm{loc}}^{1,2}(\mathbb{R}^n)$.
			By the Rellich-Kondrachov compact embedding theorem and a diagonal argument we deduce that, up to a subsequence, $\varphi_i$ converges to some function $\varphi$ both weakly in $D^{1,2}(\mathbb{R}^n;|Dv|^{p-2})\cap L^2(\mathbb{R}^n;|y|^{-1}v^{p_1^*-2})$ and strongly in $L_{\rm{loc}}^2(\mathbb{R}^n)$.
			Since $|y|^{-1}v^{p_1^*-2}$ is locally bounded away from $0$ and $\infty$ in $\mathbb{R}^n\backslash Y_0$, we can deduce that $\varphi_i$ also converges strongly to $\varphi$ in $L_{\rm{loc}}^2(\mathbb{R}^n\backslash Y_0,|y|^{-1}v^{p_1^*-2})$.
			
			On the other hand, \eqref{eq:ptib} and \eqref{eq:ptiob} implies that, for any $\rho\in(0,1)$,
			\begin{equation*}
				\int_{Y_\rho}|y|^{-1}v^{p_1^*-2}|\varphi_i|^2\,\de x\leq C(n,p,k)\rho^{\beta}\quad\mathrm{and}\quad\int_{\mathbb{R}^n\backslash B(0,\rho^{-1})}|y|^{-1}v^{p_1^*-2}|\varphi_i|^2\,\de x\leq C(n,p,k)\rho^{-1}.
			\end{equation*}
			We conclude the proof by the strong convergence of $\varphi_i$ inside $K_\rho:=\overline{B(0,\rho^{-1})\backslash Y_\rho}$, together with the arbitrariness of $\rho$ (that can be chosen arbitrarily small).
		\end{proof}
		As we shall see later, the previous result allows us to deal with the case $\frac{2n}{n+1}<p<n$, but is not enough for the small range case.
		
		\subsection{More delicate estimate for case $1<p\leq\frac{2n}{n+1}$.}
		\begin{lemma}\label{le:cl}
			Let $1<p\leq\frac{2n}{n+1}$ {\rm(which means $p<p_1^*\leq2$)} and $\{\phi_i\}_{i\in \mathbb N^+}$ be a sequence of functions in $D^{1,p}(\mathbb{R}^n)$ satisfying
			\begin{equation}\label{eq:pnb}
				\int_{\mathbb{R}^n}\bigl(|Dv|+\varepsilon_i|D\phi_i|\bigr)^{p-2}|D\phi_i|^2 \,\de x\leq1,
			\end{equation}
			where $\varepsilon_i\in(0,1)$ is a sequence that converges to $0$.
			Then up to a subsequence, $\phi_i$ converges weakly in $D^{1,p}(\mathbb{R}^n)$ to some function $\phi\in D^{1,p}(\mathbb{R}^n)\cap L^2(\mathbb{R}^n;|y|^{-1}v^{p_1^*-2})$.
			In addition, given any $C_1\geq0$, 
			\begin{equation}\label{eq:nc}
				\int_{\mathbb{R}^n}|y|^{-1}\frac{(v+C_1\varepsilon_i\phi_i)^{p_1^*}}{v^2+|\varepsilon_i\phi_i|^2}|\phi_i|^2\,\de x \ \rightarrow \ \int_{\mathbb{R}^n}|y|^{-1}v^{p_1^*-2}|\phi|^2\,\de x,\quad\quad \mathrm{as}\ i\rightarrow\infty.
			\end{equation}
		\end{lemma}
		
		\begin{proof}
			Replacing $\phi_i$ by $|\phi_i|$, we can assume $\phi_i\geq 0$. By H$\rm{\ddot{o}}$lder's inequality, one has
			\begin{equation*}
				\begin{split}
					\int_{\mathbb{R}^n}|D\phi_i|^p\,\de x\leq{}&\biggl(\int_{\mathbb{R}^n}\bigl(|Dv|+\varepsilon_i|D\phi_i|\bigr)^{p-2}|D\phi_i|^2\,\de x\biggr)^{\frac{p}{2}}\biggl(\int_{\mathbb{R}^n}\bigl(|Dv|+\varepsilon_i|D\phi_i|\bigr)^p\,\de x\biggr)^{1-\frac{p}{2}}\\
					\leq{}&C(n,p,k)\biggl(\int_{\mathbb{R}^n}\bigl(|Dv|+\varepsilon_i|D\phi_i|\bigr)^{p-2}|D\phi_i|^2\,\de x\biggr)^{\frac{p}{2}}\biggl(1+\varepsilon_i^p\int_{\mathbb{R}^n}|D\phi_i|^p\biggr)^{1-\frac{p}{2}}.
				\end{split}
			\end{equation*}
			Combining this with \eqref{eq:pnb} gives
			\begin{equation}\label{eq:hb}
				\biggl(\int_{\mathbb{R}^n}|D\phi_i|^p\,\de x\biggr)^{\frac{2}{p}}\leq C(n,p,k)\int_{\mathbb{R}^n}\bigl(|Dv|+\varepsilon_i|D\phi_i|\bigr)^{p-2}|D\phi_i|^2\,\de x\leq C(n,p,k).
			\end{equation}
			Therefore, up to a subsequence, $\phi_i$ converges weakly in $D^{1,p}(\mathbb{R}^n)$ and also a.e. to some function $\phi\in D^{1,p}(\mathbb{R}^n)$.
			Then we need to show $\phi\in L^2(\mathbb{R}^n;|y|^{-1}v^{p_1^*-2})$ and the validity of \eqref{eq:nc}.
			
			We will split the rest of our proof into two steps: In step 1, we first carry out our proof under the assumption that $\varepsilon_i\phi_i\leq\zeta v$, for some small constant $\zeta=\zeta(n,p,C_1)\in(0,1)$  to be determined later.  In step 2, we will remove the assumption.
			
			\medskip	
			\textit{ Step 1: prove \eqref{eq:nc} under the assumption: $\varepsilon_i\phi_i\leq\zeta v$}.
			By direct calculations, one has
			\begin{equation*}
				\begin{split}
					\int_{\mathbb{R}^n}|y|^{-1}(v+&\varepsilon_i\phi_i)^{p_1^*-2}|\phi_i|^2\,\de x={}\int_{\mathbb{R}^{n-k}}\int_{\mathbb{R}^k}|y|^{-1}(v+\varepsilon_i\phi_i)^{p_1^*-2}|\phi_i|^2\,\de y\de z\\
					\leq{}&C(n,p,k)\int_{\mathbb{R}^{n-k}}\int_{\mathbb{R}^k}|y|^{-1}(v+\varepsilon_i\phi_i)^{p_1^*-2}\Bigl(1+\frac{\varepsilon_i\phi_i}{v}\Bigr)^{p-2}|\phi_i|^2\,\de y\de z.
				\end{split}
			\end{equation*}
			According to Lemma \ref{le:hi}, for any $\xi\geq1$, there exists a constant $C=C(k,p,\xi)$ such that
			\begin{equation*}
				C(k,p,\xi)\int_{\mathbb{R}^{k}}|w|^p|y|^{-1}\bigl[(1+|y|)^{p-1}\bigr]^{\xi-1}\,\de y\leq\int_{\mathbb{R}^k}|D_yw|^p\bigl[(1+|y|)^{p-1}\bigr]^{\xi}\,\de y.
			\end{equation*}
			For any constant $C\geq1$, changing $y\rightarrow y/C$ in the above equation, one obtains
			\begin{equation*}
				C(k,p,\xi)\int_{\mathbb{R}^{k}}|w|^pC|y|^{-1}\biggl[\biggl(1+\frac{|y|}{C}\biggr)^{p-1}\biggr]^{\xi-1}C^{-k}\,\de y \leq \int_{\mathbb{R}^k}C^{-p}|D_yw|^p\biggl[\biggl(1+\frac{|y|}{C}\biggr)^{p-1}\biggr]^{\xi}C^{-k}\,\de y.
			\end{equation*}
			Since $-2p<0$, for any $C\geq1$, it holds
			\begin{equation}\label{eq:hpi}
				\begin{split}
					C(k,p,\xi)\int_{\mathbb{R}^{k}}|w|^p|y|^{-1}\bigl[(C+|y|)^{p-1}\bigr]^{\xi-1}\,\de y\leq{}& C^{-2p}\int_{\mathbb{R}^k}|D_yw|^p\bigl[(C+|y|)^{p-1})\bigr]^{\xi}\,\de y\\
					\leq{}&\int_{\mathbb{R}^k}|D_yw|^p\bigl[(C+|y|)^{p-1})\bigr]^{\xi}\,\de y.
				\end{split}
			\end{equation}
			Next, \eqref{eq:vdvp} implies $v^{p_1^*-2}\sim\bigl[\bigl(1+|z|+|y|\bigr)^{p-1}\bigr]^{\xi-1}$, so we obtain
			\begin{equation*}
				\begin{split}
					\int_{\mathbb{R}^n}|y|^{-1}(v+&\varepsilon_i\phi_i)^{p_1^*-2}|\phi_i|^2\,\de x\\
					\leq{}&C(n,p,k)\int_{\mathbb{R}^{n-k}}\int_{\mathbb{R}^k}|y|^{-1}(v+\varepsilon_i\phi_i)^{p_1^*-2}\biggl(1+\frac{\varepsilon_i\phi_i}{v}\biggr)^{p-2}|\phi_i|^2\,\de y\de z\\
					\leq{}&C(n,p,k)\int_{\mathbb{R}^{n-k}}\int_{\mathbb{R}^k}|y|^{-1}\bigl[\bigl(1+|z|+|y|\bigr)^{p-1}\bigr]^{\xi-1}|w_i|^p\,\de y\de z.
				\end{split}
			\end{equation*}
			When integrating on $\mathbb{R}^k$, one can consider $1+|z|$ as a constant.
			By approximation, we can apply \eqref{eq:hpi} with
			\begin{equation*}
				\xi=\frac{(2-p_1^*)(n-p)}{(p-1)^2}+1\quad\quad\mathrm{and}\quad\quad w=w_i:=\biggl(1+\frac{\varepsilon_i\phi_i}{v}\biggr)^{\frac{p-2}{p}}|\phi_i|^{\frac{2}{p}},
			\end{equation*}
			so we obtains
			\begin{equation*}
				\begin{split}
					\int_{\mathbb{R}^n}|y|^{-1}(v+\varepsilon_i\phi_i)^{p_1^*-2}|\phi_i|^2\,\de x\leq{}&C(n,p,k)\int_{\mathbb{R}^{n-k}}\int_{\mathbb{R}^k}|D_yw_i|^p\bigl[\bigl(1+|z|+|y|\bigr)^{p-1}\bigr]^{\xi}\,\de y\de z\\
					\leq{}&C(n,p,k)\int_{\mathbb{R}^{n-k}}\int_{\mathbb{R}^k}|Dw_i|^p\bigl[\bigl(1+|z|+|y|\bigr)^{p-1}\bigr]^{\xi}\,\de y\de z\\
					\leq{}&C(n,p,k)\int_{\mathbb{R}^n}|Dw_i|^p\bigl[\bigl(1+|x|\bigr)^{p-1}\bigr]^{\xi}\,\de x,
				\end{split}
			\end{equation*}
			where we have used $|D_yw_i|\leq|Dw_i|$.
			By the assumption $\varepsilon_i\phi_i\leq\zeta v$, it holds
			\begin{equation*}
				\begin{split}
					|Dw_i|^p\leq{}& C(p)\biggl[\biggl(1+\frac{\varepsilon_i\phi_i}{v}\biggr)^{-2}|\phi_i|^2\biggl(\frac{\varepsilon_i\phi_i|Dv|}{v^2}+\frac{\varepsilon_i|D\phi_i|}{v}\biggr)^p+\biggl(1+\frac{\varepsilon_i\phi_i}{v}\biggr)^{p-2}|\phi_i|^{2-p}|D\phi_i|^p\biggr]\\
					\leq{}&C(p)\biggl[|\phi_i|^2\biggl(\frac{\zeta|Dv|^2}{v}+\frac{\varepsilon_i|D\phi_i|}{v}\biggr)^p+|\phi_i|^{2-p}|D\phi_i|^p\biggr].
				\end{split}
			\end{equation*}
			Then one can apply \eqref{eq:ni2} ($a=|\phi_i|, b=|D\phi_i|, r=|x|,\varepsilon=\varepsilon_i$ and $s=1$) to deduce that, for any $\varepsilon_0>0$, there exists a $\zeta>0$ small enough such that
			\begin{equation*}
				\begin{split}
					\int_{\mathbb{R}^n}|y|^{-1}v^{p_1^*-2}|\phi_i|^2\,\de x\leq{}&
					2^{2-p_1^*}\int_{\mathbb{R}^n}|y|^{-1}(v+\varepsilon_i\phi_i)^{p_1^*-2}|\phi_i|^2\,\de x\\
					\leq{}&C(n,p,k)\| w_i\|_{D^{1,p}(\mathbb{R}^n;[(1+|x|)^{p-1}]^{\xi})}^p\\
					\leq{}& C(n,p,k)\varepsilon_0\int_{\mathbb{R}^n}|x|^{-1}(1+|x|)^{\frac{-(n-p)(p_1^*-2)}{p-1}}|\phi_i|^2\,\de x\\
					&+C(\varepsilon_0,n,p,k)\int_{\mathbb{R}^n}\Bigl(\bigl(1+|x|\bigr)^{-\frac{n-1}{p-1}}+\varepsilon_i|D\phi_i|\Bigr)^{p-2}|D\phi_i|^2\,\de x\\
					\leq{}&C(n,p,k)\varepsilon_0\int_{\mathbb{R}^n}|y|^{-1}v^{p_1^*-2}|\phi_i|^2\,\de x\\
					&+C(\varepsilon_0,n,p,k)\int_{\mathbb{R}^n}\bigl(|Dv|+\varepsilon_i|D\phi_i|\bigr)^{p-2}|D\phi_i|^2\,\de x.
				\end{split}
			\end{equation*}
			Thus, fixing $\varepsilon_0$ small enough so that $C(n,p,k)\varepsilon_0<\frac{1}{2}$, it follows from \eqref{eq:pnb} and the above inequality that
			\begin{equation}\label{eq:tnb}
				\begin{split}
					\int_{\mathbb{R}^n}|y|^{-1}v^{p_1^*-2}|\phi_i|^2\,\de x+\| w_i&\|_{D^{1,p}(\mathbb{R}^n;[(1+|x|)^{p-1}]^{\xi})}^p\\
					\leq{}&C(n,p,k)\int_{\mathbb{R}^n}\bigl(|Dv|+\varepsilon_i|D\phi_i|\bigr)^{p-2}|D\phi_i|^2\,\de x\leq C(n,p,k).
				\end{split}
			\end{equation}
			Furthermore, since $\bigl[\bigl(1+|x|\bigr)^{p-1}\bigr]^{\xi}$ is locally bounded away from $0$ and $\infty$ in $\mathbb{R}^n$, the sequence $\{w_i\}_{i\in \mathbb N^+}$ is uniformly bounded in $D_{\rm{loc}}^{1,p}(\mathbb{R}^n)$.
			By the Rellich-Kondrachov theorem, there is a subsequence (still note by $\{w_i\}_{i\in \mathbb N^+}$), converges strongly in $L^q_{\rm{loc}}(\mathbb{R}^n)$ $(\forall\,q\in(p,p^*))$ and a.e. to some function $\phi$.
			Note that, for any $R>0$,
			\begin{equation*}
				\biggl\|\biggl(1+\frac{\varepsilon_i\phi_i}{v}\biggr)^{\frac{p-2}{p}}|\phi_i|^{\frac{2}{p}}-|\phi|^{\frac{2}{p}}\biggr\|_{L^q(B(0,R))}\rightarrow0,\quad\quad\mathrm{as}\: i\rightarrow\infty,
			\end{equation*}
			so one has $\bigl\||\phi_i|^\frac{2}{p}-|\phi|^\frac{2}{p}\bigr\|_{L^q(B(0,R))}\rightarrow0$. Then $\bigl\||\phi_i|^2-|\phi|^2\bigr\|_{L^{\frac{q}{p}}(B(0,R))}\rightarrow0$.
			We deduce from the dominated convergence theorem that, for any $R>0$,
			\begin{equation}\label{eq:bnc}
				\int_{B(0,R)}|y|^{-1}\frac{(v+C_1\varepsilon_i\phi_i)^{p_1^*}}{v^2+|\varepsilon_i\phi_i|^2}|\phi_i|^2\,\de x\rightarrow\int_{B(0,R)}|y|^{-1}v^{p_1^*-2}|\phi|^2\,\de x,\quad\quad\mathrm{as}\:i\rightarrow\infty.
			\end{equation}
			
			For any $R>1$, letting $\tau:=\sqrt{(R^2-|z|^2)_+}$ and using polar coordinates, one gets
			\begin{equation*}
				\begin{split}
					&\int_{\mathbb{R}^n\backslash B(0,R)}|y|^{-1}v^{p_1^*-2}|\phi_i|^2\,\de x\leq C\int_{\mathbb{R}^n\backslash B(0,R)}|y|^{-1}v^{p_1^*-2}\biggl(1+\frac{\varepsilon_i\phi_i}{v}\biggr)^{p-2}|\phi_i|^2\,\de x\\
					&\leq{}C\int_{\mathbb{R}^n\backslash B(0,R)}|y|^{-1}\bigl[(1+|y|)^2+|z|^2\bigr]^{\frac{(2-p)n-p}{2(p-1)}}|w_i|^p\,\de x\\
					&\leq{}C\int_{\mathbb{R}^{n-k}}\int_{|y|\geq\tau}|y|^{-1}\bigl[(1+|y|)^2+|z|^2\bigr]^{\frac{(2-p)n-p}{2(p-1)}}|w_i(y,z)|^p\,\de y\de z\\
					&\leq{}C\int_{\mathbb{R}^{n-k}}\int_{\mathbb{S}^{k-1}}\int_{\tau}^{\infty}r_1^{k-2}\bigl[(1+r_1)^2+|z|^2\bigr]^{\frac{(2-p)n-p}{2(p-1)}}|w_i(r_1\theta,z)|^p\,\de r_1\de\theta\de z\\
					&\leq{}C\int_{\mathbb{R}^{n-k}}\int_{\mathbb{S}^{k-1}}\int_{\tau}^{\infty}\int_{r_1}^{\infty}r_1^{k-2}\bigl[(1+r_1)^2+|z|^2\bigr]^{\frac{(2-p)n-p}{2(p-1)}}|w_i(t\theta,z)|^{p-1}|Dw_i|\,\de t\de r_1\de\theta\de z,
				\end{split}
			\end{equation*}
			where $C=C(n,p,k)$.
			Since $\frac{(2-p)n-p}{2(p-1)}\geq0$ (recalling $ p\leq\frac{2n}{n+1}$), by Fubini's theorem and H$\rm{\ddot{o}}$lder's inequality, one can obtain
			\begin{equation*}
				\begin{split}
					&\int_{\mathbb{R}^n\backslash B(0,R)}|y|^{-1}v^{p_1^*-2}|\phi_i|^2\,\de x
					\\
					&\leq{} C\int_{\mathbb{R}^{n-k}}\int_{\mathbb{S}^{k-1}}\int_{\tau}^{\infty}\int_{\tau}^{t}\Bigl(r_1^{k-2}\bigl[(1+r_1)^2+|z|^2\bigr]^{\frac{(2-p)n-p}{2(p-1)}}|w_i(t\theta,z)|^{p-1}|Dw_i|\Bigr)\,\de r_1\de t\de\theta\de z\\
					&\leq{}C\int_{\mathbb{R}^{n-k}}\int_{\mathbb{S}^{k-1}}\int_{\tau}^{\infty}t^{k-1}\bigl[(1+t)^2+|z|^2\bigr]^{\frac{(2-p)n-p}{2(p-1)}}|w_i(t\theta,z)|^{p-1}|Dw_i|\,\de t\de\theta\de z\\
					&\leq{}C\biggl(\int_{\mathbb{R}^{n-k}}\int_{\mathbb{S}^{k-1}}\int_{\tau}^{\infty}t^{k-2}\bigl[(1+t)^2+|z|^2\bigr]^{\frac{(2-p)n-p}{2(p-1)}}|w_i(t\theta,z)|^p\,\de t\de\theta\de z\biggr)^{\frac{p-1}{p}}\\
					&\quad\cdot\biggl(\int_{\mathbb{R}^{n-k}}\int_{\mathbb{S}^{k-1}}\int_{\tau}^{\infty}t^{k-2+p}\bigl[(1+t)^2+|z|^2\bigr]^{\frac{(2-p)n-p}{2(p-1)}}|Dw_i|^p\,\de t\de\theta\de z\biggr)^{\frac{1}{p}}\\
					&\leq{}C\biggl(\int_{\mathbb{R}^n\backslash B(0,R)}|y|^{-1}v^{p_1^*-2}|\phi_i|^2\,\de x\biggr)^{\frac{p-1}{p}}\cdot\biggl(\int_{\mathbb{R}^n\backslash B(0,R)}\bigl(1+|x|\bigr)^{\frac{(2-p)n-p}{p-1}+p-1}|Dw_i|^p\,\de x\biggr)^{\frac{1}{p}}.
				\end{split}
			\end{equation*}
			Therefore, we get
			\begin{equation*}
				\begin{split}
					\int_{\mathbb{R}^n\backslash B(0,R)}|y|^{-1}\frac{(v+C_1\varepsilon_i\phi_i)^{p_1^*}}{v^2+|\varepsilon_i\phi_i|^2}|\phi_i|^2\,\de x\leq{}&\int_{\mathbb{R}^n\backslash B(0,R)}|y|^{-1}v^{p_1^*-2}|\phi_i|^2\,\de x\\
					\leq{}&\int_{\mathbb{R}^n\backslash B(0,R)}\bigl(1+|x|\bigr)^{\frac{(2-p)n-p}{p-1}+p-1}|Dw_i|^p\,\de x.
				\end{split}
			\end{equation*}
			Similar to \eqref{eq:tnb}, we can apply \eqref{eq:ni1} ($a=|\phi_i|, b=|D\phi_i|, r=|x|,\varepsilon=\varepsilon_i$ and $s=1$) to get that, for any $\varepsilon_0>0$, there exist a $\zeta>0$ small enough such that
			\begin{equation*}
				\begin{split}
					\int_{\mathbb{R}^n\backslash B(0,R)}|y|^{-1}\frac{(v+C_1\varepsilon_i\phi_i)^{p_1^*}}{v^2+|\varepsilon_i\phi_i|^2}|\phi_i|^2\,\de x\leq{}&C(n,p,k,C_1)\varepsilon_0\int_{\mathbb{R}^n\backslash B(0,R)}|y|^{-1}v^{p_1^*-2}|\phi_i|^2\,\de x\\
					+C(n,p&,k,C_1,\varepsilon_0)R^{-1}\int_{\mathbb{R}^n\backslash B(0,R)}\bigl(|Dv|+\varepsilon_i|D\phi_i|\bigr)^{p-2}|D\phi_i|^2\,\de x\\
					\leq{}&C(n,p,k,C_1)\varepsilon_0\int_{\mathbb{R}^n\backslash B(0,R)}|y|^{-1}\frac{(v+C_1\varepsilon_i\phi_i)^{p_1^*}}{v^2+|\varepsilon_i\phi_i|^2}|\phi_i|^2\,\de x\\
					+C(n,p&,k,C_1,\varepsilon_0)R^{-1}\int_{\mathbb{R}^n\backslash B(0,R)}\bigl(|Dv|+\varepsilon_i|D\phi_i|\bigr)^{p-2}|D\phi_i|^2\,\de x.
				\end{split}
			\end{equation*}
			Thus, by fixing $\varepsilon_0$ small so that $C(n,p,k,C_1)\varepsilon_0\leq1/2$, it follows that
			\begin{equation*}
				\begin{split}
					\int_{\mathbb{R}^n\backslash B(0,R)}|y|^{-1}&\frac{(v+C_1\varepsilon_i\phi_i)^{p_1^*}}{v^2+|\varepsilon_i\phi_i|^2}|\phi_i|^2\,\de x\\
					\leq{}& C(n,p,k,C_1)R^{-1}\int_{\mathbb{R}^n\backslash B(0,R)}\bigl(|Dv|+\varepsilon_i|D\phi_i|\bigr)^{p-2}|D\phi_i|^2\,\de x\leq C(n,p,k,C_1)R^{-1}.
				\end{split}
			\end{equation*}
			Combining this with \eqref{eq:tnb} and \eqref{eq:bnc}, by the arbitrariness of $R$, we conclude that $\phi\in L^2(\mathbb{R}^n,|y|^{-1}v^{p_1^*-2})$ and that \eqref{eq:nc} holds.
			This concludes the proof under the assumption that $\varepsilon_i\phi_i\leq\zeta v$ with $\zeta=\zeta(n,p,k,C_1)$ sufficiently small.
			
			\medskip	
			\textit{ Step 2: prove \eqref{eq:nc} in the general case}.
			Throughout this part, we assume that $\zeta=\zeta(n,p,k,C_1)>0$ is a small constant such that Step 1 applies.
			Observe that, by \eqref{eq:ele}, $\zeta v$ is a super-solution for the operator
			\begin{equation*}
				L_v[\psi]:=-\mathrm{div}\Bigl(\bigl(|Dv|+|D\psi|\bigr)^{p-2}+(p-2)\bigl(|Dv|+|D\psi|\bigr)^{p-3}|D\psi|D\psi\Bigr),
			\end{equation*}
			namely, $L_v[\zeta v]\geq0$.
			Therefore, multiplying $L_v[\zeta v]\geq0$ by $(\varepsilon_i\phi_i-\zeta v)_+$ and integrating by parts, we get
			\begin{equation}\label{eq:sp}
				\begin{split}
					\int_{\mathbb{R}^n}\bigl(|Dv|+\zeta|Dv|\bigr)^{p-2}\zeta Dv&\,\cdot\, D(\varepsilon_i\phi_i-\zeta v)_+ \,\de x\\
					&+(p-2)\int_{\mathbb{R}^n}\bigl(|Dv|+\zeta|Dv|\bigr)^{p-3}\zeta^2|Dv|\cdot D(\varepsilon_i\phi_i-\zeta v)_+\,\de x\geq0.
				\end{split}
			\end{equation}
			By the convexity of
			\begin{equation*}
				\mathbb{R}^n\ni x\mapsto F_t(x):=(t+|x|)^{p-2}|x|^2,\quad\quad t\geq0,
			\end{equation*}
			we have
			\begin{equation*}
				F_t(x)+DF_t(x)\cdot(x'-x)\leq F_t(x'),\quad\quad\forall\,x,x'\in\mathbb{R}^n,\:\forall\,t\geq0.
			\end{equation*}
			Hence, applying this inequality with $t=|Dv|$, $x=\zeta Dv$ and $x'=\varepsilon_iD\phi_i$, it follows from \eqref{eq:sp} that
			\begin{equation}\label{eq:bp}
				\begin{split}
					C(n,p,k)\varepsilon_i^{-2}&\int_{\{\varepsilon_i\phi_i>\zeta v\}}|Dv|^p \,\de x\leq\varepsilon_i^{-2}\int_{\{\varepsilon_i\phi_i>\zeta v\}}\bigl(|Dv|+\zeta|Dv|\bigr)^{p-2}\zeta^2|Dv|^2\,\de x\\
					&\leq{}\varepsilon_i^{-2}\int_{\{\varepsilon_i\phi_i>\zeta v\}}\bigl(|Dv|+\zeta|Dv|\bigr)^{p-2}\zeta^2|Dv|^2\,\de x\\
					&\quad+\varepsilon_i^{-2}\int_{\{\varepsilon_i\phi_i>\zeta v\}}\bigl(|Dv|+\zeta|Dv|\bigr)^{p-2}\zeta Dv\,\cdot\, D(\varepsilon_i\phi_i-\zeta v)_+ \,\de x\\
					&\quad+\varepsilon_i^{-2}(p-2)\int_{\{\varepsilon_i\phi_i>\zeta v\}}\bigl(|Dv|+\zeta|Dv|\bigr)^{p-3}\zeta^2|Dv|\cdot D(\varepsilon_i\phi_i-\zeta v)_+\,\de x\\
					&\leq{}\int_{\{\varepsilon_i\phi_i>\zeta v\}}\bigl(|Dv|+\varepsilon_i|D\phi_i|\bigr)^{p-2}|D\phi|^2\,\de x.
				\end{split}
			\end{equation}
			We now write $\phi_i=\phi_{i,1}+\phi_{i,2}$, where
			\begin{equation*}
				\phi_{i,1}:=\min\biggl\{\phi_i,\frac{\zeta v}{\varepsilon_i}\biggr\}\quad\quad\text{and}\quad\quad\phi_{i,2}:=\phi_i-\phi_{i,1}.
			\end{equation*}
			Note that, as a consequence of \eqref{eq:pnb} and \eqref{eq:bp}, we have
			\begin{equation*}
				\begin{split}
					&\int_{\mathbb{R}^n}\bigl(|Dv|+\varepsilon_i|D\phi_{i,1}|\bigr)^{p-2}|D\phi_{i,1}|^2\,\de x\\
					&\leq{}\varepsilon_i^{-2}\int_{\{\varepsilon_i\phi_i>\zeta v\}}\bigl(|Dv|+\zeta|Dv|\bigr)^{p-2}\zeta^2|Dv|^2\,\de x + \int_{\{\varepsilon_i\phi_i\leq\zeta v\}}\bigl(|Dv|+\varepsilon_i|D\phi_i|\bigr)^{p-2}|D\phi_i|^2\,\de x\\
					&\leq{}C(n,p,k)\int_{\mathbb{R}^n}\bigl(|Dv|+\varepsilon_i|D\phi_i|\bigr)^{p-2}|D\phi|^2\,\de x\leq C(n,p,k),
				\end{split}
			\end{equation*}
			and
			\begin{equation*}
				\begin{split}
					&\int_{\mathbb{R}^n}\bigl(|Dv|+\varepsilon_i|D\phi_{i,2}|\bigr)^{p-2}|D\phi_{i,2}|^2\,\de x\\
					&\leq{} C(n,p,k)\biggl(\int_{\{\varepsilon_i\phi_i>\zeta v\}}\bigl(|Dv|+\varepsilon_i|D\phi_i|\bigr)^{p-2}|D\phi_i|^2\,\de x
					+\varepsilon_i^{-2}\int_{\{\varepsilon_i\phi_i>\zeta v\}}\bigl(|Dv|+\zeta|Dv|\bigr)^{p-2}\zeta^2|Dv|^2\,\de x\biggr)\\
					&\leq{}C(n,p,k)\int_{\mathbb{R}^n}\bigl(|Dv|+\varepsilon_i|D\phi_i|\bigr)^{p-2}|D\phi|^2\,\de x\leq C(n,p,k),
				\end{split}
			\end{equation*}
			where we have used $|D\phi_{i,2}|^2\leq2\big(|D\phi_i|^2+\varepsilon_i^{-2}\zeta^2|Dv|^2\big)$.
			Then one obtains
			\begin{equation}\label{eq:tfpnb}
				\int_{\mathbb{R}^n}\bigl(|Dv|+\varepsilon_i|D\phi_{i,1}|\bigr)^{p-2}|D\phi_{i,1}|^2\,\de x + \int_{\mathbb{R}^n}\bigl(|Dv|+\varepsilon_i|D\phi_{i,2}|\bigr)^{p-2}|D\phi_{i,2}|^2\,\de x\leq C(n,p,k)
			\end{equation}
			Hence, it follows from the similar way as deriving \eqref{eq:hb} that
			\begin{equation}\label{eq:tfnb}
				\int_{\mathbb{R}^n}|D\phi_{i,1}|^p\,\de x+\int_{\mathbb{R}^n}|D\phi_{i,2}|^p\,\de x\leq C(n,p,k).
			\end{equation}
			In particular, since $|\{\varepsilon_i\phi_i>\zeta v\}\cap B(0,R)|\rightarrow0$ $(i\rightarrow\infty)$ for any $R>1$, we deduce that $\phi_{i,2}\rightharpoonup0$ in $D^{1,p}(\mathbb{R}^n)$.
			Thus, up to a subsequence, both $\phi_i$ and $\phi_{i,1}$ converge weakly in $D^{1,p}(\mathbb{R}^n)$ and also a.e. to the same function $\phi\in D^{1,p}(\mathbb{R}^n)$.
			
			Let $\eta=\eta(n,p,k)>0$ be a small exponent to be fixed.
			We discuss the following two cases respectively.
			
			\medskip	
			\textit{ Case 1}.
			Assume that
			\begin{equation*}
				\begin{split}
					\int_{\{\varepsilon_i\phi_i>\zeta v\}}|y|^{-1}|\phi_{i,1}|^{p_1^*}\,\de x >{}&\varepsilon_i^{-\eta}\int_{\{\varepsilon_i\phi_i>\zeta v\}}|y|^{-1}\biggl(\phi_i-\frac{\zeta v}{\varepsilon_i}\biggr)_+^{p_1^*}\,\de x\\
					={}&\varepsilon_i^{-\eta}\int_{\{\varepsilon_i\phi_i>\zeta v\}}|y|^{-1}|\phi_{i,2}|^{p_1^*}\,\de x.
				\end{split}
			\end{equation*}
			Since $\phi$ is also the limit of $\phi_{i,1}$, we can apply Step 1 to $\phi_{i,1}$ to deduce that $\phi\in L^2(\mathbb{R}^n;|y|^{-1}v^{p_1^*-2})$.
			We know that there exists a constant $C=C(n,p,k,C_1)$ such that
			\begin{equation*}
				\frac{1}{C}\varepsilon_i^{p_1^*-2}\phi_i^{p_1^*}\leq\frac{(v+C_1\varepsilon_i\phi_i)^{p_1^*}}{v^2+|\varepsilon_i\phi_i|^2}|\phi_i|^2\leq C\varepsilon_i^{p_1^*-2}\phi_i^{p_1^*}\quad\quad\text{inside }\{\varepsilon_i\phi_i>\zeta v\}.
			\end{equation*}
			Thus, by the Taylor expansion and H$\rm{\ddot{o}}$lder's inequality, we can obtain that, as $i\rightarrow\infty$,
			\begin{equation*}
				\begin{split}
					\biggl|\int_{\mathbb{R}^n}|y|^{-1}\frac{(v+C_1\varepsilon_i\phi_i)^{p_1^*}}{v^2+|\varepsilon_i\phi_i|^2}|\phi_i|^2\,\de x\:\: -\biggr.&\biggl.\int_{\mathbb{R}^n}|y|^{-1}\frac{(v+C_1\varepsilon_i\phi_{i,1})^{p_1^*}}{v^2+|\varepsilon_i\phi_{i,1}|^2}|\phi_{i,1}|^2\,\de x\biggr|\\
					\leq{}&C(n,p,k,C_1)\int_{\{\varepsilon_i\phi_i>\zeta v\}}|y|^{-1}\varepsilon_i^{p_1^*-2}|\phi_{i,1}|^{p_1^*-1}|\phi_{i,2}|\\
					\leq{}&O(\varepsilon_i^{\eta})\int_{\{\varepsilon_i\phi_i>\zeta v\}}|y|^{-1}\frac{(v+C_1\varepsilon_i\phi_{i,1})^{p_1^*}}{v^2+|\varepsilon_i\phi_{i,1}|^2}|\phi_{i,1}|^2\,\de x=O(\varepsilon_i^{\eta}).
				\end{split}
			\end{equation*}
			Thus
			\begin{equation*}
				\int_{\mathbb{R}^n}|y|^{-1}\frac{(v+C_1\varepsilon_i\phi_{i,1})^{p_1^*}}{v^2+|\varepsilon_i\phi_{i,1}|^2}|\phi_{i,1}|^2\,\de x\rightarrow\int_{\mathbb{R}^n}|y|^{-1}v^{p_1^*-2}|\phi|^2\,\de x,\quad\quad\mathrm{as}\:i\rightarrow\infty,
			\end{equation*}
			which proves \eqref{eq:nc}.
			
			\medskip	
			\textit{ Case 2}.
			Assume that
			\begin{equation}\label{eq:c2a}
				\int_{\{\varepsilon_i\phi_i>\zeta v\}}|y|^{-1}|\phi_{i,1}|^{p_1^*}\,\de x\leq\varepsilon_i^{-\eta}\int_{\{\varepsilon_i\phi_i>\zeta v\}}|y|^{-1}|\phi_{i,2}|^{p_1^*}\,\de x.
			\end{equation}
			We claim that
			\begin{equation}\label{eq:2w}
				\varepsilon_i^{p_1^*-2}\int_{\mathbb{R}^n}|y|^{-1}|\phi_{i,2}|^{p_1^*}\,\de x = O(\varepsilon_i^{\eta}).
			\end{equation}
			To prove this, let $A_i:=\{\varepsilon_i\phi_i>\zeta v\}$ and define
			\begin{equation*}
				E_i:=\biggl\{|D\phi_{i,2}|\leq\frac{|Dv|}{\varepsilon_i}\biggr\}\cap A_i,\quad\quad F_i:=\biggl\{|D\phi_{i,2}|>\frac{|Dv|}{\varepsilon_i}\biggr\}\cap A_i.
			\end{equation*}
			Then, since $|Dv|+\varepsilon_i|D\phi_{i,2}|\leq2|Dv|$ inside $E_i$, it follows from H$\rm{\ddot{o}}$lder's inequality that
			\begin{equation}\label{eq:wh}
				\begin{split}
					\int_{\mathbb{R}^n}|D&\phi_{i,2}|^p \,\de x=\int_{E_i}|D\phi_{i,2}|^p\,\de x+\int_{F_i}|D\phi_{i,2}|^p\,\de x\\
					\leq{}&\biggl(\int_{E_i}|Dv|^{p-2}|D\phi_{i,2}|^2\,\de x\biggr)^{\frac{p}{2}}\biggl(\int_{E_i}|Dv|^p\,\de x\biggr)^{1-\frac{p}{2}} + \int_{F_i}|D\phi_{i,2}|^p\,\de x\\
					\leq{}&\biggl(2^{2-p}\int_{E_i}\bigl(|Dv|+\varepsilon_i|D\phi_{i,2}|\bigr)^{p-2}|D\phi_{i,2}|^2\,\de x\biggr)^{\frac{p}{2}}\biggl(\int_{E_i}|Dv|^p\,\de x\biggr)^{1-\frac{p}{2}} + \int_{F_i}|D\phi_{i,2}|^p\,\de x\\
					\leq{}&C(n,p,k)\biggl(\int_{E_i}\bigl(|Dv|+\varepsilon_i|D\phi_{i,2}|\bigr)^{p-2}|D\phi_{i,2}|^2\,\de x\biggr)^{\frac{p}{2}}\biggl(\int_{E_i}|Dv|^p\,\de x\biggr)^{1-\frac{p}{2}} + \int_{F_i}|D\phi_{i,2}|^p\,\de x.
				\end{split}
			\end{equation}
			For any $\mu$ satisfying $\frac{n(p-1)}{n-1}<\mu<\frac{p(n-1)}{p-1}$ (since $1<p<n$), we define
			\begin{equation}\label{eq:Q}
				Q:=\frac{pn-1}{p(n-1)-\mu(p-1)}>1.
			\end{equation}
			Since $Q\Bigl(\frac{p(1-n)}{p-1}+\mu\Bigr)$ is monotone decreasing about $\mu$,
			we deduce from \eqref{eq:vdv} that
			\begin{equation*}
				\begin{split}
					v^{-p_1^*}&\Bigl((1+|x|)^{\frac{p(1-n)}{p-1}+\mu}\Bigr)^Q \leq C(n,p,k)(1+|x|)^{\frac{p(n-1)}{p-1}+\bigl(\frac{p(1-n)}{p-1}+\frac{n(p-1)}{n-1}\bigr)\frac{pn-1}{p(n-1)-(p-1)\frac{n(p-1)}{n-1}}}\\
					&={}C(n,p,k)(1+|x|)^{\frac{p(n-1)}{p-1}+\frac{1-pn}{p-1}}= C(n,p,k)(1+|x|)^{-1}\leq C(n,p,k)|y|^{-1}.
				\end{split}
			\end{equation*}
			Also, using the above inequality and \eqref{eq:c2a} together with H$\rm{\ddot{o}}$lder's inequality, we obtain
			\begin{equation}\label{eq:edvp}
				\begin{split}
					\int_{E_i}|Dv|^p\,\de x\leq{}&C(n,p,k)\int_{E_i}(1+|x|)^{\frac{p(1-n)}{p-1}}\,\de x\\
					\leq{}&C(n,p,k)\biggl[\int_{E_i}\biggl((1+|x|)^{\frac{p(1-n)}{p-1}+\mu}\biggr)^Q\,\de x\biggr]^\frac{1}{Q}\biggl(\int_{\mathbb{R}^n}(1+|x|)^{-\frac{\mu Q}{Q-1}}\,\de x\biggr)^{\frac{Q-1}{Q}}\\
					\leq{}&C(n,p,k)\biggl[\int_{A_i}\biggl(\frac{\varepsilon_i\phi_i}{\zeta v}\biggr)^{p_1^*}\biggl((1+|x|)^{\frac{p(1-n)}{p-1}+\mu}\biggr)^Q\,\de x\biggr]^{\frac{1}{Q}}\\
					\leq{}&C(n,p,k)\biggl(\varepsilon_i^{p_1^*}\int_{A_i}|y|^{-1}|\phi_i|^{p_1^*}\,\de x\biggr)^{\frac{1}{Q}}\leq{}C(n,p,k)\biggl(\varepsilon_i^{p_1^*-\eta}\int_{A_i}|y|^{-1}|\phi_{i,2}|^{p_1^*}\,\de x\biggr)^{\frac{1}{Q}},
				\end{split}
			\end{equation}
			where we have used $\frac{\mu Q}{Q-1}=\frac{\mu(pn-1)}{(p-1)(\mu+1)}>n$. Therefore, defining
			\begin{equation*}
				N_{i,2}:=\int_{E_i}(|Dv|+\varepsilon_i|D\phi_{i,2}|)^{p-2}|D\phi_{i,2}|^2\,\de x,
			\end{equation*}
			and using Hardy-Sobolev-Maz'ya inequalities, \eqref{eq:wh} and \eqref{eq:edvp}, we can deduce that
			\begin{equation}\label{eq:whj}
				\begin{split}
					\varepsilon_i^{p_1^*-2}\int_{\mathbb{R}^n}|y|^{-1}|\phi_{i,2}|^{p_1^*}\,\de x \leq{}& C(n,p,k)\varepsilon_i^{p_1^*-2}\biggl(\int_{\mathbb{R}^n}|D\phi_{i,2}|^p\,\de x\biggr)^{\frac{p_1^*}{p}}\\
					\leq{}&C(n,p,k)\varepsilon_i^{p_1^*-2}\biggl[N_{i,2}^{\frac{p_1^*}{2}}\biggl(\int_{E_i}|Dv|^p\,\de x\biggr)^{\frac{(2-p)p_1^*}{2p}} + \biggl(\int_{F_i}|D\phi_{i,2}|^p\biggr)^{\frac{p_1^*}{p}}\biggr]\\
					\leq{}&C(n,p,k)\varepsilon_i^{p_1^*-2}\biggl[N_{i,2}^{\frac{p_1^*}{2}}\biggl(\varepsilon_i^{p_1^*-\eta}\int_{A_i}|y|^{-1}|\phi_{i,2}|^{p_1^*}\,\de x\biggr)^{\frac{(2-p)p_1^*}{2pQ}} + \biggl(\int_{F_i}|D\phi_{i,2}|^p\biggr)\biggr],
				\end{split}
			\end{equation}
			where in the last inequality we have used \eqref{eq:tfnb} and $\frac{p_1^*}{p}\geq1$.
			
			Suppose first that
			\begin{equation*}
				\int_{F_i}|D\phi_{i,2}|^p\,\de x\geq N_{i,2}^{\frac{p_1^*}{2}}\biggl(\varepsilon_i^{p_1^*-\eta}\int_{A_i}|y|^{-1}|\phi_{i,2}|^{p_1^*}\,\de x\biggr)^{\frac{(2-p)p_1^*}{2pQ}}.
			\end{equation*}
			Then, since $|Dv|\leq\varepsilon_i|D\phi_{i,2}|\sim\varepsilon_i|D\phi_i|$ inside $F_i$ (recall that $\zeta<1$), \eqref{eq:pnb} and \eqref{eq:whj} yield
			\begin{equation}\label{eq:cf1}
				\begin{split}
					\varepsilon_i^{p_1^*-2}\int_{\mathbb{R}^n}|y|^{-1}|\phi_{i,2}|^{p_1^*}\,\de x \leq{}\varepsilon_i^{p_1^*-2}\int_{F_i}|D\phi_{i,2}|^p&\,\de x=\varepsilon_i^{p_1^*-p}\int_{F_i}\varepsilon_i^{p-2}|D\phi_{i,2}|^{p-2}|D\phi_{i,2}|^2\,\de x\\
					\leq{}&\varepsilon_i^{p_1^*-p}\int_{F_i}(|Dv|+\varepsilon_i|D\phi_{i,2}|)^{p-2}|D\phi_{i,2}|^2\,\de x,
				\end{split}
			\end{equation}
			which proves \eqref{eq:2w} by choosing $\eta\leq p_1^*-p$ (recall \eqref{eq:tfpnb}).
			
			Consider instead the case
			\begin{equation*}
				\int_{F_i}|D\phi_{i,2}|^p\,\de x< N_{i,2}^{\frac{p_1^*}{2}}\biggl(\varepsilon_i^{p_1^*-\eta}\int_{A_i}|Dv|^p\,\de x\biggr)^{\frac{(2-p)p_1^*}{2pQ}}.
			\end{equation*}
			Let $\theta:=\frac{(2-p)p_1^*}{2pQ}$, where $Q$ is defined in \eqref{eq:Q}. Then  (see Lemma \ref{le:mu} in the Appendix)
			\begin{equation*}
				1-\frac{p_1^*}{2}<\theta<1,
			\end{equation*}
			and \eqref{eq:whj} yields that
			\begin{equation*}
				\begin{split}
					\varepsilon_i^{p_1^*-2}\int_{\mathbb{R}^n}|y|^{-1}|\phi_{i,2}|^{p_1^*}\,\de x \leq{}C(n,p,k)\varepsilon_i^{p_1^*-2}N_{i,2}^{\frac{p_1^*}{2}}\biggl(\varepsilon_i^{p_1^*-\eta}\int_{A_i}|y|^{-1}|\phi_{i,2}|^{p_1^*}\,\de x\biggr)^{\theta}\\
					=C(n,p,k)\varepsilon_i^{p_1^*-2+(2-\eta)\theta}N_{i,2}^{\frac{p_1^*}{2}}\biggl(\varepsilon_i^{p_1^*-2}\int_{A_i}|y|^{-1}|\phi_{i,2}|^{p_1^*}\,\de x\biggr)^{\theta}.
				\end{split}
			\end{equation*}

			By recalling the definition of $N_{i,2}$ and \eqref{eq:tfpnb}, one has
			\begin{equation}\label{eq:cf2}
				\begin{split}
					\varepsilon_i^{p_1^*-2}\int_{\mathbb{R}^n}|y|^{-1}|\phi_{i,2}|^{p_1^*}\,\de x\leq{}& C(n,p,k)\varepsilon_i^{\frac{p_1^*-2+(2-\eta)\theta}{1-\theta}}\biggl(\int_{E_i}(|Dv|+\varepsilon_i|D\phi_i|)^{p-2}|D\phi_{i,2}|^2\,\de x\biggr)^{\frac{p_1^*}{2(1-\theta)}}\\
					\leq{}&C(n,p,k)\varepsilon_i^{\eta}\int_{E_i}(|Dv|+\varepsilon_i|D\phi_i|)^{p-2}|D\phi_{i,2}|^2\,\de x,
				\end{split}
			\end{equation}
			where the last inequality follows from choosing $\eta>0$ sufficiently small (notice that $p_1^*-2+2\theta>0$ and $\frac{p_1^*}{2(1-\theta)}>1$).
			This proves \eqref{eq:2w}.
			
			Now, combining \eqref{eq:c2a} and \eqref{eq:2w}, using Taylor expansion and Young's inequality, we finally get
			\begin{equation*}
				\begin{split}
					\biggl|\int_{\mathbb{R}^n}|y|^{-1}\biggr.&\biggl.\frac{(v+C_1\varepsilon_i\phi_i)^{p_1^*}}{v^2+|\varepsilon_i\phi_i|^2}|\phi_i|^2\,\de x - \int_{\mathbb{R}^n}|y|^{-1}\frac{(v+C_1\varepsilon_i\phi_{i,1})^{p_1^*}}{v^2+|\varepsilon_i\phi_{i,1}|^2}|\phi_{i,1}|^2\,\de x\biggr|\\
					\leq{}&C(n,p,k,C_1)\int_{A_i}|y|^{-1}\varepsilon_i^{p_1^*-2}|\phi_{i,1}|^{p_1^*-1}|\phi_{i,2}|\\
					\leq{}&C(n,p,k,C_1)\biggl(\varepsilon_i^{p_1^*-2}\int_{A_i}|y|^{-1}|\phi_{i,2}|^{p_1^*}\,\de x + \varepsilon_i^2\int_{A_i}|y|^{-1}\frac{(v+C_1\zeta v)^{p_1^*}}{v^2+|\zeta v|^2}|\zeta v|^2\,\de x\biggr)=o(1),
				\end{split}
			\end{equation*}
			as $i\rightarrow\infty$.
			Applying Step 1 to $\phi_{i,1}$ and using the above estimate, one obtains
			\begin{equation*}
				\int_{\mathbb{R}^n}|y|^{-1}\frac{(v+C_1\varepsilon_i\phi_{i,1})^{p_1^*}}{v^2+|\varepsilon_i\phi_{i,1}|^2}|\phi_{i,1}|^2\,\de x\rightarrow\int_{\mathbb{R}^n}|y|^{-1}v^{p_1^*-2}|\phi|^2\,\de x,\quad\quad\mathrm{as}\:i\rightarrow\infty,
			\end{equation*}
			which concludes the proof of Lemma \ref{le:cl}.
		\end{proof}
		
		In a similar way as proving Lemma \ref{le:cl}, we can show the following Orlicz-type Poincar\'e inequality.
		\begin{corollary}\label{cor:ope}
			Let $1<p\leq\frac{2n}{n+1}$.
			There exists $\varepsilon_0=\varepsilon_0(n,p,k)>0$ small enough such that the following holds: For any $\varepsilon\in(0,\varepsilon_0)$ and any $\phi\in D^{1,p}(\mathbb{R}^n)\cap D^{1,2}(\mathbb{R}^n;|Dv|^{p-2})$ with
			\begin{equation*}
				\int_{\mathbb{R}^n}\bigl(|Dv|+\varepsilon|D\phi|\bigr)^{p-2}|D\phi|^2\,\de x\leq1,
			\end{equation*}
			we have
			\begin{equation}\label{eq:tyu}
				\int_{\mathbb{R}^n}|y|^{-1}(v+\varepsilon\phi)^{p_1^*-2}|\phi|^2\,\de x\leq C(n,p,k)\int_{\mathbb{R}^n}\bigl(|Dv|+\varepsilon|D\phi|\bigr)^{p-2}|D\phi|^2\,\de x.
			\end{equation}
		\end{corollary}	
		\begin{proof}
			It suffices for us to consider the case $\phi\geq0$ (replacing $\phi$ by $|\phi|$). Let $\zeta\in(0,1)$ be the small constant proved in the proof of Lemma \ref{le:cl} with $C_1=1$.
			Decompose $\phi=\phi_1+\phi_2$, where
			\begin{equation*}
				\phi_1:=\min\biggl\{\phi,\frac{\zeta v}{\varepsilon}\biggr\}\quad\quad\text{and}\quad\quad\phi_2:=\phi-\phi_1.
			\end{equation*}
			Since $\varepsilon\phi_1\leq\zeta v$, one has $v\sim v+\varepsilon\phi_1$, so we can deduce that \eqref{eq:tyu} holds for $\phi=\phi_1$ in the similar way as proving \eqref{eq:tnb}.
			
			For $\phi_2$, we will discuss two cases.
			On the one hand, if
			\begin{equation*}
				\int_{\{\varepsilon\phi>\zeta v\}}|y|^{-1}|\phi_1|^{p_1^*}\,\de x >{}\varepsilon^{-\eta}\int_{\{\varepsilon\phi>\zeta v\}}|y|^{-1}\biggl(\phi-\frac{\zeta v}{\varepsilon}\biggr)_+^{p_1^*}\,\de x
				={}\varepsilon^{-\eta}\int_{\{\varepsilon\phi>\zeta v\}}|y|^{-1}|\phi_2|^{p_1^*}\,\de x,
			\end{equation*}
			then
			\begin{equation*}
				\begin{split}
					\int_{\mathbb{R}^n}\varepsilon^{p_1^*-2}|y|^{-1}|\phi_2|^{p_1^*}\,\de x={}&\int_{\{\varepsilon\phi>\zeta v\}}\varepsilon^{p_1^*-2}|y|^{-1}|\phi_2|^{p_1^*}\,\de x\leq{}\int_{\{\varepsilon\phi>\zeta v\}}\varepsilon^{p_1^*-2}|y|^{-1}|\phi_1|^{p_1^*}\,\de x\\
					\leq{}&\int_{\{\varepsilon\phi>\zeta v\}}|y|^{-1}\zeta^{p_1^*-2}v^{p_1^*-2}|\phi_1|^2\,\de x\leq{}\int_{\mathbb{R}^n}|y|^{-1}v^{p_1^*-2}|\phi_1|^2\,\de x.
				\end{split}
			\end{equation*}
			Thus, applying \eqref{eq:tyu} to $\phi_1$, we conclude that
			\begin{equation*}
				\begin{split}
					\int_{\mathbb{R}^n}&|y|^{-1}(v+\varepsilon\phi)^{p_1^*-2}|\phi|^2\,\de x\leq{}C(n,p,k)\int_{\mathbb{R}^n}|y|^{-1}v^{p_1^*-2}|\phi_1|^2\,\de x + \int_{\mathbb{R}^n}\varepsilon^{p_1^*-2}|y|^{-1}|\phi_2|^{p_1^*}\,\de x\\
					\leq{}&C(n,p,k)\int_{\mathbb{R}^n}|y|^{-1}v^{p_1^*-2}|\phi_1|^2\,\de x\leq{}C(n,p,k)\int_{\mathbb{R}^n}\bigl(|Dv|+\varepsilon|D\phi_1|\bigr)^{p-2}|D\phi_1|^2\,\de x\\
					\leq{}&C(n,p,k)\int_{\mathbb{R}^n}\bigl(|Dv|+\varepsilon|D\phi|\bigr)^{p-2}|D\phi|^2\,\de x,
				\end{split}
			\end{equation*}
			where the last step follows from the same reason for \eqref{eq:tfpnb}.
			
			On the other hand, if
			\begin{equation*}
				\int_{\{\varepsilon\phi>\zeta v\}}|y|^{-1}|\phi_1|^{p_1^*}\,\de x \leq\int_{\{\varepsilon\phi>\zeta v\}}|y|^{-1}|\phi_2|^{p_1^*}\,\de x,
			\end{equation*}	
			we can repeat the proofs of \eqref{eq:cf1} and \eqref{eq:cf2} with $\eta=0$ to deduce the validity of \eqref{eq:tyu} for $\phi=\phi_2$.
			Thus, by \eqref{eq:tfpnb} for $\phi$, and \eqref{eq:tyu} for $\phi=\phi_1$ and $\phi=\phi_2$, we eventually obtain
			\begin{equation*}
				\begin{split}
					\int_{\mathbb{R}^n}&|y|^{-1}(v+\varepsilon\phi)^{p_1^*-2}|\phi|^2\,\de x\leq{} C(n,p,k)\biggl(\int_{\mathbb{R}^n}|y|^{-1}v^{p_1^*-2}|\phi_1|^2\,\de x+\int_{\mathbb{R}^n}\varepsilon^{p_1^*-2}|y|^{-1}|\phi_2|^{p_1^*}\,\de x\biggr)\\
					\leq{}&C(n,p,k)\biggl(\int_{\mathbb{R}^n}\bigl(|Dv|+\varepsilon|D\phi_1|\bigr)^{p-2}|D\phi_1|^2\,\de x+\int_{\mathbb{R}^n}\bigl(|Dv|+\varepsilon|D\phi_2|\bigr)^{p-2}|D\phi_2|^2\,\de x\biggr)\\
					\leq{}&C(n,p,k)\int_{\mathbb{R}^n}\bigl(|Dv|+\varepsilon|D\phi|\bigr)^{p-2}|D\phi|^2\,\de x.
				\end{split}
			\end{equation*}
			The proof is completed.
		\end{proof}
		
		\section{Spectral Gaps}\label{sec:sg}
		Thanks to Proposition \ref{prop:ce} and spectral analysis, we will prove in this section that the linearized operator $\mathcal{L}_v$ has a discrete spectrum for any $1<p<n$ and obtain some  spectral inequalities that is crucial in the proof of our stability result.
		
		We occasionally employ the notations of (the classical and weighted) inner products
		\begin{equation*}
			\langle w,u\rangle=\int_{\mathbb{R}^n}wu\,\de x \quad	\text{and}\quad \langle w,u\rangle_*=\int_{\mathbb{R}^n}|y|^{-1}v^{p_1^*-2}wu\,\de x,
		\end{equation*} where $y$ is the partial (weight) variable and $v$ is any fixed extremal function in $\mathcal M$.
		\subsection{Proof of Theorem \ref{thm:ef}}
		
		\begin{proposition}\label{prop:ds}
			For any $1<p<n$, the linearized operator $\mathcal{L}_v$ defined by \eqref{eq:lv} has a discrete spectrum, denoted by $\{\alpha_i\}_{i=1}^{\infty}$.
		\end{proposition}
		\begin{proof}
			We show that the operator $\mathcal{L}_v^{-1} : L^2(\mathbb{R}^n;|y|^{-1}v^{p_1^*-2}) \rightarrow L^2(\mathbb{R}^n;|y|^{-1}v^{p_1^*-2})$ is bounded, compact, and self-adjoint.
			Thanks to \eqref{eq:pti}, the existence and uniqueness of solutions to $\mathcal{L}_v[u]=f$ for $f\in L^2(\mathbb{R}^n;|y|^{-1}v^{p_1^*-2})$ follow from the standard method, so the operator $\mathcal{L}_v^{-1}$ is well-defined.
			Since $\mathcal{L}_v$ is self-adjoint, we infer from $\bigl(\mathcal{L}_v^{-1}\bigr)^*=\bigl(\mathcal{L}_v^*\bigr)^{-1}=\mathcal{L}_v^{-1}$ that $\mathcal{L}_v^{-1}$ is also self-adjoint.
			
			From \eqref{eq:pti} and H$\rm{\ddot{o}}$lder's inequality, we have
			\begin{equation*}
				c\|u\|_{D^{1,2}(\mathbb{R}^n;|Dv|^{p-2})}^2={} c\int_{\mathbb{R}^n}|Dv|^{p-2}|Du|^2\,\de x\leq{}\langle\mathcal{L}_vu,u\rangle_*\leq{}\|u\|_{D^{1,2}(\mathbb{R}^n;|Dv|^{p-2})}\|\mathcal{L}_vu\|_{L^2(\mathbb{R}^n;|y|^{-1}v^{p_1^*-2})}.
			\end{equation*}
			This proves that $\mathcal{L}_v^{-1}$ is bounded from $L^2(\mathbb{R}^n;|y|^{-1}v^{p_1^*-2})$ to $D^{1,2}(\mathbb{R}^n;|Dv|^{p-2})$, and by Proposition \ref{prop:ce}, we see that $\mathcal{L}_v^{-1}$ is a compact operator.
			Then, one can apply the spectral theorem to deduce that $\mathcal{L}_v^{-1}$ has a discrete spectrum, hence so does $\mathcal{L}_v$.
		\end{proof}
		
		\begin{proof}[Proof of Theorem \ref{thm:ef}]
			Since a scaling argument shows that the eigenvalues of $\mathcal{L}_v$ are invariant under changes of $\lambda$ and $z'$, it suffices to consider the operator $\mathcal{L}_v$ with $v=v_{a,1,0}$.
			Without loss of generality, we can prove under the assumption that $\|v\|_{L^{p_1^*}(\mathbb{R}^n;|y|^{-1})}=1$.
			
			\medskip	
			\emph{(1) First eigenspace.}	
			Combining the assumption $\|v\|_{L^{p_1^*}(\mathbb{R}^n;|y|^{-1})}=1$ with \eqref{eq:ele}, we know that $v$ satisfies
			\begin{equation}\label{eq:ef1}
				-\Delta_p v = S^p|y|^{-1}v^{p_1^*-1}.
			\end{equation}
			Hence, taking derivative in \eqref{eq:ef1} with respect to $\lambda$ or $z'_i$ $(1\leq i\leq n-k)$, one has
			\begin{equation}\label{eq:ef2}
				\mathcal{L}_v w=(p_1^*-1)S^p|y|^{-1}v^{p_1^*-2}w,\quad\quad \forall\,w\in\mathrm{span}\{\partial_\lambda v, \partial_{z'_1}v,\ldots,\partial_{z'_{n-k}}v\}.
			\end{equation}
			From \eqref{eq:ef1} and \eqref{eq:ef2}, one can easily verify that $v$ is an eigenfunction of $\mathcal{L}_v$ associated with eigenvalue $(p-1)S^p$ and that $\partial_\lambda v$ and $\partial_{z'_i}v$ are eigenfunctions associated with eigenvalue $(p_1^*-1)S^p$.
			Furthermore, since $v>0$, it follows that $\alpha_1=(p-1)S^p$ is the first eigenvalue, which is simple, so \eqref{eq:e1} holds.
			
			\medskip	
			\emph{(2) Simplification of the linearized operator.}
			To prove \eqref{eq:e2}, we need a simplified form of $\mathcal{L}_v$.
			For $\varphi\in C^{2}(\mathbb{R}^n)$, a straightforward calculation shows that
			\begin{equation}\label{eq:div}
				\begin{split}
					\mathrm{div}(|Dv|^{p-2}D\varphi)+{}&(p-2)\mathrm{div}(|Dv|^{p-4}(Dv\cdot D\varPhi)Dv)\\
					={}&|Dv|^{p-2}\Delta\varphi + D(|Dv|^{p-2})\cdot D\varphi + (p-2)|Dv|^{p-4}(Dv\cdot D\varphi)\Delta v\\
					& + (p-2)(Dv\cdot D\varphi)(D(|Dv|^{p-4})\cdot Dv) + (p-2)|Dv|^{p-4}(D(Dv\cdot D\varphi)\cdot Dv)\\
					={}&|Dv|^{p-2}\Delta\varphi + (p-2)(p-4)|Dv|^{p-6}(Dv\cdot D\varphi)(Dv\cdot D^2v\cdot Dv)\\
					& + (p-2)|Dv|^{p-4}\bigl[(Dv\cdot D\varphi)\Delta v + 2(Dv\cdot D^2v\cdot D\varphi) + Dv\cdot D^2\varphi\cdot Dv\bigr].
				\end{split}
			\end{equation}
			For simplicity, we omit the constant $a$ in the expression of $v$, that is,
			\begin{equation*}
				v=v_{a,1,0}(x)=\bigl[(1+|y|)^2+|z|^2\bigr]^{-\frac{n-p}{2(p-1)}}.
			\end{equation*}
			Define $W(x):=\bigl[(1+|y|)^2+|z|^2\bigr]$.
			Based on the above equation one can get
			\begin{equation}\label{eq:dv}
				Dv=-\frac{n-p}{p-1}W^{-\frac{n+p-2}{2(p-1)}}\cdot\tilde{x}\quad\quad\text{with}\quad\quad\tilde{x}:=\biggl((1+|y|)\frac{y}{|y|},z\biggr),
			\end{equation}
			and
			\begin{equation*}
				\Delta v=-\frac{(p-2)(n-p)(n-1)}{(p-1)^2}W^{-\frac{n+p-2}{2(p-1)}} - \frac{(n-p)(k-1)}{p-1}W^{-\frac{n+p-2}{2(p-1)}}\frac{1}{|y|}.
			\end{equation*}
			By \eqref{eq:dv}, we obtain immediately that
			\begin{equation*}
				\begin{split}
					Dv\cdot D^2\varphi\cdot Dv={}&\frac{(n-p)^2}{(p-1)^2}W^{-\frac{n+p-2}{p-1}}\sum_{1\leq i,j\leq n}\biggl(1+\frac{\delta_{i\leq k}}{|y|}\biggr)x_i\biggl(1+\frac{\delta_{j\le k}}{|y|}\biggr)x_j\frac{\partial^2\varphi}{\partial x_i\partial x_j}\\
					={}&\frac{(n-p)^2}{(p-1)^2}W^{-\frac{n+p-2}{p-1}}\sum_{1\leq i,j\leq n}\tilde{x}_i\tilde{x}_j\frac{\partial^2\varphi}{\partial x_i\partial x_j},
				\end{split}
			\end{equation*}
			where $\tilde{x}_i=\bigl(1+\delta_{i\leq k}/|y|\bigr)x_i$ is the $i$-th component of $\tilde{x}$ and
			\begin{equation*}
				\delta_{i\leq k}=
				\left\{\begin{aligned}
					1,\quad\mathrm{if}\:i\leq k,\\
					0,\quad\mathrm{if}\:i> k.
				\end{aligned}
				\right.
			\end{equation*}
			It also holds
			\begin{equation*}
				\begin{split}
					\frac{\partial^2 v}{\partial x_i\partial x_j}={}&\frac{(n-p)(n+p-2)}{(p-2)^2}W^{-\frac{n+3p-4}{2(p-1)}}\biggl(1+\frac{\delta_{i\leq k}}{|y|}\biggr)\biggl(1+\frac{\delta_{j\leq k}}{|y|}\biggr)x_ix_j\\
					&+\frac{n-p}{p-1}W^{-\frac{n+p-2}{2(p-1)}}\frac{\delta_{i\leq k}\delta_{i\leq k}}{|y|^3}x_ix_j - \frac{n-p}{p-1}W^{-\frac{n+p-2}{2(p-1)}}\biggl(1+\frac{\delta_{i\leq k}}{|y|}\biggr)\delta_{ij},
				\end{split}
			\end{equation*}
			where
			\begin{equation*}
				\delta_{ij}=
				\left\{\begin{aligned}
					1,\quad\mathrm{if}\:i=j;\\
					0,\quad\mathrm{if}\:i\not=j.
				\end{aligned}
				\right.
			\end{equation*}
			From this, one can easily deduce that
			\begin{equation*}
				\sum_{j=1}^{n}\frac{\partial v}{\partial x_j}\frac{\partial^2 v}{\partial x_i\partial x_j}=-\frac{(n-p)^2(n-1)}{(p-1)^3}W^{-\frac{n+p-2}{p-1}}\biggl(1+\frac{\delta_{i\leq k}}{|y|}\biggr)x_i.
			\end{equation*}
			Thus we have
			\begin{equation*}
				Dv\cdot D^2v\cdot Dv=\frac{(n-p)^3(n-1)}{(p-1)^4}W^{-\frac{3n+p-4}{2(p-1)}},
			\end{equation*}
			and
			\begin{equation*}
				Dv\cdot D^2v\cdot D\varphi=-\frac{(n-p)^2(n-1)}{(p-1)^3}W^{-\frac{n+p-2}{p-1}}\tilde{x}\cdot D\varphi.
			\end{equation*}
			Substituting all the above equations into \eqref{eq:div} and set $C_{n,p}:=\big(\frac{n-p}{p-1}\big)^{p-2}$, we obtain
			\begin{equation}\label{eq:Lv}
				\begin{split}
					\mathcal{L}_v\varphi={}&-C_{n,p}W^{-\frac{(n-1)(p-2)}{2(p-1)}}\Delta\varphi \\
					&-(p-2)(k-1)C_{n,p}W^{-\frac{n(p-2)+p}{2(p-1)}}\frac{1}{|y|}\tilde{x}\cdot D\varphi\\
					&-(p-2)C_{n,p}W^{-\frac{n(p-2)+p}{2(p-1)}}\sum_{1\leq i,j\leq n}\tilde{x}_i\tilde{x}_j\frac{\partial^2 \varphi}{\partial x_i\partial x_j}.
				\end{split}
			\end{equation}
			
			\medskip
			\emph{(3) Second eigenvalue.}	
			The Rayleigh quotient characterization of eigenvalues implies that
			\begin{equation*}
				\alpha_2=\inf\biggl\{\frac{\langle\mathcal{L}_v w,w \rangle}{\langle w,w \rangle_*}=\frac{\int_{\mathbb{R}^n}\mathcal{L}_vw\cdot w\,\de x}{\int_{\mathbb{R}^n}|y|^{-1}v^{p_1^*-2}w^2\,\de x}:\:w\in L^2(\mathbb{R}^n;|y|^{-1}v^{p_1^*-2})\:\mathrm{and}\:w\perp v\biggr\},
			\end{equation*}
			where $w\perp v$ means $\langle w,v\rangle_*=0$.
			For any $\varphi\in L^{2}(\mathbb{R}^n;|y|^{-1}v^{p_1^*-2})$, we have
			\begin{equation}\label{eq:lvnj}
				\begin{split}
					\langle\mathcal{L}_v\varphi,\varphi\rangle={}&-C_{n,p}\int_{\mathbb{R}^n}W^{-\frac{(n-1)(p-2)}{2(p-1)}}\Delta\varphi\cdot\varphi\,\de x\\
					&-(p-2)(k-1)C_{n,p}\int_{\mathbb{R}^n}W^{-\frac{n(p-2)+p}{2(p-1)}}\frac{1}{|y|}(\tilde{x}\cdot\varphi)\varphi\,\de x\\
					&-(p-2)C_{n,p}\int_{\mathbb{R}^n}W^{-\frac{n(p-2)+p}{2(p-1)}}(\tilde{x}\cdot D^2\varphi\cdot\tilde{x})\varphi\,\de x=:I_1+I_2+I_3.
				\end{split}
			\end{equation}
			Integrating by parts, we get
			\begin{equation}\label{eq:i1}
				\begin{split}
					I_1={}&C_{n,p}\int_{\mathbb{R}^n}W^{-\frac{(n-1)(p-2)}{2(p-1)}}|D\varphi|^2\,\de x\\
					&-\frac{(n-1)(p-2)}{2(p-1)}C_{n,p}\int_{\mathbb{R}^n}W^{-\frac{n(p-2)+p}{2(p-1)}}(DW(x)\cdot D\varphi)D\varphi\,\de x\\
					={}&\int_{\mathbb{R}^n}|Dv|^{p-2}|D\varphi|^2\,\de x-\frac{(n-1)(p-2)}{p-1}C_{n,p}\int_{\mathbb{R}^n}W^{-\frac{n(p-2)+p}{2(p-1)}}(\tilde{x}\cdot D\varphi)D\varphi\,\de x,
				\end{split}
			\end{equation}
			where we have used $DW(x)=2\tilde{x}$.
			If $\varphi\perp v$, it is important to notice that \eqref{eq:fc} can be improved as follows:
			\begin{equation}\label{eq:ifc}
				\begin{split}
					0\leq{}&\frac{\mathrm{d}^2}{\mathrm{d}\varepsilon^2}\bigg{|}_{\varepsilon=0}G(v+\varepsilon\varphi)\\
					={}&p\int_{\mathbb{R}^n}|Dv|^{p-2}|D\varphi|^2\,\de x + p(p-2)\int_{\mathbb{R}^n}|Dv|^{p-4}(Dv\cdot D\varphi)^2\,\de x\\
					&-p(p_1^*-1)S^p\biggl(\int_{\mathbb{R}^n}|y|^{-1}|v|^{p_1^*}\,\de x\biggr)^{\frac{p}{p_1^*}-1}\int_{\mathbb{R}^n}|y|^{-1}|v|^{p_1^*-2}\varphi^2 \,\de x.
				\end{split}
			\end{equation}
			Since we have assumed that $\|v\|_{L^{p_1^*}(\mathbb{R}^n;|y|^{-1})}=1$, \eqref{eq:ifc} implies that
			\begin{equation}\label{eq:ee}
				\begin{split}
					\int_{\mathbb{R}^n}|Dv|^{p-2}|D\varphi|^2\,\de x \geq{}(p_1^*-1)S^p\int_{\mathbb{R}^n}|y|^{-1}|v|^{p_1^*-2}&\varphi^2 \,\de x\\
					&-(p-2)\int_{\mathbb{R}^n}|Dv|^{p-4}(Dv\cdot D\varphi)^2\,\de x.
				\end{split}
			\end{equation}
			Combining \eqref{eq:i1} with \eqref{eq:ee}, one has
			\begin{equation}\label{eq:i11}
				\begin{split}
					I_1\geq{}&(p_1^*-1)S^p\int_{\mathbb{R}^n}|y|^{-1}|v|^{p_1^*-2}\varphi^2 \,\de x-(p-2)\int_{\mathbb{R}^n}|Dv|^{p-4}(Dv\cdot D\varphi)^2\,\de x\\
					&-\frac{(n-1)(p-2)}{p-1}C_{n,p}\int_{\mathbb{R}^n}W^{-\frac{n(p-2)+p}{2(p-1)}}(\tilde{x}\cdot D\varphi)D\varphi\,\de x.
				\end{split}
			\end{equation}
			As for $I_3$, we can also integrate by parts to get
			\begin{equation*}
				\begin{split}
					I_3={}&-(p-2)C_{n,p}\int_{\mathbb{R}^n}W^{-\frac{n(p-2)+p}{2(p-1)}}\sum_{i=1}^n\biggl(\tilde{x}_iD\biggl(\frac{\partial\varphi}{\partial x_i}\biggr)\cdot\tilde{x}\biggr)\varphi\,\de x\\
					={}&-\frac{(p-2)(n(p-2)+p)}{2(p-1)}C_{n,p}\int_{\mathbb{R}^n}W^{-\frac{n(p-2)+p}{2(p-1)}-1}(DW(x)\cdot\tilde{x})(\tilde{x}\cdot D\varphi)\varphi\,\de x\\
					&+(p-2)C_{n,p}\int_{\mathbb{R}^n}W^{-\frac{n(p-2)+p}{2(p-1)}}\sum_{i=1}^n\biggl((D\tilde{x}_i\cdot\tilde{x})\frac{\partial\varphi}{\partial x_i}\biggr)\varphi\,\de x\\
					&+(p-2)C_{n,p}\int_{\mathbb{R}^n}W^{-\frac{n(p-2)+p}{2(p-1)}}(\tilde{x}\cdot D\varphi)(D\cdot\tilde{x})\varphi\,\de x+(p-2)C_{n,p}\int_{\mathbb{R}^n}W^{-\frac{n(p-2)+p}{2(p-1)}}(\tilde{x}\cdot D\varphi)^2\,\de x.
				\end{split}
			\end{equation*}
			Noticing the facts
			\begin{equation*}
				DW(x)\cdot\tilde{x}=2W(x),\quad D\tilde{x}_i\cdot\tilde{x}=\tilde{x}_i\quad\mathrm{and}\quad D\cdot\tilde{x}=n+\frac{k-1}{|y|},
			\end{equation*}
			we have
			\begin{equation}\label{eq:i3}
				\begin{split}
					I_3={}&\frac{(p-2)(n-1)}{p-1}C_{n,p}\int_{\mathbb{R}^n}W^{-\frac{n(p-2)+p}{2(p-1)}}(\tilde{x}\cdot D\varphi)\varphi\,\de x\\
					&+(p-2)(k-1)C_{n,p}\int_{\mathbb{R}^n}W^{-\frac{n(p-2)+p}{2(p-1)}}\frac{1}{|y|}(\tilde{x}\cdot D\varphi)\varphi\,\de x\\
					&+(p-2)C_{n,p}\int_{\mathbb{R}^n}W^{-\frac{n(p-2)+p}{2(p-1)}}(\tilde{x}\cdot D\varphi)^2\,\de x.
				\end{split}
			\end{equation}
			By \eqref{eq:lvnj}, \eqref{eq:i11} and \eqref{eq:i3}, we can deduce that
			\begin{equation*}
				\begin{split}
					\langle\mathcal{L}_v\varphi,\varphi\rangle\geq{}&(p_1^*-1)S^p\int_{\mathbb{R}^n}|y|^{-1}|v|^{p_1^*-2}\varphi^2 \,\de x-(p-2)\int_{\mathbb{R}^n}|Dv|^{p-4}(Dv\cdot D\varphi)^2\,\de x\\
					&+(p-2)C_{n,p}\int_{\mathbb{R}^n}W^{-\frac{n(p-2)+p}{2(p-1)}}(\tilde{x}\cdot D\varphi)^2\,\de x.
				\end{split}
			\end{equation*}
			Thus, according to the fact that
			\begin{equation*}
				\int_{\mathbb{R}^n}|Dv|^{p-4}(Dv\cdot D\varphi)^2\,\de x=C_{n,p}\int_{\mathbb{R}^n}W^{-\frac{n(p-2)+p}{2(p-1)}}(\tilde{x}\cdot D\varphi)^2\,\de x,
			\end{equation*}
			one can get
			\begin{equation*}
				\langle\mathcal{L}_v\varphi,\varphi\rangle\geq(p_1^*-1)S^p\int_{\mathbb{R}^n}|y|^{-1}|v|^{p_1^*-2}\varphi^2 \,\de x=(p_1^*-1)S^p\langle\varphi,\varphi\rangle_*,\quad\quad\forall\,\varphi\perp v.
			\end{equation*}
			Combining this with \eqref{eq:ef2} imply that $(p_1^*-1)S^p$ is the second eigenvalue of $\mathcal{L}_v$ by the Rayleigh quotient characterization of eigenvalues.
			
			\medskip
			\emph{(4) Second eigenspace.}
			Next, we will prove that the second eigenspace $E_2=\mathit{span}\{\partial_\lambda v, \partial_{z'_1}v,\ldots,\partial_{z'_{n-k}}v\}$.
			If $\varphi(x)\in E_2$, it follows from \eqref{eq:Lv} that
			\begin{equation}\label{eq:phi}
				\begin{split}
					\mathcal{L}_v(\varphi(x))={}&-C_{n,p}W^{-\frac{(n-1)(p-2)}{2(p-1)}}\Delta (\varphi(x)) \\
					&-(p-2)(k-1)C_{n,p}W^{-\frac{n(p-2)+p}{2(p-1)}}\frac{1}{|y|}\tilde{x}\cdot D\varphi(x)\\
					&-(p-2)C_{n,p}W^{-\frac{n(p-2)+p}{2(p-1)}}\sum_{1\leq i,j\leq n}\tilde{x}_i\tilde{x}_j\frac{\partial^2 \varphi(x)}{\partial x_i\partial x_j}\\
					={}&(p_1^*-1)S^p|y|^{-1}v^{p_1^*-2}\varphi(x).
				\end{split}
			\end{equation}
			Since $v$ is cylindrically symmetric, we apply the spherical harmonic decomposition to \eqref{eq:phi} with respect to partial variable $y\in \mathbb{R}^{k}$ and then apply the spherical harmonic decomposition twice again with respect to partial variable $z\in\mathbb{R}^{n-k}$. Namely, let $r_1=|y|$, $r_2=|z|$, $\theta\in\mathbb{S}^{k-1}$ and $\eta\in\mathbb{S}^{n-k-1}$, then we can decompose
			\begin{equation}\label{eq:dwd}
				\varphi(x)=\sum_{i=0}^{\infty}\varphi_i(r_1,z)Y_i(\theta)=\sum_{i=0}^{\infty}\sum_{j=0}^{\infty}\varphi_{i,j}(r_1,r_2)Y_i(\theta)Z_j(\eta),
			\end{equation}
			where
			\begin{equation*}
				\varphi_{i,j}(r_1,r_2)=\int_{\mathbb{S}^{k-1}}\int_{\mathbb{S}^{n-k-1}}\varphi(r_1,\theta,r_2,\eta)\,\de\theta\de\eta.
			\end{equation*}
			Here $Y_i(\theta)$ or $Z_j(\eta)$ denotes the $i$-th or $j$-th spherical harmonic function satisfying
			\begin{equation}\label{eq:lbo}
				-\Delta_{\mathbb{S}^{k-1}}Y_i(\theta)=\lambda_iY_i(\theta)\quad\quad\mathrm{and}\quad\quad-\Delta_{\mathbb{S}^{n-k-1}}Z_j(\eta)=\mu_jZ_j(\eta),
			\end{equation}
			where $\Delta_{\mathbb{S}^{k-1}}$ and $\Delta_{\mathbb{S}^{n-k-1}}$ are the Laplace-Beltrami operator on
			$\mathbb{S}^{k-1}$ and $\mathbb{S}^{n-k-1}$, respectively.
			It is well known that
			\begin{equation*}
				\begin{aligned}
					&\lambda_i=i(k+i-2)\mathrm{,}\quad &i\in\mathbb{N},\\
					&\mu_j=j(n-k+j-2)\mathrm{,}\quad &j\in\mathbb{N},
				\end{aligned}
			\end{equation*}
			whose multiplicity are
			\begin{equation*}
				\frac{(k+2i-2)(k+i-3)!}{(k-2)!k!}\quad\quad\mathrm{and}\quad\quad\frac{(n-k+2j-2)(n-k+j-3)!}{(n-k-2)!(n-k)!},
			\end{equation*}
			respectively.
			
		We aim to deriving the equations satisfied by the cylindrical symmetrical functions $\varphi_{i,j}$.
		By direct calculations, we can obtain that
		\begin{equation}\label{eq:la}
			\begin{aligned}
				\Delta\bigl(\varphi_{i,j}(r_1,r_2)Y_i(\theta)Z_j(\eta)\bigr)={}&\sum_{l=1}^{k}\partial_{y_ly_l}^2\bigl(\varphi_{i,j}(r_1,r_2)Y_i(\theta)Z_j(\eta)\bigr)+\sum_{s=1}^{n-k}\partial_{z_sz_s}^2\bigl(\varphi_{i,j}(r_1,r_2)Y_i(\theta)Z_j(\eta)\bigr)\\
				={}&\Bigl(\partial_{r_1r_1}\varphi_{i,j}+\frac{k-1}{r_1}\partial_{r_1}\varphi_{i,j}\Bigr)Y_iZ_j+\frac{1}{r_1^2}\varphi_{i,j}(\Delta_{\mathbb{S}^{k-1}}Y_i)Z_j\\
				&+\Bigl(\partial_{r_2r_2}\varphi_{i,j}+\frac{n-k-1}{r_2}\partial_{r_2}\varphi_{i,j}\Bigr)Y_iZ_j+\frac{1}{r_2^2}\varphi_{i,j}Y_i\Delta_{\mathbb{S}^{n-k-1}}Z_j.
			\end{aligned}
		\end{equation}
	Now, we compute the other terms in \eqref{eq:phi}.
	It is easy to verify that
	\begin{equation*}
		\partial_{y_l}\bigl(\varphi_{i,j}Y_iZ_j\bigr)=\partial_{r_1}\varphi_{i,j}\frac{y_l}{r_1}Y_iZ_j+\varphi_{i,j}\sum_{h}\frac{\partial Y_i}{\partial\theta_h}\frac{\partial\theta_h}{\partial y_l}Z_j,
	\end{equation*}
	and
	\begin{equation*}
		\partial_{z_l}\bigl(\varphi_{i,j}Y_iZ_j\bigr)=\partial_{r_2}\varphi_{i,j}\frac{z_l}{r_2}Y_iZ_j+\varphi_{i,j}Y_i\sum_{h}\frac{\partial Z_j}{\partial\eta_h}\frac{\partial\eta_h}{\partial z_l}.
	\end{equation*}		
	Then it holds
	\begin{equation*}
		\begin{split}
			\partial_{y_ly_s}^2\bigl(\varphi_{i,j}Y_iZ_j\bigr)={}&\partial_{r_1r_1}\varphi_{i,j}\frac{y_ly_s}{r_1^2}Y_iZ_j+\partial_{r_1}\varphi_{i,j}\biggl(\frac{\delta_{ls}}{r_1}-\frac{y_ly_s}{r_1^3}\biggr)Y_iZ_j\\
			&+\partial_{r_1}\varphi_{i,j}\frac{y_l}{r_1}\sum_{h}\frac{\partial Y_i}{\partial\theta_h}\frac{\partial\theta_h}{\partial y_s}Z_j+\partial_{r_1}\varphi_{i,j}\frac{y_s}{r_1}\sum_{h}\frac{\partial Y_i}{\partial\theta_h}\frac{\partial\theta_h}{\partial y_l}Z_j\\
			&+\varphi_{i,j}\sum_{h}\sum_{t}\frac{\partial^2Y_i}{\partial\theta_t\partial\theta_h}\frac{\partial\theta_t}{\partial y_s}\frac{\partial\theta_h}{\partial y_l}Z_j+\varphi_{i,j}\sum_{h}\frac{\partial Y_i}{\partial\theta_h}\frac{\partial^2\theta_h}{\partial y_s\partial y_l}Z_j,
		\end{split}
	\end{equation*}		
			
		\begin{equation*}
		\begin{split}
			\partial_{z_lz_s}^2\bigl(\varphi_{i,j}Y_iZ_j\bigr)={}&\partial_{r_2r_2}\varphi_{i,j}\frac{z_lz_s}{r_1^2}Y_iZ_j+\partial_{r_2}\varphi_{i,j}\biggl(\frac{\delta_{ls}}{r_2}-\frac{z_lz_s}{r_2^3}\biggr)Y_iZ_j\\
			&+\partial_{r_2}\varphi_{i,j}\frac{z_l}{r_2}Y_i\sum_{h}\frac{\partial Z_j}{\partial\eta_h}\frac{\partial\eta_h}{\partial z_s}+\partial_{r_2}\varphi_{i,j}\frac{z_s}{r_2}Y_i\sum_{h}\frac{\partial Z_j}{\partial\eta_h}\frac{\partial\eta_h}{\partial z_l}\\
			&+\varphi_{i,j}Y_i\sum_{h}\sum_{t}\frac{\partial^2Z_j}{\partial\eta_t\partial\eta_h}\frac{\partial\eta_t}{\partial z_s}\frac{\partial\eta_h}{\partial z_l}+\varphi_{i,j}Y_i\sum_{h}\frac{\partial Z_j}{\partial\eta_h}\frac{\partial^2\eta_h}{\partial z_s\partial z_l},
		\end{split}
	\end{equation*}			
	and
	\begin{equation*}
		\begin{split}
			\partial_{z_ly_s}^2\bigl(\varphi_{i,j}Y_iZ_j\bigr)={}&\partial_{r_1r_2}\varphi_{i,j}\frac{z_ly_s}{r_1r_2}Y_iZ_j+\partial_{r_1}\varphi_{i,j}\frac{y_s}{r_1}Y_i\sum_{h}\frac{\partial Z_j}{\partial\eta_h}\frac{\partial\eta_h}{\partial z_l}\\
			&+\partial_{r_2}\varphi_{i,j}\frac{z_l}{r_2}\sum_{h}\frac{\partial Y_i}{\partial\theta_h}\frac{\partial\theta_h}{\partial y_s}Z_j+\varphi_{i,j}\sum_{h}\sum_{t}\frac{\partial Y_i}{\partial\theta_h}\frac{\partial\theta_h}{\partial y_s}\frac{\partial Z_j}{\partial\eta_t}\frac{\partial\eta_t}{\partial z_l}.
		\end{split}
	\end{equation*}		
	Hence we have
	\begin{equation}\label{eq:xd}
		\tilde{x}\cdot D\bigl(\varphi_{i,j}Y_iZ_j\bigr)=(1+r_1)\partial_{r_1}\varphi_{i,j}Y_iZ_j+r_2\partial_{r_2}\varphi_{i,j}Y_iZ_j,
	\end{equation}
	and
	\begin{equation}\label{eq:xdx}
		\begin{split}
			\sum_{1\leq a,b\leq n}\tilde{x}_a\tilde{x}_b\frac{\partial^2 \bigl(\varphi_{i,j}Y_iZ_j\bigr)}{\partial x_a\partial x_b}={}&(1+r_1)^2\partial_{r_1r_1}\varphi_{i,j}Y_iZ_j+r_2^2\partial_{r_2r_2}\varphi_{i,j}Y_iZ_j\\
			&+2r_2(1+r_1)\partial_{r_1r_2}\varphi_{i,j}Y_iZ_j,
		\end{split}
	\end{equation}
where we have used the facts that
	\begin{equation*}
		\sum_{l=1}^{k}\frac{\partial\theta_h}{\partial y_l}y_l=0,\quad\quad\sum_{l,s=1}^{k}\frac{\partial^2\theta_h}{\partial y_l\partial y_s}y_ly_s=0,
	\end{equation*}
	and
	\begin{equation*}
		\sum_{l=1}^{n-k}\frac{\partial\eta_h}{\partial z_l}z_l=0,\quad\quad\sum_{l,s=1}^{n-k}\frac{\partial^2\eta_h}{\partial z_l\partial z_s}z_lz_s=0.
	\end{equation*}	
	Putting \eqref{eq:lbo}, \eqref{eq:la}, \eqref{eq:xd} and \eqref{eq:xdx} into \eqref{eq:phi}, we get the following pointwise equations for all $\varphi_{i,j}$ ($i,j\in\mathbb{N}$) ,
	\begin{equation}\label{eq:ij1}
		\begin{split}
			&C_{n,p}W^{\alpha+1}\biggl(\partial_{r_1r_1}\varphi_{i,j}+\frac{k-1}{r_1}\partial_{r_1}\varphi_{i,j}-\frac{\lambda_i}{r_1^2}\varphi_{i,j} + \partial_{r_2r_2}\varphi_{i,j}+\frac{n-k-1}{r_2}\partial_{r_2}\varphi_{i,j}-\frac{\mu_j}{r_2^2}\varphi_{i,j}\biggr)\\
			&+C_{n,p}(p-2)W^{\alpha}\big[(1+r_1)^2\partial_{r_1r_1}\varphi_{i,j}+r_2^2\partial_{r_2r_2}\varphi_{i,j}+2(1+r_1)r_2\partial_{r_1r_2}\varphi_{i,j}\big]\\
			&+C_{n,p}(p-2)(k-1)r_1^{-1}W^{\alpha}\big[(1+r_1)\partial_{r_1}\varphi_{i,j}+r_2\partial_{r_2}\varphi_{i,j}\big]\\
			&+(p_1^*-1)S^pr_1^{-1}W^{\alpha}\varphi_{i,j}=0,
		\end{split}
	\end{equation}
	where $\alpha=-\frac{n(p-2)+p}{2(p-1)}$.
	
	In what follows, we let $r:=(r_1,r_2)$ and $D:=(\partial_{r_1},\partial_{r_2})$ when dealing with binary function on $(r_1,r_2)$.
	Defining $R(r):=W^{\frac{1}{2}}=\big[(1+r_1)^2+r_2^2\big]^{\frac{1}{2}}$ and multiplying \eqref{eq:ij1} by $r_1^{k-1}r_2^{n-k-1}R^{\frac{p-2}{p-1}}$, we can write \eqref{eq:ij1} in divergence form as follows:
	\begin{equation}\label{eq:ij2}
		\begin{split}
			\mathcal{L}_{i,j}\varphi_{i,j}:={}&C_{n,p}\mathrm{div}\big(r_1^{k-1}r_2^{n-k-1}R^{2\beta+2}D\varphi_{i,j}+(p-2)r_1^{k-1}r_2^{n-k-1}R^{2\beta+2}(DR\cdot D\varphi_{i,j})DR\big)\\
			&-C_{n,p}\big(r_1^{k-3}r_2^{n-k-1}R^{2\beta+2}\lambda_i\varphi_{i,j}+r_1^{k-1}r_2^{n-k-3}R^{2\beta+2}\mu_j\varphi_{i,j}\big)\\
			&+(p_1^*-1)S^pr_1^{k-2}r_2^{n-k-1}R^{2\beta}\varphi_{i,j}=0,
		\end{split}
	\end{equation}
	where $\beta=\alpha+\frac{p-2}{2(p-1)}=-\frac{n(p-2)+2}{2(p-1)}$.
	
	Since we aim to proving the non-degeneracy of the linear equation \eqref{eq:phi}, we will look for solutions to \eqref{eq:ij2}.

	\textit{(a) The case $i=j=0$}.
	We know that the following function
	\begin{equation*}
		\psi_{0,0}:=\partial_{\lambda}v\big|_{\lambda=1,z'=0}=\frac{n-p}{p}W^{-\frac{n-p}{2(p-1)}}-\frac{n-p}{p-1}W^{-\frac{n-p}{2(p-1)}-1}\big[(1+r_1)r_1+r_2^2\big]
	\end{equation*}
	solves the equation
	\begin{equation*}
		\mathcal{L}_{0,0}\psi_{0,0}=0.
	\end{equation*}
	If there is a binary function $f\in C^{2}(\mathbb{R}^2)$ satisfying $\mathcal{L}_{0,0}f=0$, we will show that $f=c\psi_{0,0}$ for some constant $c$.
	Let $f=g\psi_{0,0}$ with $g=\frac{f}{\psi_{0,0}}$.
	Multiplying $\mathcal{L}_{0,0}f=0$ by $f$ and integrating on $\mathbb{R}^2$, we can obtain
	\begin{equation}\label{eq:00}
		\begin{split}
			&-C_{n,p}\int_{\mathbb{R}^2}r_1^{k-1}r_2^{n-k-1}R^{2\beta+2}|Df|^2+(p-2)r_1^{k-1}r_2^{n-k-1}R^{2\beta+2}(DR\cdot Df)^2\,\de r\\
			&-C_{n,p}\int_{\mathbb{R}^2}r_1^{k-3}r_2^{n-k-1}R^{2\beta+2}\lambda_0 f^2\,\de r-C_{n,p}\int_{\mathbb{R}^2}r_1^{k-1}r_2^{n-k-3}R^{2\beta+2}\mu_0 f^2\,\de r\\
			&+(p_1^*-1)S^p\int_{\mathbb{R}^2}r_1^{k-2}r_2^{n-k-1}R^{2\beta}f^2=0.
		\end{split}
	\end{equation}
	Recall that $f=g\psi_{0,0}$, the first term in the left-hand side of \eqref{eq:00} equals to
	\begin{equation}\label{eq:001}
		\begin{split}
			&-C_{n,p}\bigg(\int_{\mathbb{R}^2}g^2r_1^{k-1}r_2^{n-k-1}R^{2\beta+2}|D\psi_{0,0}|^2+(p-2)g^2r_1^{k-1}r_2^{n-k-1}R^{2\beta+2}(DR\cdot D\psi_{0,0})^2\,\de r\\
			&+\int_{\mathbb{R}^2}\psi_{0,0}^2r_1^{k-1}r_2^{n-k-1}R^{2\beta+2}|Dg|^2+(p-2)\psi_{0,0}^2r_1^{k-1}r_2^{n-k-1}R^{2\beta+2}(DR\cdot Dg)^2\,\de r\\
			&+2\int_{\mathbb{R}^2}g\psi_{0,0}r_1^{k-1}r_2^{n-k-1}R^{2\beta+2}\big(Dg\cdot D\psi_{0,0}+(p-2)(DR\cdot Dg)(DR\cdot D\psi_{0,0})\big)\,\de r\bigg)\\
			&=:-C_{n,p}(J_1+J_2+J_3).
		\end{split}
	\end{equation}
	Define $A(r):=r_1^{k-1}r_2^{n-k-1}R^{2\beta+2}D\psi_{0,0}+(p-2)r_1^{k-1}r_2^{n-k-1}R^{2\beta+2}(DR\cdot D\psi_{0,0})DR$, then we can integrate by parts and get
	\begin{equation}\label{eq:002}
		\begin{split}
			J_3=&\int_{\mathbb{R}^2}\psi_{0,0}A(r)\cdot D(g^2)\,\de r
			=-\int_{\mathbb{R}^2}g^2A(r)\cdot D\psi_{0,0}\,\de r-\int_{\mathbb{R}^2}g^2\psi_{0,0}\mathrm{div}(A(r))\,\de r\\
			=&-J_1-\int_{\mathbb{R}^2}g^2\psi_{0,0}\mathrm{div}(A(r))\,\de r.
		\end{split}
	\end{equation}
	Combining \eqref{eq:00}, \eqref{eq:001} with \eqref{eq:002}, we can obtain
	\begin{equation*}
		-C_{n,p}J_2+\int_{\mathbb{R}^2}\mathcal{L}_{0,0}(\psi_{0,0})g^2\psi_{0,0}\,\de r=0.
	\end{equation*}
	Since $\mathcal{L}_{0,0}\psi_{0,0}=0$, it holds that
	\begin{equation*}
		\begin{split}
			0=&\int_{\mathbb{R}^2}\psi_{0,0}^2r_1^{k-1}r_2^{n-k-1}R^{2\beta+2}|Dg|^2+(p-2)\psi_{0,0}^2r_1^{k-1}r_2^{n-k-1}R^{2\beta+2}(DR\cdot Dg)^2\,\de r\\
			\geq&\int_{\mathbb{R}^2}\psi_{0,0}^2r_1^{k-1}r_2^{n-k-1}R^{2\beta+2}|Dg|^2\,\de r\geq0
		\end{split}
	\end{equation*}
	provided that $p\geq2$.
	If $p<2$, we also have
	\begin{equation*}
		\begin{split}
			0=&\int_{\mathbb{R}^2}\psi_{0,0}^2r_1^{k-1}r_2^{n-k-1}R^{2\beta+2}|Dg|^2+(p-2)\psi_{0,0}^2r_1^{k-1}r_2^{n-k-1}R^{2\beta+2}(DR\cdot Dg)^2\,\de r\\
			\geq&(p-1)\int_{\mathbb{R}^2}\psi_{0,0}^2r_1^{k-1}r_2^{n-k-1}R^{2\beta+2}|Dg|^2\,\de r\geq 0,
		\end{split}
	\end{equation*}
	since $|DR|=1$.
	Thus we have
	\begin{equation*}
		\int_{\mathbb{R}^2}\psi_{0,0}^2r_1^{k-1}r_2^{n-k-1}R^{2\beta+2}|Dg|^2\,\de r=0,
	\end{equation*}
	which implies $Dg=0$ a.e. in $\mathbb{R}^2$. That is, $f=c\psi_{0,0}$ for some constant $c$.

	\textit{(b) The case $i=0$ and $j=1$}.
	We know that the functions
	\begin{equation*}
		\partial_{z_i'}v\big|_{\lambda=1,z'=0}=\frac{n-p}{p-1}W^{-\frac{n-p}{2(p-1)}-1}z_i, \quad i=1,\dots ,n-k
	\end{equation*}
	solve the equation \eqref{eq:phi}.
	We define
	\begin{equation*}
		\psi_{0,1}:=W^{-\frac{n-p}{2(p-1)}-1}|z|=[(1+r_1)^2+r_2]^{-\frac{n-p}{2(p-1)}-1}r_2,
	\end{equation*}
	then $\partial_{z_i'}v\big|_{\lambda=1,z'=0}=\frac{n-p}{p-1}\psi_{0,1}\frac{z_i}{|z|}$.
	It holds that
	\begin{equation*}
		\mathcal{L}_{0,1}\psi_{0,1}=0,
	\end{equation*}
	since $\frac{z_i}{|z|}$ $(1\leq i\leq n-k)$ are all the functions which satisfy $-\Delta_{\mathbb{S}^{n-k-1}}G(\eta)=\mu_1 G(\eta)$.
	For any function $f\in C^{2}(\mathbb{R}^2)$ that satisfy $\mathcal{L}_{0,1}f=0$, we can multiply $\mathcal{L}_{0,1}f=0$ by $f$ and integrate on $\mathbb{R}^2$. By similar way as the case $i=j=0$, we can prove $f=c\psi_{0,1}$ for some constant $c$.

	\textit{(c) The case $j\geq2$}.
	We are going to prove the solution of the equation $\mathcal{L}_{i,j}(f)=0$ must vanishes everywhere.
	Assume on the contrary that there is a nontrivial function $f\in C^2(\mathbb{R}^2)$ satisfying $\mathcal{L}_{i,j}(f)=0$.
	Without loss of generality, we can assume $f>0$ in $\Omega\subset\mathbb{R}^2$. (If $f<0$ in $\mathbb{R}^2$, replace $f$ by $-f$).
	Since $\mathcal{L}_{0,1}\psi_{0,1}=0$, one has
	\begin{equation}\label{eq:020}
		\int_{\Omega}\left[\mathcal{L}_{i,j}(f)\psi_{0,1}-\mathcal{L}_{0,1}(\psi_{0,1})f\right]\,\de r=0.
	\end{equation}
	
	Noticing that $f=0$ on $\partial\Omega$, we can integrate by parts and get
	\begin{equation}\label{eq:021}
		\begin{split}
			&\quad \int_{\Omega}\Big[\mathrm{div}(r_1^{k-1}r_2^{n-k-1}R^{2\beta+2}Df)\psi_{0,1}-\mathrm{div}(r_1^{k-1}r_2^{n-k-1}R^{2\beta+2}D\psi_{0,1})f\Big]\,\de r\\
			=&\int_{\partial\Omega}r_1^{k-1}r_2^{n-k-1}R^{2\beta+2}\psi_{0,1}Df\cdot\nu\,\de\sigma - \int_{\Omega}r_1^{k-1}r_2^{n-k-1}R^{2\beta+2}Df\cdot D\psi_{0,1}\,\de r\\
			&-\int_{\partial\Omega}r_1^{k-1}r_2^{n-k-1}R^{2\beta+2}fD\psi_{0,1}\cdot\nu \,\de\sigma + \int_{\Omega}r_1^{k-1}r_2^{n-k-1}R^{2\beta+2}Df\cdot D\psi_{0,1}\,\de r\\
			=&\int_{\partial\Omega}r_1^{k-1}r_2^{n-k-1}R^{2\beta+2}\psi_{0,1}Df\cdot\nu\,\de\sigma,
		\end{split}
	\end{equation}
	where $\nu=-\frac{Df}{|Df|}$ is the unit outer normal vector of $\partial\Omega$.
	It also holds that
	\begin{equation}\label{eq:022}
		\begin{split}
			\int_{\Omega}\Big[\mathrm{div}(r_1^{k-1}&r_2^{n-k-1}R^{2\beta+2}(DR\cdot Df)DR)\psi_{0,1}-\mathrm{div}(r_1^{k-1}r_2^{n-k-1}R^{2\beta+2}(DR\cdot D\psi_{0,1})DR)f\Big]\,\de r\\
			=&\int_{\partial\Omega}r_1^{k-1}r_2^{n-k-1}R^{2\beta+2}\big((DR\cdot Df)(DR \cdot\nu)\psi_{0,1}-(DR\cdot D\psi_{0,1})(DR \cdot\nu)f\big)\,\de\sigma\\
			&-\int_{\Omega}r_1^{k-1}r_2^{n-k-1}R^{2\beta+2}\big((DR\cdot Df)(DR \cdot D\psi_{0,1})-(DR \cdot D\psi_{0,1})(DR\cdot Df)\big)\,\de r\\
			=&\int_{\partial\Omega}r_1^{k-1}r_2^{n-k-1}R^{2\beta+2}(DR\cdot Df)(DR \cdot\nu)\psi_{0,1}\,\de\sigma.
		\end{split}
	\end{equation}
	Combining \eqref{eq:021}, \eqref{eq:022} with \eqref{eq:020}, one can obtain
	\begin{equation}\label{eq:02c}
		\begin{split}
			-&C_{n,p}\int_{\partial\Omega}r_1^{k-1}r_2^{n-k-1}R^{2\beta+2}\psi_{0,1}\bigg(|Df|+(p-2)(DR\cdot Df)\bigg(DR\cdot\frac{Df}{|Df|}\bigg)\bigg)\,\de\sigma\\
			=&(\mu_j-\mu_1)\int_{\Omega}r_1^{k-1}r_2^{n-k-3}R^{2\beta+2}f\psi_{0,1}\,\de r+(\lambda_i-\lambda_0)\int_{\Omega}r_1^{k-3}r_2^{n-k-1}R^{2\beta+2}f\psi_{0,1}\,\de r\\
			\geq&(\mu_j-\mu_1)\int_{\Omega}r_1^{k-1}r_2^{n-k-3}R^{2\beta+2}f\psi_{0,1}\,\de r>0,
		\end{split}
	\end{equation}
	since $\psi_{0,1}>0$ in $\mathbb{R}^2$.
	If $p\geq2$, we can obtain
	\begin{equation*}
			0>-C_{n,p}\int_{\partial\Omega}r_1^{k-1}r_2^{n-k-1}R^{2\beta+2}|Df|\,\de \sigma\geq(\mu_j-\mu_1)\int_{\Omega}r_1^{k-1}r_2^{n-k-3}R^{2\beta+2}f\psi_{0,1}\,\de r>0,
	\end{equation*}
	which is a contradiction.
	If $p<2$, by H$\rm{\ddot{o}}$lder's inequality, we also have that
		\begin{equation*}
			0>-C_{n,p}(p-1)\int_{\partial\Omega}r_1^{k-1}r_2^{n-k-1}R^{2\beta+2}|Df|\,\de \sigma\geq(\mu_j-\mu_1)\int_{\Omega}r_1^{k-1}r_2^{n-k-3}R^{2\beta+2}f\psi_{0,1}\,\de r>0,
	\end{equation*}
	where we have used $|DR|=1$.

	If $\Omega=\mathbb{R}^2$, then \eqref{eq:02c} becomes
	\begin{equation*}
		\begin{split}
			0=&(\mu_j-\mu_1)\int_{\Omega}r_1^{k-1}r_2^{n-k-3}R^{2\beta+2}f\psi_{0,1}\,\de r+(\lambda_i-\lambda_0)\int_{\Omega}r_1^{k-3}r_2^{n-k-1}R^{2\beta+2}f\psi_{0,1}\,\de r\\
			\geq&(\mu_j-\mu_1)\int_{\Omega}r_1^{k-1}r_2^{n-k-3}R^{2\beta+2}f\psi_{0,1}\,\de r>0,
		\end{split}
	\end{equation*}
	which is also a contradiction.

	\textit{(d) The case $i\geq1$}.
	In similar way as the case $j\geq2$, we can prove that the solution of the equation $\mathcal{L}_{i,j}(f)=0$ must vanishes everywhere.
	Assume on the contrary that there is a nontrivial function $f\in C^2(\mathbb{R}^2)$ satisfying $\mathcal{L}_{i,j}(f)=0$.
	Without loss of generality, we can assume $f>0$ in $\Omega\subset\mathbb{R}^2$. (If $f<0$ in $\mathbb{R}^2$, replace $f$ by $-f$).

	Since $\psi_{0,0}$ is not always positive on $\mathbb{R}^2$, we use $\psi_{0,0}^+$ instead of $\psi_{0,0}$ in this case.

	We can assume that $\Omega\cap{\mathrm{supp}\psi_{0,0}^+}\not=\emptyset$, where $\mathrm{supp}\psi_{0,0}^+$ denotes the support set of $\psi_{0,0}^+$.
	Indeed, if $\Omega\cap{\mathrm{supp}\psi_{0,0}^+}=\emptyset$, then $\Omega\cap{\mathrm{supp}\psi_{0,0}^-}\not=\emptyset$.
	Thus we only need to use $\psi_{0,0}^-$ instead of $\psi_{0,0}^+$.

	Since $\psi_{0,0}^+$ also satisfies $\mathcal{L}_{0,0}(\psi_{0,0}^+)=0$ a.e. in $\mathbb{R}^2$, we have
	\begin{equation}\label{eq:100}
		\int_{\Omega}\left[\mathcal{L}_{i,j}(f)\psi_{0,0}^+-\mathcal{L}_{0,0}(\psi_{0,0}^+)f\right]\,\de r=0.
	\end{equation}
	Then, by similar way as the case $j\geq2$, we obtain
		\begin{equation*}
		\begin{split}
			&-C_{n,p}\int_{\partial\Omega}r_1^{k-1}r_2^{n-k-1}R^{2\beta+2}\psi_{0,0}^+\bigg(|Df|+(p-2)(DR\cdot Df)\bigg(DR\cdot\frac{Df}{|Df|}\bigg)\bigg)\,\de \sigma\\
			=&(\mu_j-\mu_0)\int_{\Omega}r_1^{k-1}r_2^{n-k-3}R^{2\beta+2}f\psi_{0,0}^+\,\de r+(\lambda_i-\lambda_0)\int_{\Omega}r_1^{k-3}r_2^{n-k-1}R^{2\beta+2}f\psi_{0,0}^+\,\de r\\
			\geq&(\lambda_i-\lambda_0)\int_{\Omega}r_1^{k-3}r_2^{n-k-1}R^{2\beta+2}f\psi_{0,0}^+\,\de r>0,
		\end{split}
	\end{equation*}
	which leads to a contradiction due to the same reason as the case $j\geq2$. This concludes our proof of Theorem \ref{thm:ef}.
	\end{proof}

		\subsection{Poincar\'e and spectral inequalities}
		As a consequence of Theorem \ref{thm:ef}, the functions orthogonal to the tangent space $T_v\mathcal{M}$ enjoy a quantitative improvement in the Poincar\'e{}
		inequality induced by $\mathcal{L}_v$.
		More precisely, the following holds:
		\begin{proposition}\label{prop:sg}
			Given $1 < p < n$ and any extremal function $v\in \mathcal M$, there exists a constant $\lambda = \lambda(n,p,k) > 0$ such that for any function $\varphi\in L^2(\mathbb{R}^n;|y|^{-1}v^{p_1^*-2})$ orthogonal to the tangent space $T_v\mathcal{M}$,
			\begin{equation*}
				\begin{split}
					\int_{\mathbb{R}^n} |Dv|^{p-2}|D\varphi|^2 + (p-2)|Dv|^{p-4}|Dv &\cdot D\varphi|^2 \,\de x\\
					&\geq{}\bigl( (p_1^*-1)S^p + 2\lambda \bigr) \| v \|_{L^{p_1^*}(\mathbb{R}^n;|y|^{-1})}^{p-p_1^*} \int_{\mathbb{R}^n} |y|^{-1}v^{p_1^*-2}|\varphi|^2 \,\de x,
				\end{split}
			\end{equation*}
			where $S=S(n,p,k)$ is the sharp constant in the Hardy-Sobolev-Maz'ya inequalities \eqref{eq:hsm}.
		\end{proposition}
		
		Observe that $\frac{p_1^*}{p_1^*-1}=\frac{p(n-1)}{n(p-1)}$, then it is easy to verify that, $|y|^{-1}\bigl(v^{p_1^*-2}\xi\bigr)^{p_1^*/(p_1^*-1)}\in L^1(\mathbb{R}^n)$ for any $\xi\in T_v\mathcal{M}$. By H$\rm{\ddot{o}}$lder's inequality, for any $u\in D^{1,p}(\mathbb{R}^n)$ and $\xi\in T_v\mathcal{M}$, one has
		\begin{equation*}
			\int_{\mathbb{R}^n} |y|^{-1}v^{p_1^*-2} \xi u \,\de x\leq{}\biggl(\int_{\mathbb{R}^n}|y|^{-1}\bigl(v^{p_1^*-2}\xi\bigr)^{\frac{p_1^*}{p_1^*-1}}\,\de x\biggr)^{\frac{p_1^*-1}{p_1^*}} \cdot \biggl(\int_{\mathbb{R}^n}|y|^{-1}u^{p_1^*}\,\de x\biggr)^{\frac{1}{p_1^*}}<\infty.
		\end{equation*}
		Thus, we can define  the ``orthogonality to $T_v\mathcal{M}$" for functions in $D^{1,p}(\mathbb{R}^n)$ as follows.
		\begin{definition}\label{def:o}
			Let $1<p<n$. For any functions $u\in D^{1,p}(\mathbb{R}^n)$ and $v\in \mathcal M$,  we say that $u$ is orthogonal to the tangent space $T_v\mathcal{M}$ in $L^2(\mathbb{R}^n;|y|^{-1}v^{p_1^*-2})$,  if and only if
			\begin{equation*}
				\int_{\mathbb{R}^n} |y|^{-1}v^{p_1^*-2} w u \,\de x ={} 0, \quad\quad\forall\, w\in T_v\mathcal{M}.
			\end{equation*}
		\end{definition}
		
		The main result of this section is the following spectral gap estimate.
		\begin{proposition}\label{prop:sgc}
			Let $S=S(n,p,k)>0$ be the sharp constant in the Hardy-Sobolev-Maz'ya inequalities \eqref{eq:hsm}, and $\lambda=\lambda(n,p,k)>0$ be the same constant as specified in Proposition \ref{prop:sg}.
			For any $\gamma_0>0$ and $C_1>0$, there exists $\bar{\delta}=\bar{\delta}(n,p,\gamma_0,C_1)>0$ such that the following holds:
			Let $\varphi\in D^{1,p}(\mathbb{R}^n)$ be a function orthogonal to the tangent space $T_v\mathcal{M}$ in $L^2(\mathbb{R}^n,|y|^{-1}v^{p_1^*-2})$ satisfying
			\begin{equation*}
				\|\varphi\|_{D^{1,p}(\mathbb{R}^n)}\leq\bar{\delta}.
			\end{equation*}
			Let	$w=w(Dv,Dv+D\varphi):\mathbb{R}^n\rightarrow\mathbb{R}^n$ be defined as in Lemma \ref{le:FZ 2.1} with $(x_1,x_2):=(Dv, D\varphi)$.
			Then \\
			(1) when $1<p\leq\frac{2n}{n+1}$, we have
			\begin{equation*}
				\begin{split}
					\int_{\mathbb{R}^n} \Big(|Dv|^{p-2}|D\varphi|^2 + (p-2)|w|^{p-2}(|D(v&+\varphi)|-|Dv|)^2 +\gamma_0\min\bigl\{|D\varphi|^p,|Dv|^{p-2}|D\varphi|^2\bigr\}\Big)\,\de x\\
					\geq{}&\bigl((p_1^*-1)S^p+\lambda\bigr)\|v\|_{L^{p_1^*}(\mathbb{R}^n;|y|^{-1})}^{p-p_1^*}\int_{\mathbb{R}^n}|y|^{-1}\frac{(v+C_1|\varphi)^{p_1^*}}{v^2+|\varphi|^2}|\varphi|^2\,\de x;
				\end{split}
			\end{equation*}
			(2) when $\frac{2n}{n+1}<p<2$, we have
			\begin{equation*}
				\begin{split}
					\int_{\mathbb{R}^n}\Big(|Dv|^{p-2}|D\varphi|^2 + (p-2)|w|^{p-2}(|D(v+\varphi)|-|Dv|&)^2 +\gamma_0\min\bigl\{|D\varphi|^p,|Dv|^{p-2}|D\varphi|^2\bigr\}\Big)\,\de x\\
					\geq{}&\bigl((p_1^*-1)S^p+\lambda\bigr)\|v\|_{L^{p_1^*}(\mathbb{R}^n;|y|^{-1})}^{p-p_1^*}\int_{\mathbb{R}^n}|y|^{-1}v^{p_1^*-2}|\varphi|^2\,\de x;
				\end{split}
			\end{equation*}
			(3) when $2\leq p<n$, we have
			\begin{equation*}
				\begin{split}
					\int_{\mathbb{R}^n}\Big(|Dv|^{p-2}|D\varphi|^2 + (p-2)|w|^{p-2}&(|D(v+\varphi)|-|Dv|)^2 \Big)\,\de x \\
					\geq{}&\bigl((p_1^*-1)S^p+\lambda\bigr)\|v\|_{L^{p_1^*}(\mathbb{R}^n;|y|^{-1})}^{p-p_1^*}\int_{\mathbb{R}^n}|y|^{-1}v^{p_1^*-2}|\varphi|^2\,\de x.
				\end{split}
			\end{equation*}
		\end{proposition}
		\begin{proof}
			We can assume that $\|v\|_{L^{p_1^*}(\mathbb{R}^n;|y|^{-1})}=1$, as the general case follows from a scaling.
			By replacing $\varphi$ by $|\varphi|$, it suffices to consider the case $\varphi\geq0$.
			We argue by contradiction.
			
			\medskip	
			\textit{ (1) The case $1<p\leq\frac{2n}{n+1}$.}
			Suppose the desired inequality is not true for all test functions $\varphi$, 	then there exists a sequence $\{\varphi_i\}_{i\in \mathbb N^+}$ satisfying $0\not\equiv\varphi_i\rightarrow0$ in $D^{1,p}(\mathbb{R}^n)$ with $\varphi_i$ orthogonal to $T_v\mathcal{M}$, such that
			\begin{equation}\label{eq:c1f}
				\begin{split}
					\int_{\mathbb{R}^n}\Big(|Dv|^{p-2}|D\varphi_i|^2 + (p-2)|w_i|^{p-2}(D(v&+\varphi_i)|-|Dv|)^2 +\gamma_0\min\bigl\{|D\varphi_i|^p,|Dv|^{p-2}|D\varphi_i|^2\bigr\}\Big)\,\de x\\
					<{}&\bigl((p_1^*-1)S^p+\lambda\bigr)\int_{\mathbb{R}^n}|y|^{-1}\frac{(v+C_1|\varphi_i)^{p_1^*}}{v^2+|\varphi_i|^2}|\varphi_i|^2\,\de x,
				\end{split}
			\end{equation}
			where $w_i$ is determined by $\varphi_i$ and $v$ as in the statement of Proposition \ref{prop:sgc}.
			
			Let
			\begin{equation*}
				\varepsilon_i:=\biggl(\int_{\mathbb{R}^n}\bigl(|Dv|+|D\varphi_i|\bigr)^{p-2}|D\varphi_i|^2\,\de x\biggr)^{\frac{1}{2}},
			\end{equation*}
			and set $\hat{\varphi}_i:=\frac{\varphi_i}{\varepsilon_i}$.
			Since $p<2$, it holds
			\begin{equation*}
				\int_{\mathbb{R}^n}\bigl(|Dv|+|D\varphi_i|\bigr)^{p-2}|D\varphi_i|^2\,\de x\leq\int_{\mathbb{R}^n}|D\varphi_i|^{p-2}|D\varphi_i|^2\,\de x=\int_{\mathbb{R}^n}|D\varphi_i|^p\,\de x\rightarrow 0,
			\end{equation*}
			and hence $\varepsilon_i\rightarrow0$.
			For any $R>1$, set
			\begin{equation*}
				\mathcal{R}_i:=\{2|Dv|\geq|D\varphi_i|\},\quad\quad\mathcal{S}_i:=\{2|Dv|<|D\varphi_i|\},
			\end{equation*}
			and
			\begin{equation*}
				\mathcal{R}_{i,R}:=B(0,R)\cap\mathcal{R}_i,\quad\quad\mathcal{S}_{i,R}:=B(0,R)\cap\mathcal{S}_i.
			\end{equation*}
			Using  $(2.2)$ in \cite{FZ}, we can show that the integrand on the left-hand side of \eqref{eq:c1f} is nonnegative. Indeed, for any $x_1\not=0$, the function
			\begin{equation*}
				I(x_1,x_2):=p|x_1|^{p-2}|x_2|^2+p(p-2)|w|^{p-2}(|x_1|-|x_1+x_2|)^2
			\end{equation*}
			satisfies the lower bound
			\begin{equation}\label{eq:lb}
				I(x_1,x_2)\geq c(p)\frac{|x_1|}{|x_1|+|x_2|}|x_1|^{p-2}|x_2|^2\quad\quad \text{for some } c(p)>0,
			\end{equation}
			where $w=w(x_1,x_1+x_2)$ is defined as in the statement of Proposition \ref{prop:sgc}. Therefore, one can deduce that the integral on the left-hand side of \eqref{eq:c1f} is nonnegative. Thus we obtain that, for any $R>1$,
			\begin{equation}\label{eq:c1f1}
				\begin{split}
					&\int_{B(0,R)}\bigg[|Dv|^{p-2}|D\hat{\varphi}_i|^2+(p-2)|w_i|^{p-2}\biggl(\frac{|Dv+D\varphi_i|-|Dv|}{\varepsilon_i}\biggr)^2\\
					&+\gamma_0\min\bigl\{\varepsilon_i^{p-2}|D\hat{\varphi}_i|^p, |Dv|^{p-2}|D\hat{\varphi}_i|^2\bigr\}\bigg]\,\de x
					\leq\bigl((p_1^*-1)S^p+\lambda\bigr)\int_{\mathbb{R}^n}|y|^{-1}\frac{(v+C_1\varphi_i)^{p_1^*}}{v^2+|\varphi_i|^2}|\hat{\varphi}_i|^2\,\de x.
				\end{split}
			\end{equation}
			From \eqref{eq:lb}, we get
			\begin{equation*}
				\begin{split}
					|Dv|^{p-2}|D\hat{\varphi}_i|^2+(p-2)&|w_i|^{p-2}\biggl(\frac{|Dv+D\varphi_i|-|Dv|}{\varepsilon_i}\biggr)^2\\
					&\geq c(p)\frac{|Dv|}{|Dv|+|D\varphi_i|}|Dv|^{p-2}|D\hat{\varphi}_i|^2\geq c(p)|Dv|^{p-2}|D\hat{\varphi}_i|^2\quad\quad\text{on}\:\mathcal{R}_{i,R}.
				\end{split}
			\end{equation*}
		Consequently, combining this estimate with \eqref{eq:c1f1}, we get
		\begin{equation}\label{eq:c1f2}
			\begin{split}
				c(p)\int_{\mathcal{R}_{i,R}}&|Dv|^{p-2}|D\hat{\varphi}_i|^2\,\de x+\gamma_0\int_{\mathcal{S}_{i,R}}\varepsilon_i^{p-2}|D\hat{\varphi}_i|^p\,\de x\\
				\leq{}&\int_{B(0,R)}\bigg[|Dv|^{p-2}|D\hat{\varphi}_i|^2
				+(p-2)|w_i|^{p-2}\biggl(\frac{|Dv+D\varphi_i|-|Dv|}{\varepsilon_i}\biggr)^2\\
				&\quad\quad\quad\quad\quad\quad+\gamma_0\min\bigl\{\varepsilon_i^{p-2}|D\hat{\varphi}_i|^p, |Dv|^{p-2}|D\hat{\varphi}_i|^2\bigr\}\bigg]\,\de x\\
				\leq{}&\bigl((p_1^*-1)S^p+\lambda\bigr)\int_{\mathbb{R}^n}|y|^{-1}\frac{(v+C_1\varphi_i)^{p_1^*}}{v^2+|\varphi_i|^2}|\hat{\varphi}_i|^2\,\de x.
			\end{split}
		\end{equation}
		In particular, the above inequality implies that
		\begin{equation}\label{eq:c1f3}
			\begin{split}
				1=\varepsilon_i^{-2}\int_{\mathbb{R}^n}&\bigl(|Dv|+|D\varphi_i|\bigr)^{p-2}|D\varphi_i|^2\,\de x\\
				&\quad\quad\leq C(p)\biggl(\int_{\mathcal{R}_i}|Dv|^{p-2}|D\hat{\varphi}_i|^2\,\de x + \int_{\mathcal{S}_i}\varepsilon_i^{p-2}|D\hat{\varphi}_i|^p\,\de x\biggr)\\
				&\quad\quad\quad\quad\quad\quad\quad\quad\leq{}C(n,p,k,\gamma_0)\bigl((p_1^*-1)S^p+\lambda\bigr)\int_{\mathbb{R}^n}|y|^{-1}\frac{(v+C_1\varphi_i)^{p_1^*}}{v^2+|\varphi_i|^2}|\hat{\varphi}_i|^2\,\de x.
			\end{split}
		\end{equation}
		Furthermore, applying Corollary \ref{cor:ope} to $\varphi_i$, for $i$ large enough (so that $\varepsilon_i\leq\varepsilon_0$), we have
		\begin{equation}\label{eq:c1f4}
			\begin{split}
				\int_{\mathbb{R}^n}|y|^{-1}\frac{(v+C_1\varphi_i)^{p_1^*}}{v^2+|\varphi_i|^2}|\hat{\varphi}_i|^2&\,\de x\leq{} C(n,p,k,C_1)\int_{\mathbb{R}^n}|y|^{-1}(v+|\varphi_i|)^{p_1^*-2}|\hat{\varphi}_i|^2\,\de x\\
				\leq{}&C(n,p,k,C_1)\int_{\mathbb{R}^n}\bigl(|Dv|+|D\varphi_i|\bigr)^{p-2}|D\hat{\varphi}_i|^2\,\de x\leq C(n,p,k,C_1).
			\end{split}
		\end{equation}
		Hence, combining \eqref{eq:c1f2} with \eqref{eq:c1f4}, by the definition of $\mathcal{S}_{i,R}$, we get
			\begin{equation*}
				\varepsilon_i^{-2}\int_{\mathcal{S}_{i,R}}|Dv|^p\,\de x\leq\varepsilon_i^{p-2}\int_{\mathcal{S}_{i,R}}|D\hat{\varphi}_i|^p\,\de x\leq C(n,p,k,C_1).
			\end{equation*}
			Since $|Dv|$ is uniformly bounded away from zero inside $B(0,R)$, we conclude that
			\begin{equation}\label{eq:c1f5}
				|\mathcal{S}_{i,R}|\rightarrow0\quad\quad\text{as }i\rightarrow\infty,\quad\forall\,R>1.
			\end{equation}
			Now, according to Lemma \ref{le:cl}, we have that $\hat{\varphi}_i$ converges weakly in $D^{1,p}(\mathbb{R}^n)$ to some function $\hat{\varphi}\in D^{1,p}(\mathbb{R}^n)\cap L^2(\mathbb{R}^n;|y|^{-1}v^{p_1^*-2})$, and
			\begin{equation}\label{eq:c1f6}
				\int_{\mathbb{R}^n}|y|^{-1}\frac{(v+C_1\varphi_i)^{p_1^*}}{v^2+|\varphi_i|^2}|\hat{\varphi}_i|^2\,\de x \rightarrow \int_{\mathbb{R}^n}|y|^{-1}v^{p_1^*-2}|\hat{\varphi}|^2\,\de x,\quad\quad\text{as }i\rightarrow\infty.
			\end{equation}
			Using \eqref{eq:c1f2} and \eqref{eq:c1f4}, one has
			\begin{equation*}
				\int_{\mathcal{R}_{i,R}}|Dv|^{p-2}|D\hat{\varphi}_i|^2\,\de x\leq C(n,p,k,C_1).
			\end{equation*}
			Therefore, \eqref{eq:c1f5} and $\hat{\varphi}_i\rightharpoonup\hat{\varphi}$ in $D^{1,p}(\mathbb{R})$ imply that, up to a subsequence,
			\begin{equation*}
				D\hat{\varphi}_i\chi_{\mathcal{R}_{i,R}}\rightharpoonup D\hat{\varphi}\chi_{B(0,R)}\quad\quad\text{in }L^2(\mathbb{R}^n;\mathbb{R}^n),\quad\quad\forall\,R>1.
			\end{equation*}
			In particular, $\hat{\varphi}\in D_{\mathrm{loc}}^{1,2}(\mathbb{R}^n)$.
			Letting $i\rightarrow\infty$ in \eqref{eq:c1f3} and \eqref{eq:c1f4}, and using \eqref{eq:c1f6}, we deduce that
			\begin{equation}\label{eq:c1f7}
				0<c(n,p,\gamma_0)\leq\|\hat{\varphi}\|_{L^2(\mathbb{R}^n;|y|^{-1}v^{p_1^*-2})}\leq C(n,p,k,C_1).
			\end{equation}
			Let us write
			\begin{equation*}
				\hat{\varphi}_i=\hat{\varphi}+\psi_i\quad\quad\text{with}\quad\quad\psi_i:=\hat{\varphi}_i-\hat{\varphi},
			\end{equation*}
			so that
			\begin{equation*}
				\psi_i\rightharpoonup0\:\text{ in }D^{1,p}(\mathbb{R}^n)\quad\quad\text{and}\quad\quad D\psi_i\chi_{\mathcal{R}_i}\rightharpoonup0\:\text{ in }L_{\mathrm{loc}}^2(\mathbb{R}^n;\mathbb{R}^n).
			\end{equation*}
			We now look at the left-hand side of \eqref{eq:c1f1}.
			
			The fact $\varphi_i\rightarrow0$ in $D^{1,p}(\mathbb{R}^n)$ implies that, up to a subsequence, $|w_i|\rightarrow|Dv|$ almost everywhere.
			Also, we can rewrite
			\begin{equation*}
				\begin{split}
					\biggl(\frac{|Dv+D\varphi_i|-|Dv|}{\varepsilon_i}\biggr)^2=\biggl(\biggl[\int_{0}^{1}\frac{Dv+tD\varphi_i}{|Dv+tD\varphi_i|}\,\de t\biggr]\cdot D\hat{\varphi}_i\biggr)^2&\\
					={}\bigg{(}\bigg{[}&\int_{0}^{1}\frac{Dv+tD\varphi_i}{|Dv+tD\varphi_i|}\,\de t\bigg{]}\cdot \bigl(D\hat{\varphi}+D\psi_i\bigr)\bigg{)}^2.
				\end{split}
			\end{equation*}
			Hence, if we set
			\begin{equation*}
				f_{i,1}:=\bigg{[}\int_{0}^{1}\frac{Dv+tD\varphi_i}{|Dv+tD\varphi_i|}\,\de t\bigg{]}\cdot D\hat{\varphi}\quad\quad\text{and}\quad\quad f_{i,2}:=\bigg{[}\int_{0}^{1}\frac{Dv+tD\varphi_i}{|Dv+tD\varphi_i|}\,\de t\bigg{]}\cdot D\psi_i,
			\end{equation*}
			since $\frac{Dv+tD\varphi_i}{|Dv+tD\varphi_i|}\rightarrow\frac{Dv}{|Dv|}$ a.e., it follows from Lebesgue's dominated convergence theorem that
			\begin{equation*}
				f_{i,1}\rightarrow\frac{Dv}{|Dv|}\cdot D\hat{\varphi}\text{ strongly in } L_{\mathrm{loc}}^2(\mathbb{R}^n)\quad\quad\text{and}\quad\quad f_{i,2}\chi_{\mathcal{R}_i}\rightharpoonup0 \text{ weakly in } L_{\mathrm{loc}}^2(\mathbb{R}^n).
			\end{equation*}
			Thus, we can control the first two terms of the left-hand side of \eqref{eq:c1f1} from below as follows:
			\begin{equation}\label{eq:c1f8}
				\begin{split}
					\int_{\mathcal{R}_{i,R}}&\bigg[|Dv|^{p-2}|D\hat{\varphi}_i|^2+(p-2)|w_i|^{p-2}\biggl(\frac{|Dv+D\varphi_i|-|Dv|}{\varepsilon_i}\biggr)^2\bigg]\,\de x\\
					&=\int_{\mathcal{R}_{i,R}}\Big(|Dv|^{p-2}\bigl(|D\hat{\varphi}|^2+2D\psi_i\cdot D\hat{\varphi}\bigr)+(p-2)|w_i|^{p-2}(f_{i,1}^2+2f_{i,1}f_{i,2})\Big)\,\de x\\
					&\quad+\int_{\mathcal{R}_{i,R}}(|Dv|^{p-2}|D\psi_i|^2+(p-2)|w_i|^{p-2}f_{i,2}^2)\,\de x\\
					&\geq \int_{\mathcal{R}_{i,R}}\Big(|Dv|^{p-2}\bigl(|D\hat{\varphi}|^2+2D\psi_i\cdot D\hat{\varphi}\bigr)+(p-2)|w_i|^{p-2}(f_{i,1}^2+2f_{i,1}f_{i,2})\Big)\,\de x,
				\end{split}
			\end{equation}
			where the last inequality follows from the nonnegativity of $|Dv|^{p-2}|D\psi_i|^2+(p-2)|w_i|^{p-2}f_{i,2}^2$ (one can easily obtain by the fact $f_{i,2}^2\leq|D\psi|^2$ and the definition of $w_i$).
			
			Then, combining the convergences
			\begin{equation*}
				\begin{split}
					D\psi_i\chi_{\mathcal{R}_i}&\rightharpoonup0,\quad\quad f_{i,1}\rightarrow\frac{Dv}{|Dv|}\cdot D\hat{\varphi},\quad\quad f_{i,2}\chi_{\mathcal{R}_i}\rightharpoonup0\quad\text{in }L_{\mathrm{loc}}^2(\mathbb{R}^n),\\
					&|w_i|\rightarrow|Dv|\text{ a.e.}\quad\quad\text{and}\quad\quad|(B(0,R))\backslash\mathcal{R}_{i,R}|\rightarrow0
				\end{split}
			\end{equation*}
			with the fact that
			\begin{equation*}
				|w_i|^{p-2}\leq C(p)|Dv|^{p-2},
			\end{equation*}
			we infer from Lebesgue's dominated convergence theorem that the last term in \eqref{eq:c1f8} converges to
			\begin{equation*}
				\int_{B(0,R)}\bigg[|Dv|^{p-2}|D\hat{\varphi}|^2+(p-2)|Dv|^{p-2}\biggl(\frac{Dv}{|Dv|}\cdot D\hat{\varphi}\biggr)^2\bigg]\,\de x.
			\end{equation*}
			Recalling \eqref{eq:c1f1} and \eqref{eq:c1f6}, since $R>1$ is arbitrary and the integrand is nonnegative, this proves that
			\begin{equation}\label{eq:c1f9}
				\int_{\mathbb{R}^n}\bigg[|Dv|^{p-2}|D\hat{\varphi}|^2+(p-2)|Dv|^{p-2}\biggl(\frac{Dv}{|Dv|}\cdot D\hat{\varphi}\biggr)^2\bigg]\,\de x\leq\bigl((p_1^*-1)S^p+\lambda\bigr)\int_{\mathbb{R}^n}|y|^{-1}v^{p_1^*-2}|\hat{\varphi}|^2\,\de x.
			\end{equation}
			On the other hand, $\hat{\varphi}$ is the weak limit of $\hat{\varphi}_i$ in $D^{1,p}(\mathbb{R}^n)$.
			Hence, thanks to Definition 3.7, the orthogonality of $\varphi_i$ (and so of $\hat{\varphi}_i$) implies that  $\hat{\varphi}$ is also orthogonal to $T_v\mathcal{M}$.
			Since $\hat{\varphi}\in L^2(\mathbb{R}^n;|y|^{-1}v^{p_1^*-2})$, \eqref{eq:c1f7} and \eqref{eq:c1f9} contradict Proposition \ref{prop:sg}, our proof for the case $1<p\leq\frac{2n}{n+1}$ is completed.
			
			\medskip	
			\textit{ (2) The case $\frac{2n}{n+1}<p<2$.}
			Assume by contradiction that the inequality does not hold, then there exists a sequence  $\{\varphi_i\}_{i\in \mathbb N^+}$ satisfying $0\not\equiv\varphi_i\rightarrow 0$ in $D^{1,p}(\mathbb{R}^n)$ with $\varphi_i$ orthogonal to $T_v\mathcal{M}$, such that
			\begin{equation}\label{eq:c2f}
				\begin{split}
					\int_{\mathbb{R}^n}\Big(|Dv|^{p-2}|D\varphi_i|^2 + (p-2)|w_i|^{p-2}&(|D(v+\varphi_i)|-|Dv|)^2 +\gamma_0\min\bigl\{|D\varphi_i|^p,|Dv|^{p-2}|D\varphi_i|^2\bigr\}\Big)\,\de x\\			&\quad\quad\quad\quad\quad<\bigl((p_1^*-1)S^p+\lambda\bigr)\int_{\mathbb{R}^n}|y|^{-1}v^{p_1^*-2}|\varphi_i|^2\,\de x,
				\end{split}
			\end{equation}
			where $w_i$ is determined by $\varphi_i$ and $v$ as in the statement of Proposition \ref{prop:sgc}.
			
			Similar to the case $1<p\leq\frac{2n}{n+1}$, we define
			\begin{equation*}
				\varepsilon_i:=\biggl(\int_{\mathbb{R}^n}\bigl(|Dv|+|D\varphi_i|\bigr)^{p-2}|D\varphi_i|^2\,\de x\biggr)^{\frac{1}{2}},\quad\quad\hat{\varphi}_i:=\frac{\varphi_i}{\varepsilon_i},
			\end{equation*}
			and split $B(0,R)=\mathcal{R}_{i,R}\cup\mathcal{S}_{i,R}$.
			Then, the following analogues of \eqref{eq:c1f2} and \eqref{eq:c1f3} also hold for the case $\frac{2n}{n+1}<p<2$:
			\begin{equation}\label{eq:aa1}
				\begin{split}
					c(p)\int_{\mathcal{R}_{i,R}}&|Dv|^{p-2}|D\hat{\varphi}_i|^2\,\de x+\gamma_0\int_{\mathcal{S}_{i,R}}\varepsilon_i^{p-2}|D\hat{\varphi}_i|^p\,\de x\\
					\leq{}&\int_{B(0,R)}\bigg[|Dv|^{p-2}|D\hat{\varphi}_i|^2
					+(p-2)|w_i|^{p-2}\biggl(\frac{|Dv+D\varphi_i|-|Dv|}{\varepsilon_i}\biggr)^2\\
					&\quad\quad\quad\quad\quad\quad+\gamma_0\min\bigl\{\varepsilon_i^{p-2}|D\hat{\varphi}_i|^p, |Dv|^{p-2}|D\hat{\varphi}_i|^2\bigr\}\bigg]\,\de x\\
					\leq{}&\bigl((p_1^*-1)S^p+\lambda\bigr)\int_{\mathbb{R}^n}|y|^{-1}v^{p_1^*-2}|\varphi_i|^2\,\de x,
				\end{split}
			\end{equation}
			and
			\begin{equation}\label{eq:aa2}
				\begin{split}
					1=\varepsilon_i^{-2}\int_{\mathbb{R}^n}&\bigl(|Dv|+|D\varphi_i|\bigr)^{p-2}|D\varphi_i|^2\,\de x\\
					&\quad\quad\leq C(p)\biggl(\int_{\mathcal{R}_i}|Dv|^{p-2}|D\hat{\varphi}_i|^2\,\de x + \int_{\mathcal{S}_i}\varepsilon_i^{p-2}|D\hat{\varphi}_i|^p\,\de x\biggr)\\
					&\quad\quad\quad\quad\quad\quad\quad\quad\leq{}C(n,p,k,\gamma_0)\bigl((p_1^*-1)S^p+\lambda\bigr)\int_{\mathbb{R}^n}|y|^{-1}v^{p_1^*-2}|\varphi_i|^2\,\de x.
				\end{split}
			\end{equation}
			Thanks to H$\rm{\ddot{o}}$lder's inequality, we have
			\begin{equation*}
				\begin{split}
					\int_{\mathbb{R}^n}|D&\hat{\varphi}_i|^p\,\de x\leq\biggl(\int_{\mathbb{R}^n}(|Dv|+|D\varphi|_i)^{p-2}|D\hat{\varphi}_i|^2\,\de x\biggr)^{\frac{p}{2}}\biggl(\int_{\mathbb{R}^n}(|Dv|+|D\varphi_i|)^p\,\de x\biggr)^{1-\frac{p}{2}}\\
					&=\biggl(\int_{\mathbb{R}^n}(|Dv|+|D\varphi_i|)^p\,\de x\biggr)^{1-\frac{p}{2}}\leq C(p)\biggl[\biggl(\int_{\mathbb{R}^n}|Dv|^p\,\de x\biggr)^{1-\frac{p}{2}} + \varepsilon_i^{\frac{p(2-p)}{2}}\biggl(\int_{\mathbb{R}^n}|D\hat{\varphi}_i|^p\,\de x\biggr)^{1-\frac{p}{2}}\biggr].
				\end{split}
			\end{equation*}
			It follows that
			\begin{equation}\label{eq:c2f1}
				\int_{\mathbb{R}^n}|D\hat{\varphi}_i|^p\,\de x\leq C(n,p,k).
			\end{equation}
			Thus, up to a subsequence, $\hat{\varphi}_i\rightarrow\hat{\varphi}$ weakly in $D^{1,p}(\mathbb{R}^n)$ and strongly in $L_{\mathrm{loc}}^2(\mathbb{R}^n)$ (note that $p^*>2$).
			In addition, \eqref{eq:c2f1} together with H$\rm{\ddot{o}}$lder's inequality and the Hardy-Sobolev-Maz'ya inequalities yield
			\begin{equation*}
				\begin{split}
					\int_{\mathbb{R}^n\backslash B(0,R)}|y|^{-1}v^{p_1^*-2}|D\hat{\varphi}_i|^2\,\de x\leq{}&\biggl(\int_{\mathbb{R}^n\backslash B(0,R)}|y|^{-1}v^{p_1^*}\,\de x\biggr)^{1-\frac{2}{p_1^*}}\biggl(\int_{\mathbb{R}^n\backslash B(0,R)}|y|^{-1}|\hat{\varphi}_i|^{p_1^*}\,\de x\biggr)^{\frac{2}{p_1^*}}\\
					\leq{}&\biggl(\int_{\mathbb{R}^n\backslash B(0,R)}|y|^{-1}v^{p_1^*}\,\de x\biggr)^{1-\frac{2}{p_1^*}}\biggl(\int_{\mathbb{R}^n}|D\hat{\varphi}_i|^p\,\de x\biggr)^{\frac{2}{p}},\quad\quad\forall\,R\geq0.
				\end{split}
			\end{equation*}
			Combining this estimate with \eqref{eq:c2f1} and the fact $\hat{\varphi}_i\rightarrow\hat{\varphi}$ in $L_{\mathrm{loc}}^2(\mathbb{R}^n)$, we conclude that $\hat{\varphi}_i\rightarrow\hat{\varphi}$ strongly in $L^2(\mathbb{R}^n;|y|^{-1}v^{p_1^*-2})$.
			In particular, letting $i\rightarrow\infty$ in \eqref{eq:aa2}, we obtain
			\begin{equation*}
				0<C(n,p,k)<\|\hat{\varphi}\|_{L^2(\mathbb{R}^n;|y|^{-1}v^{p_1^*-2})}.
			\end{equation*}
			Similar to the case $1<p\leq\frac{2n}{n+1}$, \eqref{eq:aa1} implies that
			\begin{equation*}
				|\mathcal{S}_{i,R}|\rightarrow0\quad\text{and}\quad\int_{\mathcal{R}_{i,R}}|Dv|^{p-2}|D\hat{\varphi}_i|\,\de x\leq C(n,p,k),\quad\quad\forall\,R>1.
			\end{equation*}
			It follows from $\hat{\varphi}_i\rightharpoonup\hat{\varphi}$ in $D^{1,p}(\mathbb{R}^n)$ that, up to a subsequence,
			\begin{equation*}
				D\hat{\varphi}_i\chi_{\mathcal{R}_{i,R}}\rightharpoonup D\hat{\varphi}\chi_{B(0,R)}\quad\quad\text{in }L^2(\mathbb{R}^n;\mathbb{R}^n),\quad\quad\forall\,R>1.
			\end{equation*}
			We decompose
			\begin{equation*}
				\hat{\varphi}_i=\hat{\varphi}+\psi_i\quad\quad\text{with}\quad\quad\psi_i:=\hat{\varphi}_i-\hat{\varphi}.
			\end{equation*}	
			Through quite similar argument as in the case $1<p\leq\frac{2n}{n+1}$, we can deduce that
			\begin{equation*}
				\begin{split}
					\liminf_{i\rightarrow\infty}\int_{\mathcal{R}_{i,R}}\bigg[|Dv|^{p-2}|D\hat{\varphi}_i|^2+(p-2)&|w_i|^{p-2}\biggl(\frac{|Dv+D\varphi_i|-|Dv|}{\varepsilon_i}\biggr)^2\bigg]\,\de x\\
					\geq{}&\int_{B(0,R)}\bigg[|Dv|^{p-2}|D\hat{\varphi}|^2+(p-2)|Dv|^{p-2}\biggl(\frac{Dv}{|Dv|}\cdot D\hat{\varphi}\biggr)^2\bigg]\,\de x.
				\end{split}
			\end{equation*}
			Recalling \eqref{eq:c2f}, since $R>1$ is arbitrary and the above integrands are nonnegative, we have proved that \eqref{eq:c1f9} holds, which contradicts Proposition \ref{prop:sg} because $\hat{\varphi}$ is orthogonal to $T_v\mathcal{M}$ ($\hat{\varphi}$ is the strong $L^2(\mathbb{R}^n;|y|^{-1}v^{p_1^*-2})$-limit of $\hat{\varphi}_i$).
			
			\medskip	
			\textit{(3) The case $2\leq p<n$}.
			If the conclusion is false, then there exists a sequence $\{\varphi_i\}_{i\in \mathbb N^+}$ orthogonal to $T_v\mathcal{M}$, satisfying  $\varphi_i\not\equiv 0$ and $\varphi_i\rightarrow 0$ in $D^{1,p}(\mathbb{R}^n)$, such that
			\begin{equation}\label{eq:c3f}
				\begin{split}
					\int_{\mathbb{R}^n}\Big(|Dv|^{p-2}|D\varphi_i|^2 + (p-2)|w_i|^{p-2}&(|D(v+\varphi_i)|-|Dv|)^2\Big)\,\de x\\
					&\quad\quad\quad\quad\quad\quad<\bigl((p_1^*-1)S^p+\lambda\bigr)\int_{\mathbb{R}^n}|y|^{-1}v^{p_1^*-2}|\varphi_i|^2\,\de x,
				\end{split}
			\end{equation}
			where $w_i$ is determined by $\varphi_i$ and $v$ as in the statement of Proposition \ref{prop:sgc}.
			
			Let
			\begin{equation*}
				\varepsilon_i:=\biggl(\int_{\mathbb{R}^n}|Dv|^{p-2}|D\varphi_i|^2\,\de x\biggr)^{\frac{1}{2}}\quad\quad\text{and}\quad\quad\hat{\varphi}_i:=\frac{\varphi_i}{\varepsilon_i}.
			\end{equation*}
			Noting that $p\geq2$, it follows from H$\rm{\ddot{o}}$lder's inequality that
			\begin{equation*}
				\int_{\mathbb{R}^n}|Dv|^{p-2}|D\varphi_i|^2\,\de x\leq\biggl(\int_{\mathbb{R}^n}|Dv|^p\,\de x\biggr)^{1-\frac{p}{2}}\biggl(\int_{\mathbb{R}^n}|D\varphi_i|^p\,\de x\biggr)^{\frac{p}{2}}\rightarrow0,
			\end{equation*}
			hence $\varepsilon_i\rightarrow0$ as $i\rightarrow+\infty$.
			
			Since $1=\|\hat{\varphi}_i\|_{D^{1,2}(\mathbb{R}^n;|Dv|^{p-2})}$, Proposition \ref{prop:ce} implies that, up to a subsequence, $\hat{\varphi}_i\rightarrow\hat{\varphi}$ weakly in $D_{\rm{loc}}^{1,2}(\mathbb{R}^n;|Dv|^{p-2})$ and strongly in $L^2(\mathbb{R}^n;|y|^{-1}v^{p_1^*-2})$.
			Due to $p\geq2$, it follows from \eqref{eq:c3f} that
			\begin{equation*}
				1=\int_{\mathbb{R}^n}|Dv|^{p-2}|D\hat{\varphi}_i|^2\,\de x\leq\bigl((p_1^*-1)S^p+\lambda\bigr)\int_{\mathbb{R}^n}|y|^{-1}v^{p_1^*-2}|\varphi_i|^2\,\de x,
			\end{equation*}
			so we deduce that
			\begin{equation*}
				0<C(n,p,k)<\|\hat{\varphi}\|_{L^2(\mathbb{R}^n;|y|^{-1}v^{p_1^*-2})}.
			\end{equation*}
			In addition, since the integrand on the left-hand side of \eqref{eq:c3f} is nonnegative, for any $R>1$, it holds
			\begin{equation}\label{eq:c3f1}
				\begin{split}
					\int_{B(0,R)}\bigg[|Dv|^{p-2}|D\hat{\varphi}_i|^2+(p-2)|w_i|^{p-2}&\biggl(\frac{|Dv+D\varphi_i|-|Dv|}{\varepsilon_i}\biggr)^2\bigg]\,\de x\\
					&\quad\quad\quad\quad\quad\leq\bigl((p_1^*-1)S^p+\lambda\bigr)\int_{\mathbb{R}^n}|y|^{-1}v^{p_1^*-2}|\hat{\varphi}_i|^2\,\de x.
				\end{split}
			\end{equation}
			Noting that, for any $R>1$,
			\begin{equation*}
				0<c(R)\leq|Dv|\leq C(R)\quad\quad\text{in }B(0,R),
			\end{equation*}
			and writing
			\begin{equation*}
				\hat{\varphi}_i=\hat{\varphi}+\psi_i\quad\quad\text{with}\quad\quad\psi_i:=\hat{\varphi}_i-\hat{\varphi},
			\end{equation*}
			we have
			\begin{equation*}
				\psi_i\rightharpoonup0\quad\text{in }D_{\rm{loc}}^{1,2}(\mathbb{R}^n).
			\end{equation*}
			Considering the left-hand side of \eqref{eq:c3f1}, we can deduce in the same way as the case $\frac{2n}{n+1}<p<2$ that
			\begin{equation*}
				\begin{split}
					\liminf_{i\rightarrow\infty}\int_{B(0,R)}\bigg[|Dv|^{p-2}|D\hat{\varphi}_i|^2+(p-2)&|w_i|^{p-2}\biggl(\frac{|Dv+D\varphi_i|-|Dv|}{\varepsilon_i}\biggr)^2\bigg]\,\de x\\
					\geq{}&\int_{B(0,R)}\bigg[|Dv|^{p-2}|D\hat{\varphi}|^2+(p-2)|Dv|^{p-2}\biggl(\frac{Dv}{|Dv|}\cdot D\hat{\varphi}\biggr)^2\bigg]\,\de x.
				\end{split}
			\end{equation*}
			Since $R>1$ is arbitrary, combining \eqref{eq:c3f1} with the above estimate show that \eqref{eq:c1f9} holds, which contradicts Proposition \ref{prop:sg} due to the orthogonality of $\hat{\varphi}$ to $T_v\mathcal{M}$. This concludes our proof of Proposition \ref{prop:sgc}.
		\end{proof}

		\section{Poof of Theorem \ref{thm:main}}\label{sec:main}
		
		Thanks to the estimates proved in Sections \ref{sec:cr} and \ref{sec:sg}, we can now follow the compactness strategy of \cite{BE,FN,FZ} to prove Theorem \ref{thm:main}.
		By scaling, we can assume $\|u\|_{L^{p_1^*}(\mathbb{R}^n;|y|^{-1})}=1$. Since the right-hand side of \eqref{eq:se} is trivially bounded by $2$, it suffices to prove the result for $\delta(u) \ll1$.
		
		\subsection{Two preliminary lemmas}
		To prove Theorem \ref{thm:main}, we need a lemma based on the concentration-compactness results (Theorem \ref{thm:cc}).
		\begin{lemma}[Compactness]\label{le:cg}
			For any $\hat{\varepsilon}>0$, there exists a constant $\hat{\delta}=\hat{\delta}(n,p,k,\hat{\varepsilon})$ such that the following holds: if
			\begin{equation*}
				\|Du\|_{L^p(\mathbb{R}^n)}-S\leq\hat{\delta},
			\end{equation*}
			then there exists a function $\hat{v}\in\mathcal{M}$ which minimizes the right-hand side of \eqref{eq:se}, and $\hat{v}$ satisfies
			\begin{equation}\label{eq:fx}
				\|D(u-\hat{v})\|_{L^p(\mathbb{R}^n)}\leq\hat{\varepsilon}.
			\end{equation}
		\end{lemma}
		
		\begin{proof}
			We begin by showing the following fact: for all $\xi>0$, there exists $\delta_0=\delta(n,p,\xi)>0$ such that if
			\begin{equation*}
				\delta(u)\leq\delta_0,
			\end{equation*}
			then
			\begin{equation}\label{eq:cl1}
				\inf_{v\in\mathcal{M}}\|D(u-v)\|_{L^p(\mathbb{R}^n)}^p\leq\xi.
			\end{equation}
			Otherwise, for some $\xi>0$, there exists a sequence $\{u_i\}_{i\in \mathbb N^+}\subseteq D^{1,p}(\mathbb{R}^n)$ such that $\|u_i\|_{L^{p_1^*}(\mathbb{R}^n;|y|^{-1})}=1$ and $\delta(u_i)\rightarrow0$, while
			\begin{equation*}
				\inf_{v\in\mathcal{M}}\|D(u-v)\|_{L^p(\mathbb{R}^n)}^p>\xi.
			\end{equation*}
			Theorem \ref{thm:cc} ensures that there exist sequences $\{\lambda_i\}_{i\in \mathbb N^+}$ and $\{x_i\}_{i\in \mathbb N^+}$ such that, up to a subsequence, $\lambda_i^{-\frac{n-p}{p}}u_i(\lambda_i(x-x_i))$ converges strongly in $D^{1,p}(\mathbb{R}^n)$ to some $\bar{v}\in\mathcal{M}$, while it also holds
			\begin{equation*}
				\xi<\biggl\|Du_i-D\biggl[\lambda_i^{\frac{n-p}{p}}\bar{v}\biggl(\frac{\cdot}{\lambda_i}-x_i\biggr)\biggr]\biggr\|_{L^p(\mathbb{R}^n)}=\biggl\|D\biggl[\lambda_i^{-\frac{n-p}{p}}u_i(\lambda_i(\cdot-x_i))\biggr]-D\bar{v}\biggr\|_{L^p(\mathbb{R}^n)}\rightarrow0.
			\end{equation*}
			This gives a contradiction for $i$  sufficiently large, hence \eqref{eq:cl1} holds.
			Therefore, we can let $\xi$ as small as we want by choosing $\delta_0$ small enough.
			
			The infimum on the left-hand side of \eqref{eq:cl1} is attained.
			In fact, let $\{v_i\}$ be a minimizing sequence of \eqref{eq:fx} with $v_i=c_iv_{\lambda_i,x_i}$.
			The sequences $\{c_i\}$, $\{\lambda_i\}$, $\{1/\lambda_i\}$ and $\{x_i\}$ are bounded: if $\lambda_i\rightarrow\infty$ or $\lambda_i\rightarrow0$, then for $i$ large enough,
			\begin{equation*}
				\int_{\mathbb{R}^n}|D(u-v_i)|^p\,\de x\geq\frac{1}{2}\int_{\mathbb{R}^n}|Du|^p\,\de x,
			\end{equation*}
			which contradicts \eqref{eq:cl1}.
			The analogous argument works if $|x_i|\rightarrow\infty$ or $|c_i|\rightarrow\infty$.
			Thus $\{c_i\}$, $\{\lambda_i\}$, $\{1/\lambda_i\}$ and $\{x_i\}$ are bounded, so up to a subsequence, $(c_i,\lambda_i,x_i)\rightarrow(c_0,\lambda_0,x_0)$ for some $(c_0,\lambda_0,x_0)\in\mathbb{R}\times\mathbb{R}^+\times\mathbb{R}^n$.
			Since the functions $\{cv_{\lambda,x}\}$ are smooth, decay nicely, and depend smoothly on the parameters, we deduce that $v_i\rightarrow c_0v_{\lambda_0,x_0}=:\hat{v}$ in $D^{1,p}(\mathbb{R}^n)$ (actually, they also converge in $C^k$
			for any $k\in\mathbb{N}$).
			Hence $\hat{v}$ attains the infimum, and \eqref{eq:fx} follows by taking $\xi$ small.
		\end{proof}
		Thus, by the Hardy-Sobolev-Maz'ya inequalities we know $\hat{v}$ in Lemma \ref{le:cg} satisfies: $\|\hat{v}\|_{L^{p_1^*}(\mathbb{R}^n;|y|^{-1})}$ is close to $1$ (since $\|u\|_{L^{p_1^*}(\mathbb{R}^n;|y|^{-1})}=1$), so all the lemmas, propositions and corollaries in previous Sections \ref{sec:cr} and \ref{sec:sg} hold.
		
		The basic idea of the proof would be expanding $u$ around $\hat{v}$. Unfortunately, by our choice of $\hat{v}$, we do not have the desired orthogonality to use Proposition \ref{prop:sgc}.
		Hence, we need the following result.
		\begin{lemma}[Orthogonality]\label{le:o}
			Let $\|u\|_{L^{p_1^*}(\mathbb{R}^n;|y|^{-1})}=1$, and assume that $\|D(u-\hat{v})\|_{L^p(\mathbb{R}^n)}\leq\hat{\varepsilon}$ with $\hat{v}=v_{a,1,0}\in\mathcal{M}$.
			There exist $\varepsilon'=\varepsilon'(n,p,k)>0$ and a modulus of continuity $\omega:\mathbb{R}^+\rightarrow\mathbb{R}^+$ such that the following holds: If $\hat{\varepsilon}\leq\varepsilon'$, then there exists $v\in\mathcal{M}$ such that $u-v$ is orthogonal to $T_v\mathcal{M}$ and $\|D(u-v)\|_{L^p(\mathbb{R}^n)}\leq\omega(\hat{\varepsilon})$.
		\end{lemma}
		\begin{proof}
			Given $u$ as in the statement of Lemma \ref{le:o}, we consider the minimization of the functional
			\begin{equation}\label{eq:of}
				\mathcal{M}\ni v\mapsto\mathcal{F}_u[v]:=\frac{1}{p_1^*}\int_{\mathbb{R}^n}|y|^{-1}|v|^{p_1^*}\,\de x-\frac{1}{p_1^*-1}\int_{\mathbb{R}^n}|y|^{-1}|v|^{p_1^*-2}vu\,\de x.
			\end{equation}
			Assume first that $u=\hat{v}\in\mathcal{M}$.
			We claim that the minimizer of \eqref{eq:of} is unique and coincides with $u$.
			
			To prove this, we note that, by H$\mathrm{\ddot{o}}$lder's inequality,
			\begin{equation}\label{eq:ofe}
				\begin{split}
					\mathcal{F}_u[v]\geq{}&\frac{1}{p_1^*}\int_{\mathbb{R}^n}|y|^{-1}|v|^{p_1^*}\,\de x-\frac{1}{p_1^*-1}\biggl(\int_{\mathbb{R}^n}|y|^{-1}|u|^{p_1^*}\,\de x\biggr)^{\frac{1}{p_1^*}}\biggl(\int_{\mathbb{R}^n}|y|^{-1}|v|^{p_1^*}\,\de x\biggr)^{\frac{p_1^*-1}{p_1^*}}\\
					\geq{}&-\frac{1}{p_1^*(p_1^*-2)}\int_{\mathbb{R}^n}|y|^{-1}|u|^{p_1^*}\,\de x,
				\end{split}
			\end{equation}
			where the second inequality follows from the fact that the function
			\begin{equation*}
				\mathbb{R}^+\ni s\mapsto\frac{1}{p_1^*}s^{p_1^*} -\frac{1}{p_1^*-1}As^{p_1^*-1}
			\end{equation*}
			is uniquely minimized at $s=A$.
			Noticing that the last term in \eqref{eq:ofe} coincides with $\mathcal{F}_u[u]$, and that equality holds in the two inequalities of \eqref{eq:ofe} at the same time if and only if $v = u$, the claim was proved.
			
			Now, if $u$ is close to $\hat{v}=v_{a,1,0}$ in $D^{1,p}(\mathbb{R}^n)$-norm, it follows from compactness that the minimum of the function
			\begin{equation*}
				\mathbb{R}\times\mathbb{R}^+\times\mathbb{R}^n\ni(a,b,x_0)\mapsto\mathcal{F}_u[v_{a,1,0}]
			\end{equation*}
			is attained at some values $(a', b', x_0')$ close to $(a, 1, 0)$, hence $\|Dv_{a',b',x_0'}-D\hat{v}\|_{L^p(\mathbb{R}^n)}\ll1$.
			Thus, by the assumption that $u$ and $\hat{v}$ are close in $D^{1,p}(\mathbb{R}^n)$, we deduce that
			\begin{equation*}
				\|Du-Dv_{a',b',x_0'}\|_{L^p(\mathbb{R}^n)}\rightarrow0\quad\quad\text{as } \,\, \|Du-D\hat v\|_{L^p(\mathbb{R}^n)}\rightarrow0,
			\end{equation*}
			which proves the existence of a modulus of continuity $\omega$.
			Finally, it is not difficult to check that if $v\in\mathcal{M}$ is close to $v_{a,1,0}$ and minimizes $\mathcal{F}_u$, then
			\begin{equation*}
				0=\frac{\de}{\de t}\bigg|_{t=0}\mathcal{F}_u[v+t\xi]=\int_{\mathbb{R}^n}|y|^{-1}v^{p_1^*-2}\xi(v-u)\,\de x,\quad\quad\forall\,\xi\in T_v\mathcal{M}.
			\end{equation*}
			This concludes our proof.
		\end{proof}

		\subsection{Proof of our main result - Theorem \ref{thm:main}}
		Thanks to Lemma \ref{le:o}, given $u$ as at the beginning of Section 4 with $\delta(u)$ sufficiently small, we can find $v\in\mathcal{M}$ close to $u$ such that $u-v$ is orthogonal to $T_v\mathcal{M}$. More precisely, $u$ can be written as $u=v+\varepsilon\varphi$, where $\varepsilon\leq\omega(\hat{\varepsilon})$ with $\hat{\varepsilon}\leq\varepsilon'$, $\|D\varphi\|_{L^p(\mathbb{R}^n)}=1$, and $\varphi$ is orthogonal to $T_v\mathcal{M}$.
		
		Observe that, for $\delta(u)$ small,
		\begin{equation*}
			\delta(u)=\|Du\|_{L^p(\mathbb{R}^n)}-S\geq C(n,p,k)\bigl(\|Du\|_{L^p(\mathbb{R}^n)}^p-S^p\bigr).
		\end{equation*}
		In the following argument, several parameters will appear, and these parameters depend on each other.
		To simplify the notation, we shall not give their explicit dependence on $n$, $p$ and $k$, but we emphasize how the parameters depend on each other, at least until they were fixed. Our proof will be carried out by discussing the following three different cases.
		
		\medskip
		$\bullet$ \textit{ The case $1<p\leq\frac{2n}{n+1}$}.
		Let $\kappa>0$ be a small constant to be fixed later.
		By Lemma \ref{le:FZ 2.1}, we have
		\begin{equation}\label{eq:a1}
			\begin{split}
				\|Du\|_{L^p(\mathbb{R}^n)}^p\geq{}&\int_{\mathbb{R}^n}|Dv|^p\,\de x+\varepsilon p\int_{\mathbb{R}^n}|Dv|^{p-2}Dv\cdot D\varphi\,\de x\\
				&+\frac{\varepsilon p(1-\kappa)}{2}\Biggl(\int_{\mathbb{R}^n}\bigg[|Dv|^{p-2}|D\varphi|^2+(p-2)|w|^{p-2}\biggl(\frac{|Du|-|Dv|}{\varepsilon}\biggr)^2\bigg]\,\de x\Biggr)\\
				&+c_0(\kappa)\int_{\mathbb{R}^n}\min\bigl\{\varepsilon^p|D\varphi|^p, \varepsilon^2|Dv|^{p-2}|D\varphi|^2\bigr\}\,\de x,
			\end{split}
		\end{equation}
		where $w=w(Dv,Dv+Du)$ is defined as in Lemma \ref{le:FZ 2.1}.
		On the other hand, by Lemma \ref{le:FZ 2.4} and the concavity of  $t\mapsto t^{p/p_1^*}$, we have
		\begin{equation*}
			\begin{split}
				1={}&\|u\|_{L^{p_1^*}(\mathbb{R}^n;|y|^{-1})}^p=\biggl(\int_{\mathbb{R}^n}|y|^{-1}|v+\varepsilon\varphi|^{p_1^*}\,\de x\biggr)^{\frac{p}{p_1^*}}\\
				\leq{}&\|v\|_{L^{p_1^*}(\mathbb{R}^n;|y|^{-1})}^{p-p_1^*}\int_{\mathbb{R}^n}\bigg[\varepsilon p|y|^{-1}v^{p_1^*-1}\varphi + \varepsilon^2\biggl(\frac{p(p_1^*-1)}{2}+\frac{p}{p_1^*}\kappa\biggr)|y|^{-1}\frac{\big{(}|v|+C_1(\kappa)\varepsilon|\varphi|\big{)}^{p_1^*}}{|v|^2+\varepsilon^2\varphi^2}|\varphi|^2\bigg]\,\de x\\
				&+\|v\|_{L^{p_1^*}(\mathbb{R}^n;|y|^{-1})}^p.
			\end{split}
		\end{equation*}
		By \eqref{eq:hsm}, one has
		\begin{equation*}
			\int_{\mathbb{R}^n}|Dv|^{p-2}Dv\cdot D\varphi\,\de x=\|v\|_{L^{p_1^*}(\mathbb{R}^n;|y|^{-1})}S^p\int_{\mathbb{R}^n}|y|^{-1}v^{p_1^*}\varphi\,\de x,
		\end{equation*}
		and $\|Dv\|_{L^p(\mathbb{R}^n)}=S\|v\|_{L^{p_1^*}(\mathbb{R}^n;|y|^{-1})}$. Then we can immediately conclude that
		\begin{equation*}
			\begin{split}
				C(n,p,k)\delta(u)\geq{}&\|Du\|_{L^p(\mathbb{R}^n)}^p-S^p\|u\|_{L^{p_1^*}(\mathbb{R}^n;|y|^{-1})}^p\\
				\geq{}&\frac{\varepsilon^2p(1-\kappa)}{2}\Biggl(\int_{\mathbb{R}^n}\bigg[|Dv|^{p-2}|D\varphi|^2+(p-2)|w|^{p-2}\biggl(\frac{|Du|-|Dv|}{\varepsilon}\biggr)^2\bigg]\,\de x\Biggr)\\
				&+c_0(\kappa)\int_{\mathbb{R}^n}\min\biggl\{\varepsilon^p|D\varphi|^p, \varepsilon^2|Dv|^{p-2}|D\varphi|^2\biggr\}\,\de x\\
				&-\varepsilon^2\|v\|_{L^{p_1^*}(\mathbb{R}^n;|y|^{-1})}^{p-p_1^*}\biggl(\frac{p(p_1^*-1)}{2}+\frac{p}{p_1^*}\kappa\biggr)\int_{\mathbb{R}^n}|y|^{-1}\frac{\big{(}|v|+C_1(\kappa)\varepsilon|\varphi|\big{)}^{p_1^*}}{|v|^2+\varepsilon^2\varphi^2}|\varphi|^2\,\de x.
			\end{split}
		\end{equation*}
		Now, for $\delta(u)\leq\delta'=\delta'(\varepsilon,\kappa,\gamma_0)$ small enough, Proposition \ref{prop:sgc} allows us to absorb the last term in the above inequality. More precisely, we have
		\begin{equation*}
			\begin{split}
				C(n,p,k)&\delta(u)\\
				\geq{} p\epsilon^2&\biggl(\frac{(1-\kappa)}{2}-\frac{(p_1^*-1)+\frac{2}{p_1^*}\kappa}{2(p_1^*-1)+2\lambda S^{-p}}\biggr)\Biggl( \int_{\mathbb{R}^n}\bigg[|Dv|^{p-2}|D\varphi|^2+(p-2)|w|^{p-2}\biggl(\frac{|Du|-|Dv|}{\epsilon}\biggr)^2\bigg]\,\de x\Biggr)\\
				+{}&\biggl(c_0(\kappa)-\gamma_0\frac{p\big[(p_1^*-1)+\frac{2}{p_1^*}\kappa\big]}{2(p_1^*-1)+2\lambda S^{-p}}\biggr)\int_{\mathbb{R}^n}\min\{\epsilon^p|D\varphi|^p,\, \epsilon^2|Dv|^{p-2}|D\varphi|^2\}\,\de x.
			\end{split}
		\end{equation*}
		By choosing $\kappa=\kappa(n,p,k)>0$ small enough so that
		\begin{equation*}
			\frac{(1-\kappa)}{2} - \frac{(p_1^*-1) + \frac{2}{p_1^*}\kappa}{2(p_1^*-1) + 2\lambda S^{-p}} \geq 0,
		\end{equation*}
		and taking $\gamma_0=\gamma_0>0$ small enough so that
		\begin{equation*}
			\frac{c_0}{2} \geq \gamma_0 \frac{p\big[(p_1^*-1) + \frac{2}{p_1^*}\kappa\big]}{2(p_1^*-1) + 2\lambda S^{-p}},
		\end{equation*}
		we eventually arrive at
		\begin{equation}\label{eq:f1}
			C(n,p,k)\delta(u) \geq \frac{c_0}{2}\int_{\mathbb{R}^n} \min\bigl\{ \epsilon^p |D\varphi|^p,\, \epsilon^2 |Dv|^{p-2}|D\varphi|^2 \bigr\}\,\de x.
		\end{equation}
		Observe that, since $p<2$, it follows from H$\mathrm{\ddot{o}}$lder's inequality that
		\begin{equation*}
			\begin{split}
				\biggl(\int_{\{\epsilon|D\varphi|<|Dv|\}}|D\varphi|^p\,\de x\biggr)^\frac{2}{p}
				&\leq\biggl(\int_{\{\epsilon|D\varphi|<|Dv|\}}|Dv|^p\,\de x\biggr)^{\frac{2}{p}-1}\int_{\{\epsilon|D\varphi|<|Dv|\}}|Dv|^{p-2}|D\varphi|^2\,\de x\\
				&\leq C(n,p,k)\int_{\{\epsilon|D\varphi|<|Dv|\}}|Dv|^{p-2}|D\varphi|^2\,\de x.
			\end{split}
		\end{equation*}
		Hence, by $\|D\varphi\|_{L^p(\mathbb{R}^n)}=1$, we get
		\begin{equation}\label{eq:f2}
			\begin{split}
				&\int_{\mathbb{R}^n} \min\bigl\{ \epsilon^p|D\varphi|^p, \epsilon^2|Dv|^{p-2}|D\varphi|^2 \bigr\}\,\de x \\
				&\qquad = \int_{\{\epsilon|D\varphi|\geq|Dv|\}} \epsilon^p|D\varphi|^p\,\de x + \int_{\{\epsilon|D\varphi|<|Dv|\}} \epsilon^2|Dv|^{p-2}|D\varphi|^2\,\de x \\
				&\qquad \geq \int_{\{\epsilon|D\varphi|\geq|Dv|\}} \epsilon^p|D\varphi|^p\,\de x + c \biggl( \int_{\{\epsilon|D\varphi|<|Dv|\}} \epsilon^p|D\varphi|^p\,\de x \biggr)^{\frac{2}{p}}\geq c \biggl( \int_{\mathbb{R}^n} \epsilon^p|D\varphi|^p\,\de x \biggr)^{\frac{2}{p}},
			\end{split}
		\end{equation}
		where $c=c(n,p,k)>0$. Combining \eqref{eq:f1} and \eqref{eq:f2}, we conclude the proof of \eqref{eq:se} with $\gamma=2$.
		
		\medskip
		$\bullet$ \textit{ The case $\frac{2n}{n+1}<p<2$}.
		The proof is quite similar to the previous case $1<p\leq\frac{2n}{n+1}$.
		Let $\kappa>0$ be a small constant to be fixed later.
		By Lemma \ref{le:FZ 2.1}, we also have \eqref{eq:a1} in this case.
		On the other hand, by Lemma \ref{le:FZ 2.4} and the concavity of  $t\mapsto t^{p/p_1^*}$, one has
		\begin{equation*}
			\begin{split}
				1={}&\|u\|_{L^{p_1^*}(\mathbb{R}^n;|y|^{-1})}^p=\biggl(\int_{\mathbb{R}^n}|y|^{-1}|v+\varepsilon\varphi|^{p_1^*}\,\de x\biggr)^{\frac{p}{p_1^*}}\\
				\leq{}&\|v\|_{L^{p_1^*}(\mathbb{R}^n;|y|^{-1})}^{p-p_1^*}\int_{\mathbb{R}^n}|y|^{-1}\biggl[\varepsilon pv^{p_1^*-1}\varphi + \varepsilon^2\biggl(\frac{p(p_1^*-1)}{2}+\frac{p}{p_1^*}\kappa\biggr)v^{p_1^*-2}|\varphi|^2 + C_1(\kappa)\varepsilon^{p_1^*}\frac{p}{p_1^*}|\varphi|^{p_1^*}\biggr]\,\de x\\
				&+\|v\|_{L^{p_1^*}(\mathbb{R}^n;|y|^{-1})}^p.
			\end{split}
		\end{equation*}
		Hence, arguing as in the case $1<p\leq\frac{2n}{n+1}$, from \eqref{eq:hsm}, Proposition \ref{prop:sgc}, and \eqref{eq:f2}, by choosing first $\kappa>0$ and then $\gamma_0>0$ small enough, for $\delta(u)$ sufficiently small, we have
		\begin{equation*}
			C(n,p,k)\delta(u)\geq \biggl(\int_{\mathbb{R}^n}\varepsilon^p|D\varphi|^p\,\de x\biggr)^{\frac{2}{p}}-\varepsilon^{p_1^*}\frac{C_1}{p_1^*}\int_{\mathbb{R}^n}|y|^{-1}|\varphi|^{p_1^*}\,\de x.
		\end{equation*}
		Since $p_1^*>2$ and $1=\|D\varphi\|_{L^p(\mathbb{R}^n)}\geq S\|\varphi\|_{L^{p_1^*}(\mathbb{R}^n;|y|^{-1})}$, the result follows from the Hardy-Sobolev-Maz'ya inequalities, provided $\varepsilon$ is sufficiently small.
		
		\medskip
		$\bullet$ \textit{ The case $2\leq p<n$}.
		Let $\kappa>0$ be a small constant to be fixed later. By Lemma \ref{le:FZ 2.1}, we have
		\begin{equation*}
			\begin{split}
				\|Du\|_{L^p(\mathbb{R}^n)}^p \geq{}&\int_{\mathbb{R}^n}|Dv|^p\,\de x+\varepsilon p\int_{\mathbb{R}^n}|Dv|^{p-2}Dv\cdot D\varphi\,\de x\\
				&+\frac{\varepsilon p(1-\kappa)}{2}\Biggl(\int_{\mathbb{R}^n}\bigg[|Dv|^{p-2}|D\varphi|^2+(p-2)|w|^{p-2}\biggl(\frac{|Du|-|Dv|}{\varepsilon}\biggr)^2\bigg]\,\de x\Biggr)\\
				&+ c_0(\kappa)\varepsilon^p\int_{\mathbb{R}^n}|D\varphi|^p\,\de x,
			\end{split}
		\end{equation*}
		where $w=w(Dv,Dv+Du)$ is defined as in Lemma \ref{le:FZ 2.1}.
		On the other hand, by Lemma \ref{le:FZ 2.4} and the concavity of  $t\mapsto t^{p/p_1^*}$, one has
		\begin{equation*}
			\begin{split}
				1={}&\|u\|_{L^{p_1^*}(\mathbb{R}^n;|y|^{-1})}^p=\biggl(\int_{\mathbb{R}^n}|y|^{-1}|v+\varepsilon\varphi|^{p_1^*}\,\de x\biggr)^{\frac{p}{p_1^*}}\\
				\leq{}&\|v\|_{L^{p_1^*}(\mathbb{R}^n;|y|^{-1})}^{p-p_1^*}\int_{\mathbb{R}^n}|y|^{-1}\biggl[\varepsilon pv^{p_1^*-1}\varphi + \varepsilon^2\biggl(\frac{p(p_1^*-1)}{2}+\frac{p}{p_1^*}\kappa\biggr)v^{p_1^*-2}|\varphi|^2 + C_1(\kappa)\varepsilon^{p_1^*}\frac{p}{p_1^*}|\varphi|^{p_1^*}\biggr]\,\de x.\\
				&+\|v\|_{L^{p_1^*}(\mathbb{R}^n;|y|^{-1})}^p.
			\end{split}
		\end{equation*}
		Hence, arguing as in the case $1<p\leq\frac{2n}{n+1}$, from \eqref{eq:hsm} and Proposition \ref{prop:sgc}, by choosing first $\kappa>0$ small enough, for $\delta(u)$ sufficiently small, we have
		\begin{equation*}
			C(n,p,k)\delta(u)\geq \int_{\mathbb{R}^n}\varepsilon^p|D\varphi|^p\,\de x-\varepsilon^{p_1^*}\frac{C_1}{p_1^*}\int_{\mathbb{R}^n}|y|^{-1}|\varphi|^{p_1^*}\,\de x.
		\end{equation*}
		Since $1=\|D\varphi\|_{L^p(\mathbb{R}^n)}\geq S\|\varphi\|_{L^{p_1^*}(\mathbb{R}^n;|y|^{-1})}$, the result follows from the Hardy-Sobolev-Maz'ya inequalities, provided $\varepsilon$ is sufficiently small.
		This completes our proof of Theorem \ref{thm:main}. \hskip3cm $\qed$
		
		\vspace{1em}
		
		\begin{appendices}
			
			\section{Concentration-Compactness Results}\label{sec-cc}
			
			The associated minimization problem to \eqref{eq:hsm} is
			\begin{equation}\label{eq:mp}
				I_{\lambda}:=\inf\biggl\{ \int_{\mathbb{R}^n}|Du|^p \,\de x : u\in D^{1,p}(\mathbb{R}^n),\int_{\mathbb{R}^n}|y|^{-1}u^{p_1^*} \,\de x=\lambda\biggr\}.
			\end{equation}
			$I_{\lambda}$ is invariant with respect to dilation.
			Indeed, for any $u\in D^{1,p}(\mathbb{R}^n)$ and for any $\sigma>0$, set
			\begin{equation*}
				u_{\sigma}(\cdot):={\sigma}^{-\frac{n-1}{p^*_1}}u\biggl(\frac{\cdot}{\sigma}\biggr)={\sigma}^{-\frac{n-p}{p}}u\biggl(\frac{\cdot}{\sigma}\biggr),
			\end{equation*}
			then
			\begin{equation*}
				\int_{\mathbb{R}^n}|y|^{-1}u_{\sigma}^{p^*_1}(x)\,\de x=\int_{\mathbb{R}^n}|y|^{-1}u^{p^*_1}(x)\,\de x,
			\end{equation*}
			and
			\begin{equation*}
				\int_{\mathbb{R}^n}|Du_{\sigma}(x)|^p \,\de x=\int_{\mathbb{R}^n}|Du(x)|^p \,\de x.
			\end{equation*}
			We also notice that, for all $\lambda>0$,
			\begin{equation*}
				I_{\lambda}=\lambda^{p/p^*_1} I_1.
			\end{equation*}
			Thus, we obtains that for $\alpha\in(0,\lambda)$,
			\begin{equation*}
				I_{\alpha}+I_{\lambda-\alpha}>I_{\lambda}.
			\end{equation*}
			
			In what follows, $Y_0$ is defined at the beginning of Section \ref{sec:cr}.
			\begin{theorem}\label{thm:cc}
				Any minimizing sequence $\{u_i(x)\}_{i\in \mathbb N^+}$ of \eqref{eq:mp} is relatively compact in $D^{1,p}(\mathbb{R}^n)$ up to a translation and a dilation, i.e. there exists $\{\zeta_i\}_{i\in \mathbb N^+}\subseteq Y_0$ and $\{\sigma_i\}_{i\in \mathbb N^+}\subseteq(0,+\infty)$ such that the new minimizing sequence $\{\tilde{u}_i:={\sigma_i}^{-(n-1)/p^*_1}u\bigl((x-\zeta_i)/\sigma_i\bigr)\}_{i\in \mathbb N^+}$ is relatively compact in $D^{1,p}(\mathbb{R}^n)$.
				In particular, there exists a minimum in \eqref{eq:mp}.
			\end{theorem}
			\noindent The proof of Theorem \ref{thm:cc} is similar to \cite[Theorem I.1]{L3} and \cite[Theorem 2.4]{L4} by making use of the following lemma, which can also be proved by similar methods as \cite[Lemma I.1]{L3} and \cite[Lemma 2.4]{L4}.
			\begin{lemma}\label{le:cc1}
				Let $\{u_i\}_{i\in \mathbb N^+}$ be a bounded sequence in $D^{1,p}(\mathbb{R}^n)$ such that $|Du_i|^p$ is tight.
				We may assume that $u_i$ converges a.e. to $u\in D^{1,p}(\mathbb{R}^n)$ and $|Du_i|^p$, $|y|^{-1}|u_i|^{p_1^*}$ converge weakly to some measures $\mu$, $\nu$.
				Then we have:\\
				\noindent({\romannumeral1}) There exist a family of distinct points $\{x_j\}_{j\in J}\subseteq Y_0$ and a family of real numbers $\{\nu_j\}_{j\in J}\subseteq(0,+\infty)$, where $J$ is an at most countable set, such that
				\begin{equation}\label{eq:ccl1}
					\nu=|y|^{-1}|u|^{p_1^*} + \sum_{j\in J}\nu_j\delta_{x_j}.
				\end{equation}
				
				\noindent({\romannumeral2}) In addition, we have
				\begin{equation}\label{eq:ccl2}
					\mu\geq|Du|^p + \sum_{j\in J}\mu_j\delta_{x_j},
				\end{equation}
				where $\mu_j>0$ satisfies: for all $j\in J$,
				\begin{equation}\label{eq:ccl3}
					\nu_j^{p/p_1^*}\leq\mu/I.
				\end{equation}
				Hence, we also have
				\begin{equation*}
					\sum_{j\in J}\nu_j^{p/p_1^*}<+\infty.
				\end{equation*}
				
				\noindent({\romannumeral3}) If $\nu\in D^{1,p}(\mathbb{R}^n)$ and $|D(u_i+v)|^P$ converges weakly to some measure $\tilde{\mu}$, then $\tilde{\mu}-\mu\in L^1(\mathbb{R}^n)$.
				Therefore, one has
				\begin{equation*}
					\tilde{\mu}\geq|D(u+v)|^p+\sum_{j\in J}\mu_j\delta_{x_j}.
				\end{equation*}
				
				\noindent({\romannumeral4}) If $u\equiv0$ and $\int_{\mathbb{R}^n} \,\de\mu\leq I\bigl(\int_{\mathbb{R}^n} \,\de\nu\bigr)^{p/p_1^*}$, then $J$ is a singleton and $\nu=\gamma\delta_{x_0}=\bigl(I({\gamma}^{p/p_1^*})\bigr)^{-1}\mu$, where $\gamma>0$ is a constant and $x_0\in Y_0$.
			\end{lemma}
			
			\begin{remark}
				The fact that $x_j$ is contained in $Y_0$ is clear.
				Indeed, since $u_i$ is bounded in $L^{p^*}(\mathbb{R}^n)$, $|y|^{-1}|u_i|^{p_1^*}$ is bounded in $L_{\rm{loc}}^{\alpha}(\mathbb{R}^n\backslash Y_0)$ for some $\alpha>1$.
				Thus, if there exists an $x_{j_0}\in\mathbb{R}^n\backslash Y_0$ such that $\delta_{x_{j_0}}$ appears in \eqref{eq:ccl1}, it will contradict the weak convergence of $|y|^{-1}|u_i|^{p_1^*}$.
			\end{remark}
			
			\begin{lemma}{\rm\cite[Lemma 1.2]{L3}}\label{le:cc2}
				Let $\mu$, $\nu$ be two bounded nonnegative measures on $\mathbb{R}^n$ satisfying for some constant $C_0\geq 0$,
				\begin{equation*}
					\biggl(\int_{\mathbb{R}^n}|\varphi|^q \,\de\nu\biggr)^{1/q}\geq C_0\	\biggl(\int_{\mathbb{R}^n}|\varphi|^p \,\de\mu\biggr)^{1/p},\quad\quad\forall\,\varphi\in\mathcal{D}(\mathbb{R}^n),
				\end{equation*}
				where $1\leq p<q \leq +\infty$.
				Then there exist an at most countable set $J$, a family $\{x_j\}_{j\in J}$ of distinct point in $\mathbb{R}^n$, $\{\nu_j\}_{j\in J}$ in $(0,+\infty)$ such that
				\begin{equation*}
					\nu = \sum_{j\in J}\nu_j\delta_{x_j}\quad\quad\text{and}\quad\quad\,\de\mu\geq C_0^{-p}\sum_{j\in J}\nu_j^{p/q}\delta_{x_j}.
				\end{equation*}
				In particular,
				\begin{equation*}
					\sum_{j\in J}\nu_j^{p/q}<+\infty.
				\end{equation*}
				If, in addition, $\nu({\mathbb{R}^n})^{1/q}\geq C_0\mu({\mathbb{R}^n})^{1/p}$, then $J$ reduces to a single point and $\nu=\gamma\delta_{x_0}=\gamma^{-p/q}C_0^p\mu$ for some $x_0\in\mathbb{R}^n$ and for some $\gamma\geq 0$.
			\end{lemma}
			
			\begin{proof}[Proof of Theorem \ref{thm:cc}]
				Without loss of  generality, we can assume $\lambda=1$ in \eqref{eq:mp}.
				Denote by $\{u_i(x)\}_{i\in \mathbb N^+}$ the new minimizing sequence $\bigl\{{\sigma}^{-(n-1)/p^*_1}u((x-\zeta_i)/\sigma_i)\bigr\}_{i\in \mathbb N^+}$ and all the subsequence extracted from it.
				If $\{u_i(x)\}_{i\in \mathbb N^+}$ is a minimizing sequence and
				\begin{equation*}
					\rho_i=|y|^{-1}|u_i|^{p_1^*}+ |Du_i|^p,
				\end{equation*}
				we may choose $\sigma_i>0$ such that, for any $t\geq 0$,
				\begin{equation}\label{eq:cf}
					\mathcal{Q}_i(1)=\frac{1}{2}\quad\quad\text{with}\quad\quad\mathcal{Q}_i(t):=\sup_{x'\in\mathbb{R}^n}\int_{B(x',t)}\rho_i\,\de x.
				\end{equation}
				Define $L_i:=\int_{\mathbb{R}^n}\rho_i \,\de x$.
				As $i\rightarrow\infty$, one has
				\begin{equation*}
					L_i=\int_{\mathbb{R}^n}|u_i|^{p^*}+|Du_i|^p\,\de x\rightarrow L\geq 1+I.
				\end{equation*}
				Next, we prove that $\rho_i$ is tight up to a translation, i.e., there exists $\{\zeta_i\}_{i\in \mathbb N^+}\subseteq\mathbb{R}^n$ such that, for any $\varepsilon>0$, there exists $R<+\infty$ such that
				\begin{equation}\label{eq:tp}
					\int_{|x-\zeta_i|\geq R}\rho_i(x)\,\de x\leq \varepsilon.
				\end{equation}
				In view of \eqref{eq:cf}, as $i\rightarrow\infty$, $\mathcal{Q}_i(t)\rightarrow\mathcal{Q}(t)$ for some non-decreasing, non-negative function $\mathcal{Q}$ on $\mathbb{R}_+$, so we have
				\begin{equation*}
					0<1/2=\mathcal{Q}(1)\leq\mathcal{Q}(t)\leq C.
				\end{equation*}
				Applying the methods in \cite{L1,L2}, in order to prove \eqref{eq:tp}, we only need to show that dichotomy cannot occur.
				Assume on the contrary, there exists $\bar{\alpha}\in(0,L)$ such that, for all $\varepsilon>0$, there exists $\zeta_i\in\mathbb{R}^n$ and $R_i>R_0>0$ satisfying $R_i\rightarrow+\infty$ $(i\rightarrow\infty)$ such that
				\begin{equation}\label{eq:dct}
					\biggl|\bar{\alpha}-\int_{B(\zeta_i,R_0)}\rho_i \,\de x\biggr|\leq\varepsilon\quad\quad\text{and}\quad\quad\int_{R_0\leq|x-\zeta_i|\leq R_i}\rho_i \,\de x\leq\varepsilon.
				\end{equation}
				Let $\xi$, $\psi\in C_b^{\infty}(\mathbb{R}^n)$ satisfying: $0\leq\xi\leq 1$, $0\leq\psi\leq 1$ and $\xi=1$ if $|x|\leq1$, $\xi=0$ if $|x|\geq2$, $\psi=1$ if $|x|\geq1$, $\psi=0$ if $|x|\leq 1/2$.
				We define $\xi_i=\xi((x-\zeta_i)/R_1)$, $\psi_i=\psi((x-\zeta_i)/R_i)$, where $R_i\geq R_0$ is determined below.
				It is easy to deduce that, for $i$ large enough such that $R_i\geq4R_0$, then
				\begin{equation*}
					\biggl|\int_{\mathbb{R}^n}|Du_i|^p \,\de x-\int_{\mathbb{R}^n}|D(\xi_iu_i)|^p \,\de x-\int_{\mathbb{R}^n}|D(\psi_iu_i)|^p\,\de x\biggr|\leq C(X_i^p+X_i)+\varepsilon,
				\end{equation*}
				where
				\begin{equation*}
					X_i=\biggl(\int_{\mathbb{R}^n}\bigl(|D\xi_i|^p+|D\psi_i|^p\bigr)u_i^p \,\de x\biggr)^{1/p}.
				\end{equation*}
				By H$\rm{\ddot{o}}$lder's inequality, we obtain
				\begin{equation*}
					X_i^p\leq C\biggl(\int_{\mathbb{R}^n}|D\xi_i|^n+|D\psi_i|^n \,\de x\biggr)^{p/n}\cdot\biggl(\int_{R_0\leq|x-\zeta_i|\leq R_i}|u_i|^{p^*}\biggr)^{p/p^*},
				\end{equation*}
				where $1/n+1/p^*=1/p$.
				In view of \eqref{eq:dct}, one can deduce
				\begin{equation*}
					X_i^p\leq{} C\varepsilon\biggl(\int_{\mathbb{R}^n}|D\xi_i|^n+|D\psi_i|^n \,\de x\biggr)^{p/n}\leq{}C\varepsilon\biggl(\int_{\mathbb{R}^n}|D\xi|^n+|D\psi|^n \,\de x\biggr)^{p/n}\leq{}C\varepsilon.
				\end{equation*}
				We finally obtain
				\begin{equation}\label{eq:dct1}
					\biggl|\int_{\mathbb{R}^n}|Du_i|^p \,\de x-\int_{\mathbb{R}^n}|Du_i^1|^p \,\de x-\int_{\mathbb{R}^n}|Du_i^2|^p\,\de x\biggr|\leq C(\varepsilon^{1/p}+\varepsilon),
				\end{equation}
				where $u_i^1=\xi_iu_i$, $u_i^2=\psi_iu_i$.
				Without loss of generality, we may assume that, as $i\rightarrow\infty$,
				\begin{equation}\label{eq:dct2}
					\int_{\mathbb{R}^n}|y|^{-1}|u_i^1|^{p_1^*} \,\de x\rightarrow\alpha\quad\quad\mathrm{and}\quad\quad\int_{\mathbb{R}^n}|y|^{-1}|u_i^2|^{p_1^*} \,\de x\rightarrow\beta,
				\end{equation}
				where $0\leq\alpha,\:\beta\leq1$ and  $|1-(\alpha+\beta)|\leq\varepsilon$.
				We claim: for $\varepsilon$ small enough, $\| Du_i^1\|_{L^p(\mathbb{R}^n)}$ and $\| Du_i^2\|_{L^p(\mathbb{R}^n)}$ remains bounded away from 0.
				In fact, \eqref{eq:dct} shows that
				\begin{equation*}
					\biggl|\int_{\mathbb{R}^n}|y|^{-1}|u_i^1|^{p_1^*}+|Du_i^1|^p \,\de x -\bar{\alpha}\biggr|\leq\varepsilon.
				\end{equation*}
				Combining \eqref{eq:dct} with \eqref{eq:dct1}, one can obtain
				\begin{equation*}
					\biggl|\int_{\mathbb{R}^n}|y|^{-1}|u_i^2|^{p_1^*}+|Du_i^2|^p \,\de x -(L-\bar{\alpha})\biggr|\leq\varepsilon.
				\end{equation*}
				Denote by $\gamma>0$ some constant such that, for any $\varepsilon$ small enough and for any $i\geq1$, $\gamma\leq\| Du_i^l\|_{L^p(\mathbb{R}^n)}^p$ $(l=1,2)$.
				If for some sequence $\varepsilon_m\rightarrow0\,(m\rightarrow\infty)$, the constant $\alpha_m=\alpha(\varepsilon_m)$ in \eqref{eq:dct2} satisfies $\alpha_m\rightarrow0$ or $\alpha_m\rightarrow1$, we deduce from \eqref{eq:dct1} that
				\begin{equation*}
					\delta(\varepsilon_m)\geq\gamma+I-I,
				\end{equation*}
				where $\delta(t)\rightarrow0$ as $t\rightarrow0_+$.
				This is impossible.
				On the other hand, if $\alpha_m\rightarrow\alpha\in(0,1)$ and $\beta_m\rightarrow1-\alpha$ as $m\rightarrow\infty$, we can also deduce from \eqref{eq:dct1} that
				\begin{equation*}
					I\geq I_{\alpha}+I_{1-\alpha},
				\end{equation*}
				which is also impossible.
				We concluded the proof of \eqref{eq:tp}.
				
				Next, we show the weak limit of the minimizing sequence $u$ is not identically $0$.
				Indeed, in view of \eqref{eq:tp}, we deduce from
				the weak convergence that
				\begin{equation*}
					\int_{\mathbb{R}^n} \,\de\mu=I\quad\quad\text{and}\quad\quad\int_{\mathbb{R}^n} \,\de\nu=1.
				\end{equation*}
				If $u\equiv0$, we may apply Lemma \ref{le:cc1} $({\romannumeral4})$ and deduce that $\nu=\frac{1}{I}\mu=\delta_{x_0}$ for some $x_0\in\mathbb{R}^n$.
				Therefore, one has
				\begin{equation*}
					\frac{1}{2}=\mathcal{Q}_i(1)\geq\int_{B(x_0,1)}|y|^{-1}|u_i|^{p_1^*} \,\de x\rightarrow 1.
				\end{equation*}
				This contradiction shows that $u\not\equiv0$.
				
				Finally, we show that $u_i$ converges strongly to $u$.
				Define $\alpha:=\int_{\mathbb{R}^n}|y|^{-1}|u|^{p_1^*} \,\de x$.
				By the above proof, we know that $\alpha\in(0,1]$ and we have to show $\alpha=1$.
				Suppose that $\alpha\neq 1$, applying Lemma \ref{le:cc1}, one has
				\begin{equation*}
					\alpha=\int_{\mathbb{R}^n}|y|^{-1}|u|^{p_1^*} \,\de x,\quad\quad\sum_{j\in J}\nu_j=1-\alpha,
				\end{equation*}
				and
				\begin{equation*}
					\mu_j\geq I\nu_j^{p/p_1^*},\quad\quad\int_{\mathbb{R}^n}|Du|^p \,\de x \leq I-\sum_{j\in J}\mu_j.
				\end{equation*}
				Hence, one can obtain
				\begin{equation*}
					\int_{\mathbb{R}^n}|Du|^p\,\de x\leq{}I-\sum_{j\in J}\mu_j\leq{}I\Big{(}1-\sum_{j\in J}\nu_j^{p/p_1^*}\Big{)}<{}I\Big{(}1-\Big{(}\sum_{j\in J}\nu_j\Big{)}\Big{)}^{p/p_1^*}=I{\alpha}^{p/p_1^*},
				\end{equation*}
				while $\int_{\mathbb{R}^n}|Du|^p \,\de x\geq I_{\alpha}=I{\alpha}^{p/p_1^*}$.
				This contradiction shows that $\alpha=1$.
				
				Denote by $\pi$ the projection form $\mathbb{R}^n$ to and $Y_0$.
				We claim that $\{\pi(\zeta_i)\}_{i\in \mathbb N^+}$ remains bounded.
				If $|\pi(\zeta_i)|$ (or a subsequence)  $\rightarrow+\infty$ as $i\rightarrow\infty$.
				Let $\eta\in \mathcal{D}(\mathbb{R}^n)$ such that $\eta\equiv1$ on $B(0,1)$, $0\leq\eta\leq1$, $\mathrm{supp}\,\eta\subseteq B(0,2)$ and let $\eta_i=\eta((\cdot-\zeta_i)/R)$.
				Combining \eqref{eq:tp} with Sobolev inequality, one has
				\begin{equation*}
					\int_{\mathbb{R}^n}(u_i-v_i)^{p^*_1}\,\de x\leq\int_{\mathbb{R}^n}|D(u_i-v_i)|^p \,\de x\leq\delta(\varepsilon)\rightarrow0,\quad\quad\text{as }\varepsilon\rightarrow0,
				\end{equation*}
				where $v_i=\eta_i(u_i)$.
				Thus, for $\varepsilon$ small enough,
				\begin{equation*}
					\int_{\mathbb{R}^n}|y|^{-1}|v_i|^{p_1^*} \,\de x\geq\frac{1}{2}.
				\end{equation*}
				On the other hand, for $i$ large enough,
				\begin{equation*}
					\int_{\mathbb{R}^n}|y|^{-1}|v_i|^{p_1^*}\,\de x\leq\int_{B(\zeta_i,2R)}|y|^{-1}|u_i|^{p_1^*} \,\de x\leq C(|\pi(\zeta_i)-2R|)^{-1},
				\end{equation*}
				which is a contradiction.
				Hence $\{\pi(\zeta_i)\}_{i\in \mathbb N^+}$ is bounded and we may as well take $\pi(\zeta_i)=0$, i.e., $\{\zeta_i\}_{i\in \mathbb N^+}\subseteq Y_0$, so we complete the proof.
			\end{proof}
			
			\begin{proof}[Proof of Lemma \ref{le:cc1}]
				We first treat the case $u\equiv0$.
				The idea is to obtain a reversed H$\rm{\ddot{o}}$lder's inequality between $\nu$ and $\mu$, which will give various information we need.
				
				Let $\varphi\in\mathcal{D}(\mathbb{R}^n)$, by Hardy-Sobolev-Maz'ya inequalities, we have
				\begin{equation}\label{eq:hsmi}
					\biggl(\int_{\mathbb{R}^n}|y|^{-1}|u_i|^{p_1^*}|\varphi|^{p_1^*} \,\de x\biggr)^{1/p_1^*}I^{1/p}\leq\biggl(\int_{\mathbb{R}^n}|D(\varphi u_i)|^p \,\de x\biggr)^{1/p}.
				\end{equation}
				The left-hand side of \eqref{eq:hsmi}  $\rightarrow\bigl(\int_{\mathbb{R}^n}|\varphi|^{p_1^*} \,\de\nu\bigr)^{1/p_1^*}$ as $n\rightarrow\infty$, and the right-hand side of \eqref{eq:hsmi} can be estimated as follows:
				\begin{equation*}
					\biggl|\biggl(\int_{\mathbb{R}^n}|D(\varphi u_i)|^p \,\de x\biggr)^{1/p} - \biggl(\int_{\mathbb{R}^n}|\varphi|^p|Du_i|^p\biggr)^{1/p}\biggr|\leq C\biggl(\int_{\mathbb{R}^n}|D\varphi|^p|u_i|^p \,\de x\biggr)^{1/p}.
				\end{equation*}
				Combining the fact that $|Du_i|^p$ is tight, $u\equiv0$ and Rellich-Kondrachov theorem, one can deduce that the right-hand side of the above inequality $\rightarrow0$ as $i\rightarrow\infty$.
				Therefore, passing to the limit in \eqref{eq:hsmi}, it holds
				\begin{equation*}
					\biggl(\int_{\mathbb{R}^n}|\varphi|^{p_1^*} \,\de\nu\biggr)^{1/p_1^*}\leq I^{-1/p}\biggl(\int_{\mathbb{R}^n}|\varphi|^p \,\de\mu\biggr)^{1/p_1^*},\quad\quad\forall\,\varphi\in\mathcal{D}(\mathbb{R}^n).
				\end{equation*}
				As a consequence, Lemma \ref{le:cc1} is proved in the case $u\equiv0$ by applying Lemma \ref{le:cc2}.
				
				We now consider the general case $u\not\equiv0$.
				It is easy to find that \eqref{eq:hsmi} still holds.
				Letting $v_i:=u_i-u$, by Brezis-Lieb Lemma \cite{BL_a}, one has, for all $\varphi\in\mathcal{D}(\mathbb{R}^n)$,
				\begin{equation*}
					\int_{\mathbb{R}^n}|y|^{-1}|u_i|^{p_1^*}|\varphi|^{p_1^*}\,\de x - \int_{\mathbb{R}^n}|y|^{-1}|v_i|^{p_1^*}|\varphi|^{p_1^*}\,\de x \rightarrow\int_{\mathbb{R}^n}|y|^{-1}|u|^{p_1^*}|\varphi|^{p_1^*}\,\de x.
				\end{equation*}
				Since $v_i$ is bounded in $\mathcal{D}^{1,p}(\mathbb{R}^n)$ and $|Dv_i|^p$ is tight, applying Lemma \ref{le:cc2} to $v_i$, we obtain \eqref{eq:ccl1}.
				Next, passing to the limit in \eqref{eq:hsmi} and using Rellich-Kondrachov theorem as before, we obtain that, for any $\varphi\in\mathcal{D}(\mathbb{R}^n)$,
				\begin{equation*}
					\biggl(\int_{\mathbb{R}^n}|\varphi|^{p_1^*} \,\de\nu\biggr)^{1/p_1^*}I^{1/p}\leq\biggl(\int_{\mathbb{R}^n}|\varphi|^p \,\de\mu\biggr)^{1/p} + C\biggl(\int_{\mathbb{R}^n}|D\varphi|^p|u|^p \,\de x\biggr)^{1/p}.
				\end{equation*}
				If $\varphi$ satisfies $0\leq\varphi\leq1$, $\varphi(0)=1$, $\mathrm{supp}\,\varphi=B(0,1)$ and $\varphi\in\mathcal{D}(\mathbb{R}^n)$, we can apply the above inequality to $\varphi((x-x_j)/\varepsilon)$ for $\varepsilon>0$ and $j\in J$ and obtain that
				\begin{equation*}
					\nu_j^{1/p_1^*}I^{1/p}\leq\mu(B(x_j,\varepsilon))^{1/p}+C\biggl(\int_{B(x_j,\varepsilon)}{\varepsilon}^{-p}\biggl|D\varphi\biggl(\frac{x-x_j}{\varepsilon}\biggr)\biggr|^p|u|^p \,\de x\biggr)^{1/p}.
				\end{equation*}
				Noticing $1/n+1/p^*=1/p$, by H$\rm{\ddot{o}}$lder's inequality, we get
				\begin{equation*}
					\begin{split}
						\int_{B(x_j,\varepsilon)}&{\varepsilon}^{-p}\biggl|D\varphi\biggl(\frac{x-x_j}{\varepsilon}\biggr)\biggr|^p|u|^p \,\de x\\
						\leq{}&\biggl(\int_{B(x_j,\varepsilon)}|u|^{p^*} \,\de x\biggr)^{p/p^*}{\varepsilon^{-p}}\biggl(\int_{\mathbb{R}^n}\biggl|D\varphi\biggl(\frac{x}{\varepsilon}\biggr)\biggr|^n \,\de x\biggr)^{(p^*-p)/p^*}\\
						\leq{}&C\biggl(\int_{B(x_j,\varepsilon)}|u|^{p^*} \,\de x\biggr)^{p/p^*}.
					\end{split}
				\end{equation*}
				Hence, we have
				\begin{equation*}
					\nu_j^{1/p_1^*}I^{1/p}\leq\mu(B(x_j,\varepsilon))^{1/p}+C\biggl(\int_{B(x_j,\varepsilon)}|u|^{p^*} \,\de x\biggr)^{1/p^*}.
				\end{equation*}
				Letting $\varepsilon\rightarrow0$, the above inequality implies that $\mu(\{x_j\})>0$ and for all $j\in J$,
				\begin{equation*}
					\mu\geq\nu_j^{p/p_1^*}I\delta_{x_j}.
				\end{equation*}
				Thus, one can deduce
				\begin{equation*}
					\mu\geq\sum_{j\in J}\nu_j^{p/p_1^*}I\delta_{x_j}=:\mu_1.
				\end{equation*}
				By the weak convergence of $Du_i$ and Fatou's Lemma, we can obtain $\mu\geq|Du|^p$.
				Since $|Du|^p$ and $\mu_1$ are orthogonal, \eqref{eq:ccl2} and \eqref{eq:ccl3} are proved.
				
				Finally, in order to prove Lemma \ref{le:cc1} $({\romannumeral3})$, one observes that, for all $\varphi\in C_b(\mathbb{R}^n)$ (the space of functions in $\mathbb{R}^n$ which is continuous and bounded) satisfying $\varphi\geq 0$, it holds
				\begin{equation*}
					\biggl|\biggl(\int_{\mathbb{R}^n}\varphi|D(u_i+v)|^p \,\de x\biggr)^{1/p}-\biggl(\int_{\mathbb{R}^n}\varphi|Du_i|^p \,\de x\biggr)^{1/p}\biggr|\leq\biggl(\int_{\mathbb{R}^n}\varphi|Dv|^p \,\de x\biggr)^{1/p}.
				\end{equation*}
				Letting $i\rightarrow\infty$, one has
				\begin{equation*}
					\biggl|\biggl(\int_{\mathbb{R}^n}\varphi \,\de\tilde{\mu}\biggr)^{1/p}-\biggl(\int_{\mathbb{R}^n}\varphi \,\de\mu\biggr)^{1/p}\biggr|\leq\biggl(\int_{\mathbb{R}^n}\varphi h \,\de x\biggr)^{1/p},
				\end{equation*}
				where $h\in L_+^1(\mathbb{R}^n)$.
				This shows that the singular parts of $\tilde{\mu}$ and $\mu$ are the same.
				Thus the proof is finished.
			\end{proof}

			\section{Some useful inequalities}\label{sec-ue}
			\begin{lemma}{\rm\cite[(2.3) and (2.1)]{PV}}\label{lem:sobs}\\
				\noindent(\romannumeral1) The Hardy-Sobolev inequality: for any $q\geq1$, $s>q-n$ and $R\geq0$, there exists a constant $C:=C(n,q,s)$ such that, for any $\varphi\in C_c^1(\mathbb{R}^n)$,
				\begin{equation}\label{eq:hsi}
					\int_{\mathbb{R}^n\backslash B(0,R)}|x|^{s-q}|\varphi|^q\,\de x\leq C(n,q,s)\int_{\mathbb{R}^n}|x|^s|D\varphi|^q\,\de x.
				\end{equation}
				
				\noindent(\romannumeral2) 	The Caffarelli-Kohn-Nirenberg inequality: for any $r,q\geq1$, $\frac{1}{r}+\frac{\beta}{n}=\frac{1}{q}+\frac{\alpha-1}{n}>0$, then there exists a constant $C:=C(n,q,\alpha,\beta)$ such that, for any $\phi\in C_c^1(\mathbb{R}^n)$,
				\begin{equation}\label{eq:ckni}
					\bigl\||x|^{\beta}\phi\bigr\|_{L^r(\mathbb{R}^n)}\leq C\bigl\||x|^{\alpha} |D\phi| \bigr\|_{L^q(\mathbb{R}^n)}.
				\end{equation}
				
			\end{lemma}

			\begin{lemma}{\rm\cite[Lemma 2.1]{FZ}}\label{le:FZ 2.1}
				Let $x_1,x_2\in\mathbb{R}^n$.
				For any $\kappa>0$, there exists a positive constant $c_0=c_0(p,\kappa)$ such that the following holds:\\
				\noindent(\romannumeral1) for $1<p<2$,
				\begin{equation*}
					\begin{split}
						|x_1+x_2|^p\geq{}|x|^p+p|x_1|^{p-2}x_1\cdot x_2+\frac{1-\kappa}{2}\Big{(}p|x_1|^{p-2}|x_2|^2+p(p-2)&|w|^{p-2}(|x_1|-|x_1+x_2|)^2\Big{)}\\
						&+ c_0\min\bigl\{|x_2|^p,|x_1|^{p-2}|x_2^2|\bigr\},
					\end{split}
				\end{equation*}
				where
				\begin{equation*}
					w=w(x_1,x_1+x_2):=\left\{
					\begin{aligned}
						&\biggl(\frac{|x_1+x_2|}{(2-p)|x_1+x_2|+(p-1)|x_1|}\biggr)^{\frac{1}{p-2}}x_1\quad&\text{if }|x_1|<|x_1+x_2|,\\
						&\quad\quad\quad\quad\quad\quad\quad\quad x_1&\text{if }|x_1|\geq|x_1+x_2|;
					\end{aligned}
					\right.
				\end{equation*}
				
				\noindent(\romannumeral2) for $2\leq p<n$,
				\begin{equation*}
					|x_1+x_2|^p\geq{}|x_1|^p+p|x_1|^{p-2}x_1\cdot x_2+\frac{1-\kappa}{2}\Big{(}p|x_1|^{p-2}|x_2|^2+p(p-2)|w|^{p-2}(|x_1|-|x_1+x_2|)^2\Big{)}
					+ c_0|x_2|^p,
				\end{equation*}
				where
				\begin{equation*}
					w=w(x_1,x_1+x_2):=\left\{
					\begin{aligned}
						&\quad\quad\quad\quad\quad\quad x_1&\text{if }|x_1|<|x_1+x_2|,\\
						&\biggl(\frac{|x_1+x_2|}{|x_1|}\biggr)^{\frac{1}{p-2}}D(x_1+x_2)\quad&\text{if }|x_1|\geq|x_1+x_2|.
					\end{aligned}
					\right.
				\end{equation*}
				
			\end{lemma}
			Note that, for the case $p=2$, since the coefficient $p(p-2)$ vanishes, the exact definition of $w$ is irrelevant.
			
			\begin{lemma}\label{le:FZ 2.4}
				Let $a,b\in\mathbb{R}$.
				For any $\kappa>0$, there exists a constant $C_1=C_1(p_1^*,\kappa)>0$ such that the following inequalities hold:\\
				\noindent(\romannumeral1) for $1<p\leq\frac{2n}{n+1}$,
				\begin{equation*}
					|a+b|^{p_1^*}\leq|a|^{p_1^*}+p_1^*|a|^{p_1^*-2}ab+\biggl(\frac{p_1^*(p_1^*-1)}{2}+\kappa\biggr)\frac{(|a|+C_1|b|)^{p_1^*}|b|^2}{|a|^2+|b|^2}|b|^2;
				\end{equation*}
				
				\noindent(\romannumeral2) for $\frac{2n}{n+1}<p<n$,
				\begin{equation*}
					|a+b|^{p_1^*}\leq|a|^{p_1^*}+p_1^*|a|^{p_1^*-2}ab+\biggl(\frac{p_1^*(p_1^*-1)}{2}+\kappa\biggr)|a|^{p_1^*-2}|b|^2+C_1|b|^{p_1^*}.
				\end{equation*}	
			\end{lemma}
			\begin{proof}
				Notice that $1<p\leq\frac{2n}{n+1}$ implies $p_1^*\leq2$ and $\frac{2n}{n+1}<p<n$ implies $p_1^*>2$, then the inequalities in Lemma \ref{le:FZ 2.4} follow from \cite[Lemma 3.2]{FN} and \cite[Lemma 2.4]{FZ}.
			\end{proof}
			
			\begin{lemma}{\rm\cite[Lemma B.1]{DT1}\label{le:hi}}
				Let $p>1$ and $\xi\geq1$.
				For any compactly supported function $w\in D^{1,p}(\mathbb{R}^k)$, one has
				\begin{equation*}
					C(k,p,\xi)\int_{\mathbb{R}^{k}}|w|^p|y|^{-1}\bigl[(1+|y|)^{p-1}\bigr]^{\xi-1}\,\de y\leq\int_{\mathbb{R}^k}|Dw|^p\bigl[(1+|y|)^{p-1}\bigr]^{\xi}\,\de y,
				\end{equation*}
				where $C(k,p,\xi)$ is a positive constant.
			\end{lemma}
			\begin{remark}
				Although \cite[Lemma B.1]{DT1} requires $p$ is strictly less than the spatial dimension, we find  that actually it is valid for any $p>1$ (c.f. \cite[Theorem 4.1]{Si}).
			\end{remark}
			
			\begin{lemma}{\rm\cite[Lemma B.2]{DT1}}\label{le:ni}
				Let $0<s<p$, $1<p\leq\frac{2n}{n+2-s}$ and $p_s^*:=\frac{p(n-s)}{n-p}$.
				Given $\varepsilon>0$, there exists a constant $\zeta(\varepsilon_0)>0$ small enough such that the following inequality holds for any nonnegative numbers $\varepsilon, r, a, b$ satisfying $\varepsilon\in(0,1)$ and $\varepsilon a\leq\zeta\Big(1+r^{\frac{p-s}{p-1}}\Big)^{-\frac{n-p}{p-s}}$$\mathrm{:}$
				\begin{align}
					\nonumber\Big(1&+r^{\frac{p-s}{p-1}}\Big)^{-\frac{n-p}{p-s}(p_s^*-2)+p-1}\biggl[a^2\zeta^pr^{\frac{p(1-s)}{p-1}}\Big(1+r^{\frac{p-s}{p-1}}\Big)^{-p} + a^2\varepsilon^pb^p\Big(1+r^{\frac{p-s}{p-1}}\Big)^{\frac{(n-p)p}{p-s}}+a^{2-p}b^p\biggr]\\
					\leq{}&\varepsilon_0r^{-s}\Big(1+r^{\frac{p-s}{p-1}}\Big)^{-\frac{n-p}{p-s}(p_s^*-2)}a^2+C(\varepsilon_0,n,p,s)(1+r)^{-\frac{p-s}{p-1}}\bigg[\Big(1+r^{\frac{p-s}{p-1}}\Big)^{-\frac{n-p}{p-s}}r^{\frac{1-s}{p-1}}+\varepsilon b\bigg]^{p-2}b^2\label{eq:ni1}\\
					\leq{}&\varepsilon_0r^{-s}\Big(1+r^{\frac{p-s}{p-1}}\Big)^{-\frac{n-p}{p-s}(p_s^*-2)}a^2 + C(\varepsilon_0,n,p,s)\bigg[\Big(1+r^{\frac{p-s}{p-1}}\Big)^{-\frac{n-p}{p-s}}r^{\frac{1-s}{p-1}}+\varepsilon b\bigg]^{p-2}b^2.\label{eq:ni2}
				\end{align}
			\end{lemma}
			
			\begin{lemma}\label{le:mu}
				Let $\theta:=\frac{(2-p)p_1^*}{2pQ}$ where $Q$ is defined in \eqref{eq:Q}, then we have $1-\frac{p_1^*}{2}<\theta<1$.
			\end{lemma}
			\begin{proof}
				First, we notice that
				\begin{equation*}
					\begin{split}
						\theta<1 \Leftrightarrow\:{}&Q>\frac{(2-p)p_1^*}{2p}\\
						\Leftrightarrow\:{}&\frac{pn-1}{p(n-1)-(p-1)\mu}>\frac{(2-p)(n-1)}{2(n-p)}\\
						\Leftrightarrow\:{}&\mu>\frac{p(2-p)(n-1)^2-2(pn-1)(n-p)}{(2-p)(p-1)(n-1)}=:\mu_1,
					\end{split}
				\end{equation*}
				and
				\begin{equation*}
					\begin{split}
						1-\frac{p_1^*}{2}<\theta\:\Leftrightarrow\:{}&\frac{p(n-1)}{n-p}-2+\frac{(2-p)p_1^*}{pQ}>0\\
						\Leftrightarrow\:{}&\frac{p(n-1)-2(n-p)}{n-p} + \frac{(2-p)(n-1)}{n-p}\cdot\frac{p(n-1)-\mu(p-1)}{pn-1}>0\\
						\Leftrightarrow\:{}&\mu<\frac{p(2-p)(n-1)^2-(pn-1)[2(n-p)-p(n-1)]}{(2-p)(p-1)(n-1)}=:\mu_2.
					\end{split}
				\end{equation*}
				Since $-p(n-1)<0$, one has $\mu_1<\mu_2$.
				Moreover, since
				\[\frac{n(p-1)}{n-1}<\mu<\frac{p(n-1)}{p-1},\]
				it suffices to check
				\begin{equation*}
					\begin{split}
						\mu_1<\frac{p(n-1)}{p-1}\:
						\Leftrightarrow\:{}&\frac{2(pn-1)(n-p)}{p(2-p)(n-1)^2}>0,
					\end{split}
				\end{equation*}
				and
				\begin{equation*}
					\begin{split}
						\mu_2>\frac{(p-1)n}{n-1}\:
						\Leftrightarrow\:{}&\frac{2n(p-1)+p(n-1)}{n(2-p)(p-1)}>1,
					\end{split}
				\end{equation*}
				which are obvious since $1<p\leq\frac{2n}{n+1}<2$. We complete the proof.
			\end{proof}
		\end{appendices}

		\section*{Acknowledgements}
		Wei Dai is supported by the National Nature Science Fund of China (No. 12222102 and No. 12571113), the National Science and Technology Major Project (2022ZD0116401) and the Fundamental Research Funds for the Central Universities. Jingze Fu is supported by the National Nature Science Fund of China (No. 12571113) and the Fundamental Research Funds for the Central Universities. An Zhang is partially supported by National Key R\&D Program of China (No. 2024YFA1015300), Beijing Natural Science Foundation (No. 1242009),  National Natural Science Foundation of China (No.11801536), the China Scholarship Council (No. 202506020208) and the Fundamental Research Funds for the Central Universities.
		
The authors would like to thank Lu Chen for valuable discussions.

		\section*{Statements and Declarations}
		\textbf{Data availability statement.} Data sharing not applicable to this article as no datasets were generated or analyzed during the current study.
		
		\textbf{Conflict of interests.} The authors have no relevant financial or non-financial interests to disclose.
		
		\section*{References}
		\renewcommand{\refname}{}

	\end{small}	
	
\end{document}